\documentclass[12pt,reqno]{amsart}
\usepackage{amsfonts, amssymb, amsmath, bm}
\usepackage{amscd}
\usepackage{hyperref}
\usepackage{enumitem}

\newtheorem{theorem}{Theorem}

\newtheorem{corollary}[theorem]{Corollary}

\newtheorem{definition}[theorem]{Definition}
\newtheorem{example}[theorem]{Example}

\newtheorem{lemma}[theorem]{Lemma}

\newtheorem{proposition}[theorem]{Proposition}
\newtheorem{remark}[theorem]{Remark}

\begin{document}
\title{The eta invariant on two-step nilmanifolds}
\date{\today}
\author{Ruth Gornet}
\author{Ken Richardson}
\address{Department of Mathematics\\
University of Texas at Arlington\\
Arlington, Texas 76019-0408}
\email{rgornet@uta.edu}
\address{Department of Mathematics\\
Texas Christian University\\
Box 298900 \\
Fort Worth, Texas 76129}
\email{k.richardson@tcu.edu}
\date{revised April, 2024}

\begin{abstract}
The eta invariant appears regularly in index theorems but is known to be
directly computable from the spectrum only in certain examples of locally
symmetric spaces of compact type. In this work, we derive some general
formulas useful for calculating the eta invariant on closed manifolds.
Specifically, we study the eta invariant on nilmanifolds by decomposing the
spin Dirac operator using Kirillov theory. In particular, for general
Heisenberg three-manifolds, the spectrum of the Dirac operator and the eta
invariant are computed in terms of the metric, lattice, and spin structure
data. There are continuous families of geometrically, spectrally different
Heisenberg three-manifolds whose Dirac operators have constant eta
invariant. In the appendix, some needed results of L. Richardson and C. C.
Moore are extended from spaces of functions to spaces of spinors.
\end{abstract}

\maketitle

\setcounter{tocdepth}{1}
\tableofcontents

\section{Introduction}

The eta invariant was introduced in the famous paper of M. F. Atiyah, V. K.
Patodi, and I. M. Singer (see \cite{APS1}), in order to produce an index
theorem for manifolds with boundary. The eta invariant of a linear
self-adjoint operator is roughly the difference between the number of
positive eigenvalues and the number of negative eigenvalues, which of course
is undefined when these numbers are both infinite. However, this quantity
may be regularized to make it well-defined for classical pseudodifferential
operators, using methods similar to the zeta-function regularization of the
determinant of the Laplacian and methods used by physicists to regularize
divergent integrals. The eta function is analogous to Dirichlet $L$%
-functions in the same way that the zeta function of elliptic operators is
analogous to the Riemann zeta function.

Let $D:C^{\infty }\left( E\right) \rightarrow C^{\infty }\left( E\right) $
be an essentially self-adjoint elliptic classical pseudodifferential
operator of order $d$ on sections of a vector bundle $E\rightarrow M$, where 
$M$ is a closed (compact, without boundary) Riemannian manifold of dimension 
$n$. Let $\left\{ \lambda \right\} $ be the collection of eigenvalues with
multiplicity. The eta function is defined as 
\begin{equation*}
\eta \left( s\right) =\sum_{\lambda \neq 0}\mathrm{sgn}\left( \lambda
\right) \left\vert \lambda \right\vert ^{-s}.
\end{equation*}%
This reduces to the zeta function if $D$ has only nonnegative eigenvalues.
The eta function is holomorphic in $s$ for large $\mathrm{Re}\left( s\right) 
$ and can be analytically continued to a meromorphic function using heat
kernel techniques. It is true but not obvious that $\eta \left( s\right) $
is regular at $s=0$, and $\eta \left( 0\right) $ is always real; the eta
invariant is defined as $\eta \left( 0\right) $. See \cite{APS1}, \cite{APS2}%
, \cite{Gilk} for general information about the eta invariant.

The eta function and its generalizations have been studied and utilized in
index theorems for noncompact manifolds and for families of operators and in
gluing formulas. The sign of the eta invariant of the boundary signature
operator on a $4$-manifold with boundary has important geometric content; in
the case of a ball, it determines whether the conformal class of the
boundary metric contains a metric induced from a self-dual Einstein metric
on the interior (see \cite{Hi}). In physics, the eta invariant of the spin
Dirac operators has practical importance, for example in the regularization
of Feynman path integrals (see \cite{W1}). Recently, in the work of J. Br%
\"{u}ning, F. W. Kamber, and K. Richardson, the eta invariant is utilized in
a new equivariant index formula for $G$-manifolds and an index formula for
Riemannian foliations (see \cite{BKR1}, \cite{BKR2}, \cite{BKR3}).

It is very difficult to calculate the eta invariant for a given operator
such as a Dirac operator on a Riemannian manifold; much work has been done
to calculate this invariant for space forms, lens spaces and flat tori (see,
for example, \cite{Do}, \cite{Gi}, \cite{BoGi}). More recently, S. Goette
has calculated formulas for the eta invariant and equivariant eta invariants
on homogeneous spaces of the form $G\diagup H$ with $G$ compact (see \cite%
{Go}). In \cite{ADS}, M. Atiyah, H. Donnelly and I. Singer computed the eta
invariant of the boundary signature operator of a framed solvmanifold in
terms of the signature defect of a manifold whose boundary is that
solvmanifold. In \cite{DenSing}, C. Deninger and W. Singhof computed eta
invariants of modified versions of Dirac operators on Heisenberg manifolds
and were able to compute the eta invariants up to local correction terms. In 
\cite{LMP}, P. Loya, S. Moroianu and J. Park studied the spectrum of the
Dirac operator on a certain three-dimensional circle bundle over a
noncompact Riemann surface with cusps, that is, a noncompact manifold that
is a cofinite quotient of $PSL\left( 2,\mathbb{R}\right) $. They also study
the adiabatic limit of the eta invariant as the fibers are collapsed. The
first explicit computations of eigenvalues of Dirac operators on homogeneous
spaces corresponding to noncompact Lie groups has been done by B. Ammann and
C. B\"{a}r (see \cite{Am-Ba}, \cite{Ba}), where the eigenvalues of the spin
Dirac operator on certain (rectangular) Heisenberg manifolds were computed
explicitly. In \cite{MiaPod}, R. Miatello and R. Podest\'{a} compute the eta
invariant on compact flat spin manifolds with cyclic holonomy of odd prime
order (see also \cite{MiaPod2} for related work). While different techniques
are employed, the Miatello-Podest\'{a} result has a similar flavor to our
main result, in that the final statement relies on metric data, spin
structure data, lattice data and prominently exploits group actions.

A Riemannian nilmanifold is a closed manifold of the form $\left( \Gamma
\diagdown G,g\right) $ where $G$ is a simply connected nilpotent Lie group, $%
\Gamma $ is a cocompact (i.e., $\Gamma \diagdown G$ is compact) discrete
subgroup of $G$, and $g$ is a left-invariant metric on $G$, which descends
to a Riemannian metric on $\Gamma \diagdown G$ that is also denoted by $g$.
A Heisenberg manifold is a two-step Riemannian nilmanifold whose covering
Lie group $G$ is one of the $(2n+1)$-dimensional Heisenberg Lie groups (see,
for example, \cite{GW1}). The study of nilmanifolds and nilpotent Lie groups
has long been relevant to inverse spectral problems (see \cite{GGt} for a
survey). Nilmanifolds play an important role in the study of Dirac
eigenvalues, as was shown in a paper of Ammann and C. Sprouse (see \cite%
{AmSp}). They show that if a Riemannian spin manifold with bounded sectional
curvature and finite diameter has scalar curvature bounded from below by a
sufficiently small negative number and if the smallest Dirac eigenvalue $%
\lambda $ is sufficiently close to zero, then the manifold is diffeomorphic
to a nilmanifold.

In this paper, we prove several results concerning the computation of the
eta invariant on closed manifolds. In Section \ref{etaZetaPerturbedSection},
we discuss the interesting relationships between the zeta and eta functions
of operators, which can be derived from \cite[Proposition 2.10]{APS3}. The
main point is Proposition \ref{etaZetaProposition}, the formula $\frac{%
\partial }{\partial c}\eta _{c}\left( s\right) =-s\zeta _{\left( D+c\right)
^{2}}\left( \frac{s+1}{2}\right) $, where $\eta _{c}$ is the eta function
corresponding to the operator $D+c=D+c\mathbf{1}$, where $c$ is a real
number, and where $\zeta _{\left( D+c\right) ^{2}}$ is the zeta function
corresponding to the operator $\left( D+c\right) ^{2}$. From this we see
that changes in the eta invariant of an elliptic first order operator on a
closed, odd-dimensional manifold is related to a particular residue of a
pole of the zeta function corresponding to the second order operator $\left(
D+c\right) ^{2}$. This residue is, up to a constant, a coefficient in the
asymptotic expansion of the trace of the heat operator $\exp \left( -t\left(
D+c\right) ^{2}\right) $. In Section \ref{HeatKernelAsymptoticsSection},
this coefficient is computed explicitly as a function of $c$.

Using these general results about $\frac{\partial }{\partial c}\eta
_{c}\left( 0\right) $, if $\eta _{c}\left( 0\right) $ is known at a single
value of $c$, the heat kernel asymptotic formula and knowledge of small
eigenvalues determine $\eta _{0}\left( 0\right) $, the eta invariant of $D$.
In Theorem \ref{etaArbitraryDimensionTheorem}, we prove a general formula
for the eta invariant of a Dirac-type operator on a closed manifold in the
case that the spectrum of the operator is symmetric about a
certain real number $\overline{\lambda }$, or when
$\eta _{-\overline{\lambda }}\left( 0\right) $ is known. We deduce from this formula a more specific formula
for Dirac-type operators on three-manifolds in Section \ref%
{zetaEtaThreeMflds}, which calculates the eta invariant in terms of the
volume, the total scalar curvature, the total trace of the twisting
curvature, and small eigenvalues of the Dirac-type operator (notation
defined in that section):%
\begin{eqnarray*}
\eta \left( 0\right) -\eta _{-\overline{\lambda }}\left( 0\right) &=&-\frac{%
\widehat{n}\overline{\lambda }^{3}}{6\pi ^{2}}\mathrm{vol}\left( M\right) +%
\frac{\overline{\lambda }}{4\pi ^{2}}\left( \frac{\widehat{n}}{12}\int_{M}%
\mathrm{Scal}+\int_{M}Tr\left( F^{W}\right) \right) \\
&&+\mathrm{sgn}\left( \overline{\lambda }\right) \left( 2\#\left( \sigma
\left( D\right) \cap \left( 0,\overline{\lambda }\right) \right) +\#\left(
\sigma \left( D\right) \cap \left\{ 0,\overline{\lambda }\right\} \right)
\right) .
\end{eqnarray*}

Using Kirillov theory, the spin Dirac operator on two-step nilmanifolds is
decomposed explicitly in terms of irreducible subspaces of the right
quasi-regular representation in Section \ref{DiracOp2StepSection}. To that
end, occurrence and multiplicity conditions for Dirac eigenspinors are
developed in Section \ref{AnaloguePescesThmSpinorsSection} in analogy to
Pesce's known work \cite{Pe} concerning the Laplacian. It is here that we
utilize analogues of the work of C.C. Moore \cite{Moo}\ and L. Richardson 
\cite{Ri}, developed in the appendix, Section \ref%
{MooreRichardsonPapersSection}. Explicit formulas for the Dirac operator are
computed in terms of a special basis of spinors for each invariant subspace.

For general Heisenberg three-manifolds, the spectrum of the spin Dirac
operator and the eta invariant are computed in terms of the metric, the
lattice and spin structure in Section \ref{EtaInvt3DHeisenSection}. The
formula for the eta invariant has the form%
\begin{equation*}
\eta \left( 0\right) =\frac{r^{2}m_{v}}{96\pi ^{2}A^{2}}+\frac{m_{v}}{24}%
\left( 3\varepsilon _{3}+1\right) -N\left( A,r,w_{2},m_{v},m_{w},\varepsilon
\right) ,
\end{equation*}%
where $N\left( A,r,w_{2},m_{v},m_{w},\varepsilon \right) $ is a nonnegative
integer specified in terms of $A;r,w_{2},m_{v},m_{w};\varepsilon $, the
metric, lattice, and spin structure data. In this section, we exhibit
continuous families of geometrically, spectrally different Heisenberg
three-manifolds whose spin Dirac operators have constant eta invariant.
Computations for a general Heisenberg nilmanifold are done in Section \ref%
{DiracEigenvaluesGenHeisenbergSection}; in particular, we show how to
calculate the Dirac spectrum for any example. We explore symmetries of the
Dirac spectrum in higher-dimensional Heisenberg manifolds in Section \ref%
{VanishingEtaSection}. In Section \ref{nonHeisenbergExampleSection}, we
compute the Dirac operator of a particular five-dimensional non-Heisenberg
nilmanifold, and we show that the techniques used in previous sections do
not yield explicit formulas for the eigenvalues in this case.

The authors would like to thank the referees for a very thorough reading of
the original manuscript. We also thank Michael Stone \cite{Stone} for
pointing out a mistake in the calculation of Dirac eigenvalues of Heisenberg
three-manifolds in a previous version.

\section{The eta invariant}

\subsection{Eta and zeta functions of perturbed operators\label%
{etaZetaPerturbedSection}}

\vspace{0in}In this section, we exhibit some general results relating
families of eta and zeta functions that may be well-known to experts. In
particular, Proposition 2.10 in \cite{APS3} relates the derivative of the
eta invariant of a family of operators to a trace that can be identified
with the zeta function in our particular application. Also, in \cite[Lemma
2.1]{BransGi} and in \cite[Lemma 9]{BransGi2}, the researchers use the same
idea to relate the residues at the poles of the eta function to the
asymptotics of a heat kernel. For the sake of exposition and completeness,
we include the proofs of very specific results that have not been previously
stated in this form, which will be needed in later sections.

In this section and throughout the paper, we will often use the notation $%
\left( D+c\right) $ for an operator, where $D$ is an operator and $c$ is a
scalar, and we regard $c$ in this expression as $c$ times the identity. We
also use the notation $\sigma \left( D\right) $ to denote the spectrum of $D$%
, with multiplicities.

\begin{proposition}
\label{etaZetaProposition}Let $D$ be any self-adjoint operator for which $%
\eta \left( s\right) $ is defined and analytic at $s=0$. Suppose in addition
that there exists an interval $I\subset \mathbb{R}$ and a constant $B>0$
such that for all $c\in I$,

\begin{enumerate}
\item $\sum_{\lambda }\mathrm{sgn}\left( \lambda +c\right) \left\vert
\lambda +c\right\vert ^{-s}$ and $\sum_{\lambda }\left( \left( \lambda
+c\right) ^{2}\right) ^{-\frac{s+1}{2}}$ converge absolutely for $\mathrm{Re}%
\left( s\right) >B$, and

\item $-c$ is not an eigenvalue of $D$.
\end{enumerate}

Then the eta function $\eta _{c}\left( s\right) $ corresponding to the
operator $D+c$ satisfies, on its domain,%
\begin{equation*}
\frac{d}{dc}\eta _{c}\left( s\right) =-s\zeta _{\left( D+c\right)
^{2}}\left( \frac{s+1}{2}\right) ,
\end{equation*}%
where $\zeta _{\left( D+c\right) ^{2}}$ is the zeta function corresponding
to the nonnegative operator $\left( D+c\right) ^{2}$, that is 
\begin{equation*}
\zeta _{\left( D+c\right) ^{2}}\left( s\right) =\sum_{\mu >0}\mu ^{-s},
\end{equation*}%
where the sum is over all positive eigenvalues with multiplicity $\left\{
\mu \right\} $ of the operator $\left( D+c\right) ^{2}$. In particular, if $%
D $ is a first-order, elliptic, essentially self-adjoint differential
operator, then $\frac{d}{dc}\eta _{c}\left( 0\right) $ is the residue of the
simple pole of the meromorphic function $\zeta _{\left( D+c\right)
^{2}}\left( \frac{s+1}{2}\right) $ at $s=0$. (If $\zeta _{\left( D+c\right)
^{2}}\left( \frac{s+1}{2}\right) $ is regular at $s=0$, then $\frac{d}{dc}%
\eta _{c}\left( 0\right) =0$.)
\end{proposition}

\textbf{Remark:} It is known that second-order essentially self-adjoint
elliptic differential operators such as $\left( D+c\right) ^{2}$ on a
manifold of dimension $n$ yield zeta functions with at most simple poles,
and they are located at $s=\frac{n}{2}$, $s=\frac{n}{2}-1$, $s=\frac{n}{2}-2$%
, ... for $n$ odd and at $s=\frac{n}{2}$, $s=\frac{n}{2}-1$, ... , $s=1$ for 
$n$ even. See \cite{Gilk} for specifics. Further, the residues at these
poles are given by explicitly computable integrals of locally-defined
functions.

\begin{proof}
We know that for each eigenvalue $\lambda $ of $D$, $sgn\left( \lambda
+c\right) $ does not vary with $c\in I$. Then for large $\mathrm{Re}\left(
s\right) $, 
\begin{eqnarray*}
\eta _{c}\left( s\right) &=&\sum_{\lambda }\mathrm{sgn}\left( \lambda
+c\right) \left( \left( \lambda +c\right) ^{2}\right) ^{-s/2} \\
\frac{d}{dc}\eta _{c}\left( s\right) &=&\sum_{\lambda }\mathrm{sgn}\left(
\lambda +c\right) \left( -\frac{s}{2}\left( \left( \lambda +c\right)
^{2}\right) ^{-s/2-1}\right) 2\left( \lambda +c\right) \\
&=&-s\sum_{\lambda }\mathrm{sgn}\left( \lambda +c\right) \left\vert \lambda
+c\right\vert ^{-s-2}\left( \lambda +c\right) \\
&=&-s\sum_{\lambda }\left( \left( \lambda +c\right) ^{2}\right) ^{-\frac{s+1%
}{2}}=-s\zeta _{\left( D+c\right) ^{2}}\left( \frac{s+1}{2}\right) \cdot
\end{eqnarray*}%
Since both sides are analytic in $s$ for large $\mathrm{Re}\left( s\right) $%
, the statement must remain true after analytic continuation.
\end{proof}

We are interested in the eta invariant, which is $\eta _{c}\left( 0\right) $%
. By the formula in the proposition above, the relevant information is the
residue of the pole of the zeta function $\zeta _{\left( D+c\right)
^{2}}\left( z\right) $ at $z=\frac{1}{2}$. For odd-dimensional manifolds,
this is a constant times one of the heat invariants. If the manifold is
even-dimensional, there is no pole at $z=\frac{1}{2}$, so that $\frac{d}{dc}%
\eta _{c}\left( 0\right) =0$.

\begin{corollary}
If the manifold is even-dimensional, then $\frac{d}{dc}\eta _{c}\left(
0\right) =0$, so that the eta invariant is constant with respect to $c$ on
intervals where $D+c$ has trivial kernel, and then it changes by integral
jumps in general.
\end{corollary}

We also have the following result about perturbations of zeta functions.

\begin{proposition}
With the assumptions of Proposition \ref{etaZetaProposition}, 
\begin{equation*}
\frac{d}{dc}\zeta _{\left( D+c\right) ^{2}}\left( s\right) =-2s\eta
_{c}\left( 2s+1\right) .
\end{equation*}

\begin{proof}
For large $\mathrm{Re}\left( s\right) $,%
\begin{eqnarray*}
\frac{d}{dc}\zeta _{\left( D+c\right) ^{2}}\left( s\right) &=&\frac{d}{dc}%
\sum_{\lambda }\left( \left( \lambda +c\right) ^{2}\right) ^{-s} \\
&=&\sum_{\lambda }-s\left( \left( \lambda +c\right) ^{2}\right)
^{-s-1}2\left( \lambda +c\right) =-2s\sum_{\lambda }\left\vert \lambda
+c\right\vert ^{-2s-2}\left( \lambda +c\right) \\
&=&-2s\sum_{\lambda }\mathrm{sgn}\left( \lambda +c\right) \left\vert \lambda
+c\right\vert ^{-2s-1}=-2s\eta _{c}\left( 2s+1\right) .
\end{eqnarray*}%
Since both sides are analytic in $s$ for large $\mathrm{Re}\left( s\right) $%
, the statement must remain true after analytic continuation.
\end{proof}
\end{proposition}

\subsection{Heat Kernel Asymptotics\label{HeatKernelAsymptoticsSection}}

Because of Proposition \ref{etaZetaProposition}, we will be interested in
the residues of $\zeta _{\left( D+c\right) ^{2}}\left( s\right) $ at its
poles, which are determined by the heat kernel asymptotics (see Section \ref%
{etaInvtArbMfldsSection}). Specifically, we need the asymptotics as $%
t\rightarrow 0^{+}$ of 
\begin{equation*}
Tr\left( \exp \left( -t\left( D+c\right) ^{2}\right) \right)
=\int_{M}Tr~K_{c}\left( t,x,x\right) ~d\mathrm{vol},
\end{equation*}%
where we assume $D=\sum \left( e_{j}\diamond \right) \nabla
_{e_{j}}:C^{\infty }\left( E\right) \rightarrow C^{\infty }\left( E\right) $
is a Dirac-type operator and $c\in \mathbb{R}$. That is, the Leibniz rule $%
\nabla _{X}\left( v\diamond s\right) =\left( \nabla _{X}^{M}v\right)
\diamond s+v\diamond \nabla _{X}s$ is satisfied for all vector fields $X$
and $v$ and sections $s\in C^{\infty }\left( E\right) $, where $\nabla ^{M}$
is the Levi-Civita connection. We will let $n$ be the dimension of the
manifold $M$, and we will let $\widehat{n}$ be the rank of the vector bundle 
$E$. Here and in what follows, we use the $\diamond $ symbol to denote
Clifford multiplication. The element $K_{c}\left( t,x,x\right) \in \mathrm{%
End}\left( E_{x}\right) $ is 
\begin{equation*}
K_{c}\left( t,x,x\right) =e^{-t\left( D+c\right) ^{2}}\left( x,x\right) ,
\end{equation*}%
which satisfies%
\begin{eqnarray*}
\left( \frac{\partial }{\partial t}+\left( D+c\right) ^{2}\right)
K_{c}\left( t,x,y\right) &=&0 \\
\lim_{t\rightarrow 0^{+}}K_{c}\left( t,x,y\right) &=&\delta _{xy},
\end{eqnarray*}%
where $\delta _{xy}$ is the Dirac delta distribution. To find the
asymptotics as $t\rightarrow 0^{+}$, we need to solve for $u_{k}\left(
x,y\right) \in \mathrm{Hom}\left( E_{y},E_{x}\right) $, where 
\begin{equation}
K_{c}\left( t,x,y\right) \sim \frac{1}{\left( 4\pi t\right) ^{n/2}}%
e^{-r^{2}/4t}\left( u_{0}\left( x,y\right) +tu_{1}\left( x,y\right)
+t^{2}u_{2}\left( x,y\right) +...\right)  \label{K(txy)_asymptotics}
\end{equation}%
where $r=\mathrm{dist}\left( x,y\right) $. Such an asymptotic expansion
exists, since $\left( D+c\right) ^{2}$ is a generalized Laplacian (see \cite%
{Be-G-V}, \cite{Gilk}, \cite{Roe}).

We will assume that we have chosen geodesic normal coordinates $x=\left(
x_{1},...,x_{n}\right) $ centered at $y=0$ and that the frame field $\left(
e_{1},...,e_{n}\right) $ is parallel translated radially from the origin
(i.e. $y$) such that 
\begin{equation*}
e_{j}\left( 0\right) =\partial _{j}~.
\end{equation*}

Then in these coordinates, we may map $E_{x}$ to $E_{y}$ via radial parallel
translation, so that each $u_{k}\left( x,y\right) $ may be regarded as a
matrix-valued function of $x$, with $\mathbb{R}^{\widehat{n}}$ identified
with $E_{y}$. Observe that the Dirac operator may be expressed as%
\begin{equation*}
D=\sum_{j}e_{j}\diamond \nabla _{e_{j}}=\sum_{p,q}g^{pq}\partial
_{p}\diamond \nabla _{\partial _{q}}~,
\end{equation*}%
where in the first case we are summing over an orthonormal frame, and in the
second case we are using the coordinate vector fields, with $\left(
g^{pq}\right) $ the inverse of the metric matrix $\left( g_{ij}\right) $.

We have, using the Einstein summation convention, 
\begin{eqnarray}
-\left( D+c\right) ^{2} &=&-\left( e_{i}\diamond \nabla _{e_{i}}+c\right)
^{2}  \notag \\
&=&-\left( e_{i}\diamond \nabla _{e_{i}}\right) \left( e_{j}\diamond \nabla
_{e_{j}}\right) -2c\left( e_{i}\diamond \nabla _{e_{i}}\right) -c^{2}  \notag
\\
&=&-\left( e_{i}\diamond \right) \left( e_{j}\diamond \right) \nabla
_{e_{i}}\nabla _{e_{j}}+\left[ -\left( e_{i}\diamond \right) \left( \left(
\nabla _{e_{i}}e_{j}\right) \diamond \right) -2c\left( e_{j}\diamond \right) %
\right] \nabla _{e_{j}}-c^{2}  \notag \\
&=&\nabla _{e_{i}}\nabla _{e_{i}}-\sum_{i<j}\left( e_{i}\diamond \right)
\left( e_{j}\diamond \right) \left[ \nabla _{e_{i}},\nabla _{e_{j}}\right] +%
\left[ -\left( e_{i}\diamond \right) \left( \left( \nabla
_{e_{i}}e_{j}\right) \diamond \right) -2c\left( e_{j}\diamond \right) \right]
\nabla _{e_{j}}-c^{2}  \notag \\
&=&\nabla _{e_{i}}\nabla _{e_{i}}-\sum_{i<j}\left( e_{i}\diamond \right)
\left( e_{j}\diamond \right) \left( \nabla _{\left[ e_{i},e_{j}\right]
}\right) +\left[ -\left( e_{i}\diamond \right) \left( \left( \nabla
_{e_{i}}e_{j}\right) \diamond \right) -2c\left( e_{j}\diamond \right) \right]
\nabla _{e_{j}}  \notag \\
&&-\sum_{i<j}\underset{\text{define this to be }K_{ij}}{\underbrace{\left(
e_{i}\diamond \right) \left( e_{j}\diamond \right) \left( \left[ \nabla
_{e_{i}},\nabla _{e_{j}}\right] -\nabla _{\left[ e_{i},e_{j}\right] }\right) 
}}-c^{2}.  \label{Kformula}
\end{eqnarray}%
Further, let $K=\sum_{i<j}K_{ij}\in \mathrm{End}\left( E_{x}\right) $.

\vspace{1pt}Next, let $s$ be a bundle endomorphism, and let $f$ be any
function. Let $h=\frac{1}{\left( 4\pi t\right) ^{n/2}}e^{-r^{2}/4t}$, and
let $g=\det \left( g_{ij}\right) $, where $r$ is the geodesic distance to $%
y=0$. Then from the formulas in \cite[pp. 99-100]{Roe} (extended, as is
common, to endomorphisms),%
\begin{eqnarray*}
\nabla h &=&-\frac{h}{2t}r\partial _{r} \\
\frac{\partial h}{\partial t}+\Delta h &=&\frac{rh\partial _{r}g}{4gt} \\
D\left( fs\right) -fDs &=&\left( \nabla f\right) \diamond s \\
D^{2}\left( fs\right) -fD^{2}s &=&\left( \Delta f\right) s-2\nabla _{\nabla
f}s,
\end{eqnarray*}%
so%
\begin{eqnarray*}
\left( -\left( D+c\right) ^{2}\right) \left( fs\right) &=&-\left(
D^{2}+2cD+c^{2}\right) \left( fs\right) \\
&=&-\left( fD^{2}s+\left( \Delta f\right) s-2\nabla _{\nabla f}s\right)
-2c\left( fDs+\left( \nabla f\right) \diamond s\right) -c^{2}fs \\
&=&-f\left( D+c\right) ^{2}s-\left( \Delta f\right) s+2\nabla _{\nabla
f}s-2c\left( \nabla f\right) \diamond s.
\end{eqnarray*}%
Then

\begin{eqnarray*}
\frac{1}{h}\left( \partial _{t}+\left( D+c\right) ^{2}\right) \left(
hs\right) &=&\left( -\frac{1}{h}\Delta h+\frac{r\partial _{r}g}{4gt}\right)
s+\partial _{t}s+\left( D+c\right) ^{2}s \\
&&+\left( \frac{\Delta h}{h}\right) s-\frac{2}{h}\nabla _{\nabla h}s+\frac{2c%
}{h}\left( \nabla h\right) \diamond s \\
&=&\partial _{t}s+\left( D+c\right) ^{2}s+\frac{r}{4gt}\partial _{r}gs+\frac{%
1}{t}\nabla _{r\partial _{r}}s-\frac{c}{t}\left( r\partial _{r}\right)
\diamond s.
\end{eqnarray*}%
Writing%
\begin{equation*}
s=u_{0}+tu_{1}+t^{2}u_{2}+...,
\end{equation*}%
we solve $\left( \partial _{t}+\left( D+c\right) ^{2}\right) \left(
hs\right) =0$ and get the equations%
\begin{equation}
\nabla _{r\partial _{r}}u_{j}+\left( j+\frac{r\partial _{r}g}{4g}-c\left(
r\partial _{r}\diamond \right) \right) u_{j}=-\left( D+c\right) ^{2}u_{j-1}~,
\label{0thRecursiveFormula}
\end{equation}%
or%
\begin{equation}
\nabla _{\partial _{r}}u_{j}+\left( \frac{j}{r}+\frac{\partial _{r}g}{4g}%
-c\left( \partial _{r}\diamond \right) \right) u_{j}=-\frac{1}{r}\left(
D+c\right) ^{2}u_{j-1}~  \label{firstRecursiveFormula}
\end{equation}%
This is an ordinary differential equation along a geodesic emanating from $y$%
, the center of the geodesic coordinates.

Note that for any smooth function $f$, 
\begin{eqnarray*}
\exp \left( f\left( r\right) \left( \partial _{r}\diamond \right) \right)
&=&\sum_{k\geq 0}\frac{1}{\left( 2k\right) !}f\left( r\right) ^{2k}\left(
\partial _{r}\diamond \right) ^{2k}+\sum_{k\geq 0}\frac{1}{\left(
2k+1\right) !}f\left( r\right) ^{2k+1}\left( \partial _{r}\diamond \right)
^{2k+1} \\
&=&\sum_{k\geq 0}\frac{\left( -1\right) ^{k}}{\left( 2k\right) !}f\left(
r\right) ^{2k}+\left( \sum_{k\geq 0}\frac{\left( -1\right) ^{k}}{\left(
2k+1\right) !}f\left( r\right) ^{2k+1}\right) \left( \partial _{r}\diamond
\right) \\
&=&\cos \left( f\left( r\right) \right) \mathbf{1}+\sin \left( f\left(
r\right) \right) \left( \partial _{r}\diamond \right) .
\end{eqnarray*}%
We also have the operator equation 
\begin{eqnarray*}
&&\nabla _{\partial _{r}}\left[ \cos \left( f\left( r\right) \right) \mathbf{%
1}+\sin \left( f\left( r\right) \right) \left( \partial _{r}\diamond \right) %
\right] \\
&=&-f^{\prime }\left( r\right) \sin \left( f\left( r\right) \right) \mathbf{%
1+}\cos \left( f\left( r\right) \right) \nabla _{\partial _{r}}+f^{\prime
}\left( r\right) \cos \left( f\left( r\right) \right) \left( \partial
_{r}\diamond \right) +\sin \left( f\left( r\right) \right) \left( \partial
_{r}\diamond \right) \nabla _{\partial _{r}} \\
&=&\left[ \cos \left( f\left( r\right) \right) \mathbf{1}+\sin \left(
f\left( r\right) \right) \left( \partial _{r}\diamond \right) \right] \nabla
_{\partial _{r}}+-f^{\prime }\left( r\right) \sin \left( f\left( r\right)
\right) \mathbf{1}+f^{\prime }\left( r\right) \cos \left( f\left( r\right)
\right) \left( \partial _{r}\diamond \right) \\
&=&\left[ \cos \left( f\left( r\right) \right) \mathbf{1}+\sin \left(
f\left( r\right) \right) \left( \partial _{r}\diamond \right) \right] \left(
\nabla _{\partial _{r}}+f^{\prime }\left( r\right) \left( \partial
_{r}\diamond \right) \right) .
\end{eqnarray*}%
Thus we multiply (\ref{firstRecursiveFormula}) by $r^{j}g^{1/4}\left[ \cos
\left( -cr\right) \mathbf{1}+\sin \left( -cr\right) \left( \partial
_{r}\diamond \right) \right] $. Then observe that%
\begin{eqnarray*}
&&\nabla _{\partial _{r}}\left( r^{j}g^{1/4}\left[ \cos \left( -cr\right) 
\mathbf{1}+\sin \left( -cr\right) \left( \partial _{r}\diamond \right) %
\right] u_{j}\right) \\
&=&r^{j}g^{1/4}\left[ \cos \left( -cr\right) \mathbf{1}+\sin \left(
-cr\right) \left( \partial _{r}\diamond \right) \right] \left( \nabla
_{\partial _{r}}+\left( \frac{j}{r}+\frac{\partial _{r}g}{4g}-c\left(
\partial _{r}\diamond \right) \right) \right) u_{j} \\
&=&-\frac{1}{r}r^{j}g^{1/4}\left[ \cos \left( -cr\right) \mathbf{1}+\sin
\left( -cr\right) \left( \partial _{r}\diamond \right) \right] \left(
D+c\right) ^{2}u_{j-1},
\end{eqnarray*}%
so the new recursion formula is%
\begin{eqnarray}
&&\nabla _{\partial _{r}}\left( r^{j}g^{1/4}\left[ \cos \left( -cr\right) 
\mathbf{1}+\sin \left( -cr\right) \left( \partial _{r}\diamond \right) %
\right] u_{j}\right)  \notag \\
&=&-r^{j-1}g^{1/4}\left[ \cos \left( -cr\right) \mathbf{1}+\sin \left(
-cr\right) \left( \partial _{r}\diamond \right) \right] \left( D+c\right)
^{2}u_{j-1}~.  \label{finalRecursionFormula}
\end{eqnarray}%
Substituting $j=0$, we see that $g^{1/4}\left[ \cos \left( -cr\right) 
\mathbf{1}+\sin \left( -cr\right) \left( \partial _{r}\diamond \right) %
\right] u_{0}$ is parallel along radial geodesics, which means that 
\begin{eqnarray}
u_{0}\left( r\right) &=&g^{-1/4}\left[ \cos \left( -cr\right) \mathbf{1}%
-\sin \left( -cr\right) \left( \partial _{r}\diamond \right) \right]  \notag
\\
&=&g^{-1/4}\left[ \cos \left( cr\right) \mathbf{1}+\sin \left( cr\right)
\left( \partial _{r}\diamond \right) \right] .  \label{u0formula}
\end{eqnarray}

In other words, $u_{0}\left( r\right) $ is the linear map from $E_{y}$ to $%
E_{x}$ (with $y$ being the origin of the geodesic coordinate system and $r$
being the distance from $y$ to $x$) defined by%
\begin{equation*}
s\left( y\right) \mapsto g^{-1/4}\left[ \cos \left( cr\right) \mathbf{1}%
+\sin \left( cr\right) \left( \partial _{r}\diamond \right) \right] s\left(
x\right) ,
\end{equation*}%
where $s\left( x\right) $ is the radial parallel translate of $s\left(
y\right) $ along the geodesic connecting $y$ to $x$.

By writing 
\begin{equation*}
u_{1}=u_{1}\left( 0\right) +\mathcal{O}\left( r\right) ,
\end{equation*}%
from (\ref{0thRecursiveFormula}) we see

\begin{equation*}
u_{1}+r\left( \nabla _{\partial _{r}}u_{1}+\left( \frac{\partial _{r}g}{4g}%
-c\left( \partial _{r}\diamond \right) \right) u_{1}\right) =-\left(
D+c\right) ^{2}u_{0}.
\end{equation*}%
\vspace{1pt}In particular,%
\begin{equation*}
u_{1}\left( 0\right) =\left( -\left( D+c\right) ^{2}u_{0}\right) \left(
0\right) .
\end{equation*}%
We have $r^{2}=x_{j}x_{j}$ , $r\partial _{r}=x_{j}\partial _{j}$ , and $g=1%
\mathbf{+}\frac{1}{3}R_{ipqi}x_{p}x_{q}+\mathcal{O}\left( r^{3}\right) $ in
geodesic normal coordinates in terms of the Riemann curvature tensor $%
R_{ijk\ell }$ at $x=0$ (see \cite[p. 104]{Roe}), using the convention that $%
R_{ijkl}=\left\langle \left( \nabla _{\partial _{k}}^{M}\nabla _{\partial
_{\ell }}^{M}-\nabla _{\partial _{\ell }}^{M}\nabla _{\partial
_{k}}^{M}\right) \partial _{j},\partial _{i}\right\rangle $. Using the
binomial expansion,

\begin{eqnarray*}
u_{0} &=&g^{-1/4}\left[ \cos \left( cr\right) \mathbf{1}+\sin \left(
cr\right) \left( \partial _{r}\diamond \right) \right] \\
&=&\mathbf{1+}cr\left( \partial _{r}\diamond \right) -\frac{c^{2}r^{2}}{2}%
\mathbf{1}-\frac{1}{12}R_{ijki}x_{j}x_{k}\mathbf{\mathbf{1}+}\mathcal{O}%
\left( r^{3}\right) \\
&=&\mathbf{1+}cx_{j}\left( \partial _{j}\diamond \right) -\frac{%
c^{2}x_{j}x_{j}}{2}\mathbf{1}-\frac{1}{12}R_{ijki}x_{j}x_{k}\mathbf{1+}%
\mathcal{O}\left( r^{3}\right) .
\end{eqnarray*}%
Then at $0$, 
\begin{eqnarray*}
\left( Du_{0}\right) \left( 0\right) &=&g^{pq}\left( \partial _{p}\diamond
\right) \nabla _{\partial _{q}}u_{0} \\
&=&\left( \partial _{p}\diamond \right) \nabla _{\partial _{p}}u_{0} \\
&=&\left( \partial _{p}\diamond \right) c\left( \partial _{p}\diamond
\right) =-nc\mathbf{1}.
\end{eqnarray*}%
At $0$, $\nabla _{\partial _{p}}\partial _{q}=0$ for all $p,q$ ; thus, from (%
\ref{Kformula}) and the above, 
\begin{eqnarray*}
\left( D^{2}u_{0}\right) \left( 0\right) &=&\left( -\nabla _{\partial
_{p}}\nabla _{\partial _{p}}+K\right) u_{0} \\
&=&\left( nc^{2}\mathbf{+}\frac{1}{6}R_{ijji}+K\right) \mathbf{1}=\left(
nc^{2}-\frac{1}{6}\mathrm{Scal}+K\right) \mathbf{1},
\end{eqnarray*}%
where $\mathrm{Scal}$ denotes the scalar curvature. Then 
\begin{eqnarray}
u_{1}\left( 0\right) &=&\left( -\left( D^{2}+2cD+c^{2}\right) u_{0}\right)
\left( 0\right)  \notag \\
&=&-\left( nc^{2}-\frac{1}{6}\mathrm{Scal}+K-2cnc+c^{2}\right) \mathbf{1} 
\notag \\
&=&\left( \left( n-1\right) c^{2}+\frac{1}{6}\mathrm{Scal}\right) \mathbf{1}%
-K\mathbf{.}  \label{u1Formula}
\end{eqnarray}%
We have shown that the heat kernel for $\left( D+c\right) ^{2}$ has the
expansion%
\begin{eqnarray*}
K_{c}\left( t,x,x\right) &:=&\exp \left( -t\left( D+c\right) ^{2}\right)
\left( x,x\right) \\
&=&\frac{1}{\left( 4\pi t\right) ^{n/2}}\left( \mathbf{1}+t\left( \left(
\left( n-1\right) c^{2}+\frac{1}{6}\mathrm{Scal}\right) \mathbf{1}-K\right) +%
\mathcal{O}\left( t^{2}\right) \right) ,  \notag
\end{eqnarray*}%
\begin{multline}
\mathrm{Tr}\exp \left( -t\left( D+c\right) ^{2}\right) =\frac{1}{\left( 4\pi
t\right) ^{n/2}}\Biggm(\widehat{n}\mathrm{vol}\left( M\right)  \notag \\
+t\left[ \widehat{n}\left( n-1\right) c^{2}\mathrm{vol}\left( M\right) +%
\frac{\widehat{n}}{6}\int_{M}\mathrm{Scal}-\int_{M}Tr\left( K\right) \right]
+\mathcal{O}\left( t^{2}\right) \Biggm).  \label{heatInvariants}
\end{multline}%
Here, $n$ is the dimension of the manifold, and $\widehat{n}$ is the rank of
the bundle $E$.

\vspace{1pt}The Clifford contracted curvature term $K$ has the form (see 
\cite[pp. 48--49]{Roe}, \cite[Thm. 3.52]{Be-G-V})%
\begin{equation*}
K=\frac{\mathrm{Scal}}{4}+F^{E\diagup S}.
\end{equation*}%
On a spin manifold, if $S$ is the spinor bundle, then $E\cong S\otimes W$
with connection $\nabla ^{S\otimes W}=\nabla ^{W}\otimes 1+1\otimes \nabla
^{S}$, and $F^{E\diagup S}$ is the twisting curvature of $E$, meaning%
\begin{equation*}
F^{E\diagup S}=F^{W}=\sum_{i<j}F^{W}\left( e_{i},e_{j}\right) \left(
e^{i}\diamond \right) \left( e^{j}\diamond \right) ,
\end{equation*}%
with $F^{W}$ the curvature of $\nabla ^{W}$. In particular, if $D$ is the
spin Dirac operator on a spin manifold, then $F^{W}=0$ and%
\begin{equation*}
\mathrm{Tr}\exp \left( -t\left( D+c\right) ^{2}\right) =\frac{1}{\left( 4\pi
t\right) ^{n/2}}\left( \widehat{n}\mathrm{vol}\left( M\right) +t\left[ 
\widehat{n}\left( n-1\right) c^{2}\mathrm{vol}\left( M\right) -\frac{%
\widehat{n}}{12}\int_{M}\mathrm{Scal}\right] +\mathcal{O}\left( t^{2}\right)
\right) .
\end{equation*}%
\vspace{1pt}

Observe that our first recursion formula (\ref{firstRecursiveFormula}) for
the heat invariant endomorphism $u_{j}$ corresponding to $\left( D+c\right)
^{2}$ is 
\begin{equation*}
\nabla _{\partial _{r}}u_{j}+\left( \frac{j}{r}+\frac{\partial _{r}g}{4g}%
-c\left( \partial _{r}\diamond \right) \right) u_{j}=-\frac{1}{r}\left(
D+c\right) ^{2}u_{j-1}~,
\end{equation*}%
where $r$ is the distance from the origin of the coordinate system, and the
differential equation holds along a geodesic from $0$ to $x$. For $j\geq 0$,
we expand 
\begin{equation*}
u_{j}=\sum_{k=0}^{K}c^{k}u_{j,k}~+\mathcal{O}\left( c^{K+1}\right) ,
\end{equation*}%
where each $u_{j,k}$ is independent of $c\in \mathbb{R}$. For consistency we
declare that $u_{j,k}=0$ if either $j$ or $k$ is negative. Our recursive
formula above implies that (collecting powers of $c$)%
\begin{equation}
\nabla _{\partial _{r}}u_{j,k}+\left( \frac{j}{r}+\frac{\partial _{r}g}{4g}%
\right) u_{j,k}=\left( \partial _{r}\diamond \right) u_{j,\left( k-1\right)
}-\frac{1}{r}D^{2}u_{\left( j-1\right) ,k}-\frac{2}{r}Du_{\left( j-1\right)
,\left( k-1\right) }-\frac{1}{r}u_{\left( j-1\right) ,\left( k-2\right) }.
\label{ujkRecursiveFormula}
\end{equation}

\begin{proposition}
We have 
\begin{equation*}
u_{j,k}=\mathcal{O}\left( r^{\max \left\{ k-2j,0\right\} }\right) .
\end{equation*}%
In particular, 
\begin{equation*}
u_{j,k}\left( 0\right) =0
\end{equation*}%
if $k>2j$, so that $u_{j}$ is a polynomial in $c$ of degree at most $2j$.
\end{proposition}

\begin{proof}
Clearly, $u_{j,0}=\mathcal{O}\left( 1\right) $ for all $j\geq 0$, as these
refer to the standard heat invariants (with $c=0$). Also, the formula holds
for $u_{0,k}$ by Taylor analysis of the expicit formula (\ref{u0formula}).
We prove the general case by induction; assume that the theorem holds for
all $\left( j,k\right) $ such that $0\leq j<J$ and $k\geq 0$ or $j=J$ and $%
0\leq k\leq K$, with $J\geq 1$ and $K\geq 0$. Then the formula preceding the
statement implies that%
\begin{equation*}
r\nabla _{\partial _{r}}u_{J,K+1}+\left( J+\frac{r\partial _{r}g}{4g}\right)
u_{J,K+1}=r\left( \partial _{r}\diamond \right) u_{J,K}-D^{2}u_{\left(
J-1\right) ,\left( K+1\right) }-2Du_{\left( J-1\right) ,K}-u_{\left(
J-1\right) ,\left( K-1\right) }~.
\end{equation*}%
Note that, given $A\left( r\right) =\mathcal{O}\left( r^{p}\right) $ is
smooth in $r$, we have $r\partial _{r}A\left( r\right) =\mathcal{O}\left(
r^{p}\right) $ if $p\neq 0$ and $r\partial _{r}A\left( r\right) =\mathcal{O}%
\left( r\right) $ if $p=0$. Similarly, $r\left( \partial _{r}\diamond
\right) A\left( r\right) =\mathcal{O}\left( r^{p+1}\right) $ since $\partial
_{r}\diamond $ is bounded and has constant norm. Then, by the induction
hypothesis,%
\begin{eqnarray*}
u_{J,K+1} &=&\mathcal{O}\left( r^{\max \left\{ K-2J,0\right\} +1}\right) +%
\mathcal{O}\left( r^{\max \left\{ K-2J+3-2,0\right\} }\right) \\
&&+\mathcal{O}\left( r^{\max \left\{ K-2J+2-1,0\right\} }\right) +\mathcal{O}%
\left( r^{\max \left\{ K-2J+1,0\right\} }\right) \\
&=&\mathcal{O}\left( r^{\max \left\{ K-2J+1,0\right\} }\right) ,
\end{eqnarray*}%
since $D\left( \mathcal{O}\left( r^{p}\right) \right) =\mathcal{O}\left(
r^{\max \left\{ p-1,0\right\} }\right) $ as long as the quantities are
smooth in $r$.
\end{proof}

Because $\left( \partial _{r}\diamond \right) ^{2j}=\left( -1\right) ^{j}$
and $\left( \partial _{r}\diamond \right) ^{2j+1}=\left( -1\right)
^{j}\left( \partial _{r}\diamond \right) $, from (\ref{u0formula}) we have 
\begin{equation*}
u_{0,k}=\frac{1}{k!}g^{-1/4}r^{k}\left( \partial _{r}\diamond \right) ^{k}.
\end{equation*}%
Also, since all of the $u_{j,0}$ are known (the standard heat invariants),
we may use (\ref{ujkRecursiveFormula}) to calculate $u_{j,k}$ for all $j\geq
0$, $k\geq 0$. That is,%
\begin{equation*}
\nabla _{\partial _{r}}\left( r^{j}g^{1/4}u_{j,k}\right) =r^{j}g^{1/4}\left(
\left( \partial _{r}\diamond \right) u_{j,\left( k-1\right) }-\frac{1}{r}%
D^{2}u_{\left( j-1\right) ,k}-\frac{2}{r}Du_{\left( j-1\right) ,\left(
k-1\right) }-\frac{1}{r}u_{\left( j-1\right) ,\left( k-2\right) }\right) ,
\end{equation*}%
and so the expression may be integrated along a radial geodesic to solve for 
$u_{j,k}$. We note that the formulas above and below for $u_{j,k}$ are
well-known for the case $k=0$ (see, for example, \cite[pp. 101ff]{Roe}, \cite%
{Gilk}); they are not easily found in the literature for general $k$ but may
be known to experts. From the formulas for $u_{0,k}$ and (\ref{u1Formula})
we have 
\begin{equation*}
u_{0,0}\left( 0\right) =1,~~u_{1,0}\left( 0\right) =\left( \frac{1}{6}%
\mathrm{Scal}\right) \mathbf{1}-K,~~u_{1,1}\left( 0\right)
=0,~~u_{1,2}\left( 0\right) =\left( n-1\right) \mathbf{1.}
\end{equation*}%
Let%
\begin{equation*}
a_{j,k}=\int_{M}\mathrm{tr}\left( u_{j,k}\left( x,x\right) \right) ~\mathrm{d%
\mathrm{\mathrm{vo}l}}~,
\end{equation*}%
where $u_{j,k}\left( x,x\right) $ is the expression at $r=0$ of $u_{j,k}$
found above. In particular, if $n$ is the dimension of the manifold $M$ and $%
\widehat{n}$ is the rank of the bundle $E$,

\begin{equation}
a_{0,0}=\widehat{n}\mathrm{vol}\left( M\right) ,~~a_{1,0}=\frac{\widehat{n}}{%
6}\int_{M}\mathrm{Scal}-\int_{M}Tr\left( K\right) ,~~a_{1,1}=0,~~a_{1,2}=%
\widehat{n}\left( n-1\right) \mathrm{vol}\left( M\right) .
\label{ajkFormulas}
\end{equation}%
Then the heat invariants $a_{j}\left( c\right) $ corresponding to $\left(
D+c\right) ^{2}$ satisfy%
\begin{equation}
a_{j}\left( c\right) =\int_{M}\mathrm{tr}\left( u_{j}\left( x,x\right)
\right) ~\mathrm{d\mathrm{\mathrm{vo}l}}=~\sum_{k=0}^{2j}c^{k}a_{j,k}~.
\label{aj(c)Formula}
\end{equation}

\subsection{The eta invariant for arbitrary manifolds with spectral symmetry 
\label{etaInvtArbMfldsSection}}

Suppose that $M$ is a closed Riemannian manifold of dimension $n$. Recall
from Proposition \ref{etaZetaProposition}, we wish to calculate $%
\lim\limits_{s\rightarrow 0}-s\zeta _{\left( D+c\right) ^{2}}\left( \frac{s+1%
}{2}\right) $, at a particular value of $c$ where $\dim \ker \left(
D+c\right) ^{2}=\left\{ 0\right\} $. From (\ref{K(txy)_asymptotics}), as $%
t\rightarrow 0^{+}$,%
\begin{equation*}
\sum_{\mu }e^{-t\mu }=\int_{M}\mathrm{tr}K_{c}\left( t,x,x\right) ~dV\left(
x\right) \sim \frac{1}{\left( 4\pi t\right) ^{n/2}}\left(
a_{0}+ta_{1}+t^{2}a_{2}+...\right) ,
\end{equation*}%
where $\left\{ \mu \right\} $ are the eigenvalues of $\left( D+c\right) ^{2}$
with multiplicities. The standard derivation of the analytic continuation of
the zeta function is as follows. For large $\mathrm{Re}\left( s\right) $, 
\begin{eqnarray*}
\zeta _{\left( D+c\right) ^{2}}\left( s\right) &=&\sum_{\mu }\mu ^{-s}=\frac{%
1}{\Gamma \left( s\right) }\int_{0}^{\infty }t^{s-1}\left( \sum_{\mu
}e^{-t\mu }\right) dt \\
&=&\frac{1}{\Gamma \left( s\right) }\int_{0}^{1}t^{s-1}\left( \frac{1}{%
\left( 4\pi t\right) ^{n/2}}\left( a_{0}+a_{1}t+...+a_{N}t^{N}\right)
\right) dt \\
&&+\frac{1}{\Gamma \left( s\right) }\int_{0}^{1}t^{s-1}\left( \sum e^{-t\mu
}-\frac{1}{\left( 4\pi t\right) ^{n/2}}\left(
a_{0}+a_{1}t+...+a_{N}t^{N}\right) \right) \\
&&+\frac{1}{\Gamma \left( s\right) }\int_{1}^{\infty }t^{s-1}\left( \sum
e^{-t\mu }\right) dt
\end{eqnarray*}%
\begin{equation*}
=\frac{1}{\left( 4\pi \right) ^{n/2}\Gamma \left( s\right) }%
\sum_{j=0}^{N}a_{j}\int_{0}^{1}t^{s-1-\frac{n}{2}+j}dt+\phi _{N}\left(
s\right) =\frac{1}{\left( 4\pi \right) ^{n/2}\Gamma \left( s\right) }%
\sum_{j=0}^{N}\frac{a_{j}}{s-\frac{n}{2}+j}+\phi _{N}\left( s\right) ,
\end{equation*}%
where $\phi _{N}\left( s\right) $ is holomorphic for $\mathrm{Re}s>\frac{n}{2%
}-N-1$, $\Gamma \left( \cdot \right) $ is the Gamma function, and $a_{j}$ is
the heat invariant corresponding to $\left( D+c\right) ^{2}$:%
\begin{equation*}
a_{j}=\int_{M}\mathrm{tr}\left( u_{j}\left( x,x\right) \right) ~\mathrm{d%
\mathrm{\mathrm{vo}l~.}}
\end{equation*}%
Then, since $\Gamma \left( \frac{1}{2}\right) =\sqrt{\pi }$, 
\begin{equation*}
\lim_{s\rightarrow 0}-s\zeta _{\left( D+c\right) ^{2}}\left( \frac{s+1}{2}%
\right) =\lim_{s\rightarrow 0}\frac{-s}{\left( 4\pi \right) ^{n/2}\Gamma
\left( \frac{s+1}{2}\right) }\frac{a_{\frac{n-1}{2}}}{\left( \frac{s+1}{2}-%
\frac{1}{2}\right) }=-2^{1-n}\pi ^{-\left( n+1\right) /2}a_{\frac{n-1}{2}},
\end{equation*}%
or%
\begin{equation*}
\frac{d}{dc}\eta _{c}\left( 0\right) =-2^{1-n}\pi ^{-\left( n+1\right) /2}a_{%
\frac{n-1}{2}}\left( c\right) .
\end{equation*}%
Note that if $n$ is even, $\frac{d}{dc}\eta _{c}\left( 0\right) =0$.

\vspace{0in}Now, suppose that there is a point of symmetry, $\overline{%
\lambda }<0$, in the spectrum $\sigma \left( D\right) $ of $D$, meaning that 
$\sigma \left( D\right) -\overline{\lambda }$ is symmetric about $0$ in $%
\mathbb{R}$. Then $\eta _{-\overline{\lambda }}\left( 0\right) =0$. We then
integrate the formula above from $c=0$ to $c=-\overline{\lambda }$. We have
a discontinuity (a jump of $+2$) at each $c\in \left( 0,-\overline{\lambda }%
\right) $ that is an eigenvalue of $-D$, due to the fact that $c\mapsto 
\mathrm{sgn}\left( \lambda +c\right) $ has a similar discontinuity near $%
c=-\lambda $. Also, if either $0$ or $-\overline{\lambda }$ are contained in
the spectrum of $-D$, then we will have a jump discontinuity of $+1$ at
those points. Let $c_{1}\leq ...\leq c_{k}$ be the points of $\left( 0,-%
\overline{\lambda }\right) $ that are eigenvalues of $-D$. Let $n_{0}$ be
the multiplicity of $0$ in $\sigma \left( D\right) $, $n_{-\overline{\lambda 
}}$ be the multiplicity of $\overline{\lambda }$ in $\sigma \left( D\right) $%
. Then the fundamental theorem of calculus yields%
\begin{eqnarray*}
\int_{0}^{c_{1}}\frac{d}{dc}\eta _{c}\left( 0\right) ~dc &=&\eta
_{c_{1}}\left( 0\right) -\eta _{0}\left( 0\right) -1-n_{0}, \\
\int_{c_{j}}^{c_{j+1}}\frac{d}{dc}\eta _{c}\left( 0\right) ~dc &=&\eta
_{c_{j+1}}\left( 0\right) -\eta _{c_{j}}\left( 0\right) -2, \\
\int_{c_{k}}^{-\overline{\lambda }}\frac{d}{dc}\eta _{c}\left( 0\right) ~dc
&=&\eta _{-\overline{\lambda }}\left( 0\right) -\eta _{k}\left( 0\right)
-1-n_{-\overline{\lambda }},
\end{eqnarray*}%
which add to%
\begin{equation*}
\int_{0}^{-\overline{\lambda }}\frac{d}{dc}\eta _{c}\left( 0\right) ~dc=\eta
_{-\overline{\lambda }}\left( 0\right) -\eta _{0}\left( 0\right)
-2k-n_{0}-n_{-\overline{\lambda }}.
\end{equation*}%
Therefore, since $\eta _{-\overline{\lambda }}\left( 0\right) =0$ and $\eta
_{0}\left( 0\right) =\eta \left( 0\right) $, 
\begin{eqnarray*}
\eta \left( 0\right) &=&-\int_{0}^{-\overline{\lambda }}\frac{d}{dc}\eta
_{c}\left( 0\right) ~dc-2k-n_{0}-n_{-\overline{\lambda }} \\
&=&\int_{-\overline{\lambda }}^{0}\frac{d}{dc}\eta _{c}\left( 0\right)
~dc-2k-n_{0}-n_{-\overline{\lambda }}
\end{eqnarray*}%
In the case where the point of symmetry is positive ($\overline{\lambda }>0$%
), the calculation above may be adapted in the following ways. We integrate
the formula for $\frac{d}{dc}\eta _{c}\left( 0\right) $ from $c=-\overline{%
\lambda }$ to $c=0$, and if $c_{1}\leq ...\leq c_{k}$ are the points of $%
\left( -\overline{\lambda },0\right) $ that are eigenvalues of $-D$, we have%
\begin{eqnarray*}
\int_{-\overline{\lambda }}^{c_{1}}\frac{d}{dc}\eta _{c}\left( 0\right) ~dc
&=&\eta _{c_{1}}\left( 0\right) -\eta _{-\overline{\lambda }}\left( 0\right)
-1-n_{-\overline{\lambda }}, \\
\int_{c_{j}}^{c_{j+1}}\frac{d}{dc}\eta _{c}\left( 0\right) ~dc &=&\eta
_{c_{j+1}}\left( 0\right) -\eta _{c_{j}}\left( 0\right) -2, \\
\int_{c_{k}}^{0}\frac{d}{dc}\eta _{c}\left( 0\right) ~dc &=&\eta _{0}\left(
0\right) -\eta _{k}\left( 0\right) -1-n_{0},
\end{eqnarray*}%
which yields%
\begin{equation*}
\eta _{0}\left( 0\right) =\int_{-\overline{\lambda }}^{0}\frac{d}{dc}\eta
_{c}\left( 0\right) ~dc+2k+n_{0}+n_{-\overline{\lambda }},
\end{equation*}%
with $n_{0},~n_{-\overline{\lambda }}$ defined above.

In general, if $\overline{\lambda }$ is the point of symmetry of $\sigma
\left( D\right) $,%
\begin{multline}
\eta _{0}\left( 0\right) =\eta _{-\overline{\lambda }}\left( 0\right)
+\int_{-\overline{\lambda }}^{0}\frac{d}{dc}\eta _{c}\left( 0\right) ~dc+%
\mathrm{sgn}\left( \overline{\lambda }\right) \left( 2\#\left( \sigma \left(
D\right) \cap I_{\overline{\lambda }}\right) +\#\left( \sigma \left(
-D\right) \cap \left\{ 0,-\overline{\lambda }\right\} \right) \right)  \notag
\\
=-2^{1-n}\pi ^{-\left( n+1\right) /2}\int_{-\overline{\lambda }}^{0}a_{\frac{%
n-1}{2}}\left( c\right) ~dc+\mathrm{sgn}\left( \overline{\lambda }\right)
2\#\left( \sigma \left( D\right) \cap I_{\overline{\lambda }}\right) +%
\mathrm{sgn}\left( \overline{\lambda }\right) \#\left( \sigma \left(
D\right) \cap \left\{ 0,\overline{\lambda }\right\} \right) ,
\label{etaArbitraryFirstForm}
\end{multline}%
where $I_{\overline{\lambda }}$ is $\left( 0,\overline{\lambda }\right) $ or 
$\left( \overline{\lambda },0\right) $, depending on the sign of $\overline{%
\lambda }$, and where the last two terms include multiplicities.

Thus, from the formula above and the expression for the heat invariant
coefficients $a_{j,k}$ in (\ref{aj(c)Formula}), we now have a formula for $%
\eta \left( 0\right) =\eta _{0}\left( 0\right) $. The same argument also
yields a formula for $\eta \left( 0\right) -\eta _{-\overline{\lambda }%
}\left( 0\right) $ in general.

\begin{theorem}
\label{etaArbitraryDimensionTheorem}\label%
{etaArbitraryDimensionTheoremNoSymmetry}

\begin{enumerate}
\item[]\text{\hspace{6in}}\newline

\item[(A)] Let $\sigma \left( D\right) -\overline{\lambda }$ be symmetric
about $0$ in $\mathbb{R}$. Then the eta invariant satisfies%
\begin{eqnarray*}
\eta \left( 0\right) &=&-2^{1-n}\pi ^{-\left( n+1\right) /2}\left(
~\sum_{k=0}^{n-1}\frac{\left( -1\right) ^{k}}{k+1}\overline{\lambda }%
^{k+1}a_{\frac{n-1}{2},k}~\right) \\
&&+\mathrm{sgn}\left( \overline{\lambda }\right) 2\#\left( \sigma \left(
D\right) \cap I_{\overline{\lambda }}\right) +\mathrm{sgn}\left( \overline{%
\lambda }\right) \#\left( \sigma \left( D\right) \cap \left\{ 0,\overline{%
\lambda }\right\} \right) ,
\end{eqnarray*}%
\newline
where $I_{\overline{\lambda }}$ is the open interval between $0$ and $%
\overline{\lambda }$, and where implicitly the last two terms include
multiplicities.

\item[(B)] Let $\overline{\lambda }$ be any real number. Then the eta
invariant satisfies%
\begin{eqnarray*}
\eta \left( 0\right) -\eta _{-\overline{\lambda }}\left( 0\right)
&=&-2^{1-n}\pi ^{-\left( n+1\right) /2}\left( ~\sum_{k=0}^{n-1}\frac{\left(
-1\right) ^{k}}{k+1}\overline{\lambda }^{k+1}a_{\frac{n-1}{2},k}~\right) \\
&&+\mathrm{sgn}\left( \overline{\lambda }\right) 2\#\left( \sigma \left(
D\right) \cap I_{\overline{\lambda }}\right) +\mathrm{sgn}\left( \overline{%
\lambda }\right) \#\left( \sigma \left( D\right) \cap \left\{ 0,\overline{%
\lambda }\right\} \right) ,
\end{eqnarray*}%
\newline
where $I_{\overline{\lambda }}$ is the open interval between $0$ and $%
\overline{\lambda }$, and where implicitly the last two terms include
multiplicities.
\end{enumerate}
\end{theorem}

\subsection{The zeta function and the eta invariant for three-manifolds\label%
{residuesNHeatInvts}\label{zetaEtaThreeMflds}}

By Theorem \ref{etaArbitraryDimensionTheoremNoSymmetry}, for $n=3$ we have%
\begin{eqnarray*}
\eta \left( 0\right) -\eta _{-\overline{\lambda }}\left( 0\right)
&=&-2^{-2}\pi ^{-2}\left( \overline{\lambda }^{1}a_{1,0}-\frac{1}{2}%
\overline{\lambda }^{2}a_{1,1}~+\frac{1}{3}\overline{\lambda }%
^{3}a_{1,2}~\right) \\
&&+\mathrm{sgn}\left( \overline{\lambda }\right) 2\#\left( \sigma \left(
D\right) \cap I_{\overline{\lambda }}\right) +\mathrm{sgn}\left( \overline{%
\lambda }\right) \#\left( \sigma \left( D\right) \cap \left\{ 0,\overline{%
\lambda }\right\} \right) ,
\end{eqnarray*}

From (\ref{ajkFormulas}), 
\begin{eqnarray*}
\eta \left( 0\right) -\eta _{-\overline{\lambda }}\left( 0\right) &=&-\frac{%
\widehat{n}\overline{\lambda }^{3}}{6\pi ^{2}}\mathrm{vol}\left( M\right) -%
\frac{\overline{\lambda }}{4\pi ^{2}}\left( \frac{\widehat{n}}{6}\int_{M}%
\mathrm{Scal}-\int_{M}Tr\left( K\right) \right) \\
&&+\mathrm{sgn}\left( \overline{\lambda }\right) \left( 2\#\left( \sigma
\left( D\right) \cap \left( 0,\overline{\lambda }\right) \right) +\#\left(
\sigma \left( D\right) \cap \left\{ 0,\overline{\lambda }\right\} \right)
\right) ,
\end{eqnarray*}%
where implicitly the last two terms include multiplicities. Note that every
three-manifold is spin, and thus if we let $F^{W}$ be the twisting
curvature, then 
\begin{equation*}
\int_{M}Tr\left( K\right) =\int_{M}\frac{\widehat{n}\mathrm{Scal}}{4}%
+\int_{M}Tr\left( F^{W}\right) .
\end{equation*}%
Then we have%
\begin{eqnarray}
\eta \left( 0\right) -\eta _{-\overline{\lambda }}\left( 0\right) &=&-\frac{%
\widehat{n}\overline{\lambda }^{3}}{6\pi ^{2}}\mathrm{vol}\left( M\right) +%
\frac{\overline{\lambda }}{4\pi ^{2}}\left( \frac{\widehat{n}}{12}\int_{M}%
\mathrm{Scal}+\int_{M}Tr\left( F^{W}\right) \right)  \notag \\
&&+\mathrm{sgn}\left( \overline{\lambda }\right) \left( 2\#\left( \sigma
\left( D\right) \cap \left( 0,\overline{\lambda }\right) \right) +\#\left(
\sigma \left( D\right) \cap \left\{ 0,\overline{\lambda }\right\} \right)
\right) .  \label{etaInvt3MfldsNegCase}
\end{eqnarray}

\section{Two-step Nilmanifolds and Dirac operators}

\subsection{Two-step Nilmanifolds and the Laplace-Beltrami operator\label%
{2StepN ilmfldsLaplacianSection}}

We review known results about the Laplacian on two-step nilmanifolds in this
section. A Lie algebra $\mathfrak{g}$ is two-step nilpotent if its derived
algebra $\mathfrak{z}^{\prime }=\left[ \mathfrak{g},\mathfrak{g}\right]
\not\equiv 0$ is contained in its center; i.e., $\left[ \mathfrak{g},\left[ 
\mathfrak{g},\mathfrak{g}\right] \right] \equiv 0$ but\ $\left[ \mathfrak{g},%
\mathfrak{g}\right] \not\equiv 0$. A Lie group $G$ is two-step nilpotent if
its Lie algebra is. Let $G$ be a simply connected two-step nilpotent Lie
group of dimension $n$ with Lie algebra $\mathfrak{g}$. Let $\Gamma $ be a
cocompact (i.e., $\Gamma \diagdown G$ compact), discrete subgroup of $G$,
and denote $M=\Gamma \diagdown G$. Fix an inner product $\left\langle
\;,\;\right\rangle $ on $\mathfrak{g}$, which corresponds to a
left-invariant metric on $G$, and which descends to a Riemannian metric on $%
M $. Note that left translation by noncentral elements is no longer an
isometry on $M$. Let $\left\{ X_{i}\right\} $ be an orthonormal basis of
left-invariant vector fields of $\mathfrak{g}$.

All nilpotent Lie groups are unimodular \cite[Proposition 1.2.10]{CoGr}, so
that the Laplace-Beltrami operator acting on smooth functions on $G$ can be
expressed as 
\begin{equation*}
\Delta =-\sum X_{i}^{2}\text{.}
\end{equation*}%
Denote by $\rho $ the (right) quasi-regular representation of $G$ on $%
L^{2}\left( \Gamma \diagdown G\right) $; i.e., for $g\in G$, $f\in
L^{2}\left( \Gamma \diagdown G\right) $,%
\begin{equation*}
\left( \rho \left( g\right) f\right) \left( x\right) =f\left( xg\right) 
\text{.}
\end{equation*}%
This is a unitary representation of $G$, and $\rho $ is the induced
representation of the trivial representation of $\Gamma $. Denote by $\rho
_{\ast }$ the associated unitary action of $\mathfrak{g}$ on $C^{\infty
}\left( \Gamma \diagdown G\right) \subset L^{2}\left( \Gamma \diagdown
G\right) $; i.e., for $X\in \mathfrak{g}$, $f\in C^{\infty }\left( \Gamma
\diagdown G\right) $, 
\begin{equation*}
\left( \rho _{\ast }\left( X\right) f\right) \left( x\right) =\left. \frac{d%
}{dt}\right\vert _{0}f\left( x\exp \left( tX\right) \right) \text{.}
\end{equation*}%
Because on smooth functions\ $\rho _{\ast }\left( X\right) f=Xf$, we may
rewrite the Laplacian as%
\begin{equation*}
\Delta =-\sum \left( \rho _{\ast }X_{i}\right) ^{2}.
\end{equation*}

By expressing the Laplace-Beltrami operator in terms of the representation $%
\rho $, we see that irreducible subspaces of the representation are also
invariant subspaces of the Laplacian. By restricting $\Delta $ to an
irreducible subspace of $L^{2}\left( \Gamma \diagdown G\right) $, \ Gordon,
Wilson, and Pesce (\cite{GW1}, \cite{Pe}) have been able in the two-step
nilpotent case to explicitly solve for its eigenvalues and eigenfunctions.
The Laplace spectrum of $\Gamma \diagdown G$ is then the union over all
irreducible subspaces of the spectrum of the restricted Laplacian. The
multiplicity of an eigenvalue is the sum over the irreducible subspaces of $%
L^{2}\left( \Gamma \diagdown G\right) $ of the eigenvalue's multiplicity in
the irreducible subspace times the multiplicity of the irreducible subspace
in $L^{2}\left( \Gamma \diagdown G\right) $.\vspace{1pt} The key ingredient
that distinguishes the nilpotent case in general, and the two-step nilpotent
case in particular, is that occurrence conditions, eigenvalues,
eigenfunctions, and multiplicities can be explicitly expressed in terms of $%
\log \Gamma $ and $\left( \mathfrak{g},\left\langle ~,~\right\rangle \right) 
$ using Kirillov theory. For more details, see \cite{GGt}.

Kirillov (\cite{Kir1}, \cite{Kir2}) proved that equivalence classes of
irreducible unitary representations of nilpotent Lie groups $G$ are in $1$-$%
1 $ correspondence with the orbits of the coadjoint action of $G$ on $%
\mathfrak{g}^{\ast }$. The coadjoint action is defined by, for $x\in
G,~\alpha \in \mathfrak{g}^{\ast },$%
\begin{equation*}
x\cdot \alpha =\alpha \circ \mathrm{Ad}\left( x^{-1}\right) \text{.}
\end{equation*}%
Given a fixed representative $\alpha \in \mathfrak{g}^{\ast }$ corresponding
to a coadjoint orbit, let $\pi _{\alpha }$ denote the associated irreducible
unitary representation of $G$ with representation space $W_{\alpha }$. The
possible dimensions of $W_{\alpha }$ are either $1$ (characters) or
infinite. L. F. Richardson (\cite{Ri}) computed the decomposition of $\rho $
into irreducibles.

\textbf{Notation: }Given $\alpha \in \mathfrak{g}^{\ast }$, let $B_{\alpha }:%
\mathfrak{g}\times \mathfrak{g}\rightarrow \mathbb{R}$ be defined by 
\begin{equation*}
B_{\alpha }\left( X,Y\right) =\alpha \left( \left[ X,Y\right] \right) .
\end{equation*}%
Let $\mathfrak{g}_{\alpha }=\ker \left( B_{\alpha }\right) =\left\{ X\in 
\mathfrak{g}:B_{\alpha }\left( X,Y\right) =0\text{ for all }Y\in \mathfrak{g}%
\right\} $, let $\overline{B_{\alpha }}$ be the nondegenerate skew-symmetric
bilinear form induced by $B_{\alpha }$ on $\mathfrak{g}\diagup \mathfrak{g}%
_{\alpha }$, and denote by $\pm i~d_{1},...,\pm i~d_{r}$ \ the eigenvalues
of $\overline{B_{\alpha }}$. Note $\log \Gamma $ generates a lattice $%
\mathcal{L}$ in $\mathfrak{g}$ \cite[proof of Thm 2.4]{GW1}. Let $\mathcal{A}%
_{\alpha }=\mathcal{L}\diagup \left( \mathcal{L}\cap \mathfrak{g}_{\alpha
}\right) $. Let 
\begin{equation*}
\Delta _{\alpha }=\left. \Delta \right\vert _{W_{\alpha }}\text{.}
\end{equation*}

In the two-step nilpotent case, H. Pesce explicitly calculated the spectrum
of the restricted Laplace-Beltrami operator $\Delta _{\alpha }$ as follows.

\begin{proposition}
(\cite[Section II\ and Appendix A]{Pe}) We continue the notation above.

\begin{enumerate}
\item $\pi _{\alpha }$ occurs in the representation $L^{2}\left( \Gamma
\diagdown G\right) $ if and only if 
\begin{equation*}
\alpha \left( \log \Gamma \cap \mathfrak{g}_{\alpha }\right) \subset \mathbb{%
Z.}
\end{equation*}

\item If $\pi _{\alpha }$ occurs and $\alpha \left( \left[ \mathfrak{g},%
\mathfrak{g}\right] \right) =\left\{ 0\right\} $, then $\pi _{\alpha }$ is
one-dimensional and occurs with multiplicity $m_{a}=1$. The Laplace spectrum
associated to this irreducible subspace is 
\begin{equation*}
spec\left( \Delta _{\alpha }\right) =\left\{ 4\pi ^{2}\left\Vert \alpha
\right\Vert ^{2}\right\} .
\end{equation*}

\item If $\pi _{\alpha }$ occurs and $\alpha \left( \left[ \mathfrak{g},%
\mathfrak{g}\right] \right) \neq \left\{ 0\right\} $, then $\pi _{\alpha }$
is infinite-dimensional and occurs with multiplicity 
\begin{equation*}
m_{\alpha }=\sqrt{\det \left( \overline{B_{\alpha }}\right) },
\end{equation*}%
where the determinant is computed with respect to (any) lattice basis of $%
\mathcal{A}_{\alpha }\subset \mathfrak{g}\diagup \mathfrak{g}_{\alpha }$.
The Laplace spectrum associated to this irreducible subspace is%
\begin{equation*}
\mathrm{spec}\left( \Delta _{\alpha }\right) =\left\{ \mu \left( \alpha
,p\right) :p\in \left( \mathbb{Z}_{\geq 0}\right) ^{m}\right\} ,
\end{equation*}%
where 
\begin{equation*}
\mu \left( \alpha ,p\right) =4\pi ^{2}\sum \alpha \left( Z_{i}\right)
^{2}+2\pi \sum \left( 2p_{j}+1\right) d_{j},
\end{equation*}%
with $\left\{ Z_{1},...,Z_{k}\right\} $ an orthonormal basis of $\mathfrak{g}%
_{\alpha }$. The multiplicity of $\mu $ in $spec\left( \Delta _{\alpha
}\right) $ is the number of $p\in \left( \mathbb{Z}_{\geq 0}\right) ^{m}$
satisfying $\mu \left( \alpha ,p\right) =\mu $.
\end{enumerate}
\end{proposition}

\begin{remark}
In other words, the multiplicity of an eigenvalue $\lambda $ is the sum of
the multiplicity of $\lambda $ as an eigenvalue in each $\Delta _{\alpha }$
times the multiplicity of $\pi _{a}$ in the representation $L^{2}\left(
\Gamma \diagdown G\right) $.
\end{remark}

\subsection{The Dirac operator on two-step nilmanifolds\label%
{DiracOp2StepSection}}

\vspace{0in}As we intend to calculate the eta invariant of the spin Dirac
operator, we now extend Pesce's results to the Dirac setting. Recall that $G$
is a simply connected $n$-dimensional two-step nilpotent Lie group with Lie
algebra $\mathfrak{g}$ and $\Gamma $ is a cocompact, discrete subgroup of $G$%
. We fix an inner product on $\mathfrak{g}$, which corresponds to a
left-invariant metric on $G$, which descends to a Riemannian metric on $%
\Gamma \diagdown G$.

Let $\Sigma _{n}$ be a standard irreducible spinor representation (see \cite[%
Section 3.2]{Be-G-V}), also considered as a trivial bundle over $G$. A spin
structure and the corresponding spinor bundle $\Sigma _{\varepsilon }$ over $%
\Gamma \diagdown G$ are determined by $\Sigma _{n}$ and a homomorphism $%
\varepsilon :\Gamma \rightarrow \left\{ \pm 1\right\} $ (see \cite[Prop
3.34, p. 114]{Be-G-V}). We have 
\begin{equation}
L^{2}\left( \Gamma \diagdown G,\Sigma _{\varepsilon }\right) \cong
L_{\varepsilon }^{2}\left( \Gamma \diagdown G\right) \otimes _{\mathbb{C}%
}\Sigma _{n},  \label{L2TensorProductIsomorphism}
\end{equation}%
where $L_{\varepsilon }^{2}\left( \Gamma \diagdown G\right) $ is defined by%
\begin{equation}
L_{\varepsilon }^{2}\left( \Gamma \diagdown G\right) =\left\{ f\in L_{%
\mathrm{loc}}^{2}\left( G\right) :f\left( \gamma x\right) =\varepsilon
\left( \gamma \right) f\left( x\right) \text{ for all }\gamma \in \Gamma
,x\in G\right\} .  \label{L2epsilon}
\end{equation}%
The isomorphism from $L_{\varepsilon }^{2}\left( \Gamma \diagdown G\right)
\otimes _{\mathbb{C}}\Sigma _{n}$ to $L^{2}\left( \Gamma \diagdown G,\Sigma
_{\varepsilon }\right) $ is $f\otimes s\mapsto fs$, where $\Sigma _{n}$ is
identified with the constant sections $G\rightarrow G\times \Sigma _{n}$.
Clifford multiplication by elements of $T\left( \Gamma \diagdown G\right)
\cong \Gamma \diagdown G\times \mathfrak{g}$ is given by the standard
Clifford action $\diamond $ of $\mathbb{C}\mathrm{l}\left( \mathfrak{g}%
\right) ~$on $\Sigma _{n}$. That is, $\xi \in \mathfrak{g}$ acts on $%
L_{\varepsilon }^{2}\left( \Gamma \diagdown G\right) \otimes _{\mathbb{C}%
}\Sigma _{n}$ by%
\begin{equation*}
\xi \diamond \left( fs\right) =f\left( \xi \diamond s\right) .
\end{equation*}%
By construction, $\left( \xi \diamond \right) $ is a constant matrix on $%
\Gamma \diagdown G$ for every left-invariant vector field $\xi $.

Note that the (Clifford) connection on any spinor bundle is given by%
\begin{equation}
\nabla _{E_{i}}^{\Sigma }=\partial _{E_{i}}+\frac{1}{4}\sum_{j,k}\Gamma
_{ij}^{k}\left( E_{j}\diamond \right) \left( E_{k}\diamond \right)
\label{connection Formula}
\end{equation}%
according to the Ammann-B\"{a}r formula \cite[formula 1.1]{Am-Ba}, where $%
\left\{ E_{j}\right\} $ is a left-invariant orthonormal basis of the tangent
space, $\Gamma _{ij}^{k}$ are the Christoffel symbols associated to the
metric and frame, and $\partial _{E_{i}}$ is a directional derivative. In
our case, we use the left-invariant metric on $\mathfrak{g}$, yielding a
metric on $\Gamma \diagdown G$. Then the Dirac operator $D$ on $\Gamma
\diagdown G$ acts on $L_{\varepsilon }^{2}\left( \Gamma \diagdown G\right)
\otimes \Sigma _{n}$ by%
\begin{eqnarray*}
D &=&\sum \left( E_{i}\diamond \right) \nabla _{E_{i}}^{\Sigma } \\
&=&\sum_{i}\left( E_{i}\diamond \right) \partial _{E_{i}}\mathbf{+}\frac{1}{4%
}\sum_{i,j,k}\Gamma _{ij}^{k}\left( E_{i}\diamond E_{j}\diamond
E_{k}\diamond \right)
\end{eqnarray*}%
If $\rho _{\varepsilon }$ denotes right multiplication acting on $%
L_{\varepsilon }^{2}\left( \Gamma \diagdown G\right) $, we have 
\begin{eqnarray*}
\partial _{E_{i}} &=&\left. \frac{d}{dt}\right\vert _{0}\rho _{\varepsilon
}\left( \exp \left( tE_{i}\right) \right) \\
&=&\rho _{\varepsilon \ast }\left( E_{i}\right) .
\end{eqnarray*}%
Note that $\rho _{\varepsilon }$ is the induced representation of $%
\varepsilon :\Gamma \rightarrow \left\{ \pm 1\right\} $ to $G$. The
Christoffel symbols are defined by 
\begin{equation*}
\nabla _{E_{i}}E_{j}=\sum \Gamma _{ij}^{k}E_{k},
\end{equation*}%
and the Koszul formula gives 
\begin{equation*}
2\Gamma _{ij}^{k}=-\left\langle E_{i},\left[ E_{j},E_{k}\right]
\right\rangle +\left\langle E_{j},\left[ E_{k},E_{i}\right] \right\rangle
+\left\langle E_{k},\left[ E_{i},E_{j}\right] \right\rangle .
\end{equation*}%
At this point, the formulas given above are completely general for any Lie
group $G$ with a left-invariant metric.

We now assume $G$ is 2-step nilpotent, so that $\left\langle E_{i},\left[
E_{j},E_{k}\right] \right\rangle =0$ unless $E_{i}$ is in the center of $%
\mathfrak{g}$. If $\mathfrak{g}=\mathfrak{z}\oplus \mathfrak{v}$ with $%
\mathfrak{z}$ the center and $\mathfrak{v}=\mathfrak{z}^{\bot }$, its
orthogonal complement, then the inner product on $\mathfrak{g}$ is
determined by and determines the map $j:\mathfrak{z}\rightarrow \mathfrak{so}%
\left( \mathfrak{v}\right) $ defined as 
\begin{equation}
\left\langle j\left( Z\right) X,A\right\rangle =\left\langle Z,\left[ X,A%
\right] \right\rangle  \label{j map}
\end{equation}%
for all $Z\in \mathfrak{z}$ and all $X,A\in \mathfrak{v}$. See, for example, 
\cite[p.618ff]{Eb}. Note that if $\left\langle Z,\left[ \mathfrak{g},%
\mathfrak{g}\right] \right\rangle =0$, then $j\left( Z\right) $ is the zero
map.

Let $k_{0}$ be the dimension of the center and $k_{0}+m_{0}$ the dimension
of $\mathfrak{g}$, and we choose the orthonormal basis $\left\{
Z_{1},...,Z_{k_{0}},X_{1},...,X_{m_{0}}\right\} $ so that $\left\{
Z_{i}\right\} $ is an orthonormal basis of $\mathfrak{z}$ and $\left\{
X_{i}\right\} $ is an orthonormal basis of $\mathfrak{v}$. Then one easily
verifies that%
\begin{equation*}
\nabla _{Z_{i}}X_{k}=\nabla _{X_{k}}Z_{i}=-\frac{1}{2}j\left( Z_{i}\right)
X_{k},~~\nabla _{X_{i}}X_{k}=\frac{1}{2}\left[ X_{i},X_{k}\right] ,~~\nabla
_{Z_{i}}Z_{k}=0.
\end{equation*}

We label $%
E_{1}=Z_{1},...,E_{k_{0}}=Z_{k_{0}},E_{k_{0}+1}=X_{1},...,E_{k_{0}+m_{0}}=X_{m_{0}} 
$. The Christoffel symbols satisfy $\Gamma _{pq}^{r}=0$ if at least two of $%
p,q,r$ are $\leq k_{0}$ or if $p,q,r>k_{0}$. If $a\leq k_{0}$, $b,q>k_{0}$,%
\begin{eqnarray*}
2\Gamma _{bq}^{a} &=&-2\Gamma _{qb}^{a}=2\Gamma _{aq}^{b}=2\Gamma
_{qa}^{b}=-2\Gamma _{ab}^{q}=-2\Gamma _{ba}^{q} \\
&=&\left\langle Z_{a},\left[ X_{b-k_{0}},X_{q-k_{0}}\right] \right\rangle
=\left\langle j\left( Z_{a}\right) X_{b-k_{0}},X_{q-k_{0}}\right\rangle .
\end{eqnarray*}

Letting $\partial _{E_{i}}=\partial _{i}$, $C_{i}=\left( E_{i}\diamond
\right) $, $C_{abq}=\left( E_{a}\diamond E_{b}\diamond E_{q}\diamond \right) 
$, etc., the Dirac operator is%
\begin{eqnarray*}
D &=&\sum_{i}\partial _{i}C_{i}\mathbf{+}\frac{1}{4}\sum_{i,j,k}\Gamma
_{ij}^{k}C_{ijk} \\
&=&\sum_{i}\partial _{i}C_{i}\mathbf{+}\frac{1}{4}\sum_{a\leq
k_{0};~b,q>k_{0}}\left( \Gamma _{bq}^{a}C_{bqa}+\Gamma
_{aq}^{b}C_{aqb}+\Gamma _{ba}^{q}C_{baq}\right)  \\
&=&\sum_{i}\partial _{i}C_{i}\mathbf{+}\frac{1}{4}\sum_{a\leq
k_{0};~q>b>k_{0}}\left( \Gamma _{bq}^{a}C_{bqa}+\Gamma
_{qb}^{a}C_{qba}+\Gamma _{aq}^{b}C_{aqb}+\Gamma _{ab}^{q}C_{abq}+\Gamma
_{ba}^{q}C_{baq}+\Gamma _{qa}^{b}C_{qab}\right)  \\
&=&\sum_{i}\partial _{i}C_{i}\mathbf{+}\frac{1}{2}\sum_{a\leq
k_{0};~q>b>k_{0}}\Gamma _{bq}^{a}C_{abq} \\
&=&\sum_{i}\partial _{i}C_{i}\mathbf{+}\frac{1}{4}\sum_{a\leq
k_{0};~~q>b>k_{0}}\left\langle Z_{a},\left[ X_{b-k_{0}},X_{q-k_{0}}\right]
\right\rangle \left( Z_{a}\diamond X_{b-k_{0}}\diamond X_{q-k_{0}}\diamond
\right) ,
\end{eqnarray*}%
so%
\begin{eqnarray}
D &=&\sum_{i}\left( E_{i}\diamond \right) \partial _{E_{i}}\mathbf{+}\frac{1%
}{4}\sum_{a\leq k_{0};~b<i\leq m_{0}}\left\langle Z_{a},\left[ X_{b},X_{i}%
\right] \right\rangle \left( Z_{a}\diamond X_{b}\diamond X_{i}\diamond
\right)   \notag \\
&=&\sum_{i}\left( E_{i}\diamond \right) \rho _{\varepsilon \ast }\left(
E_{i}\right) \mathbf{+}\frac{1}{2}\sum_{a\leq k_{0}}Z_{a}\diamond j\left(
Z_{a}\right) .  \label{Dirac2Step}
\end{eqnarray}%
In the expression above, we have used the fact that $j\left( Z_{a}\right)
\in \mathfrak{so}\left( m_{0}\right) =\mathfrak{spin}\left( m_{0}\right) $
and have therefore identified 
\begin{equation*}
j\left( Z_{a}\right) :=\frac{1}{2}\sum_{b<i\leq m_{0}}\left\langle j\left(
Z_{a}\right) X_{b},X_{i}\right\rangle X_{b}\diamond X_{i}\diamond .
\end{equation*}%
The formula above works for any two-step nilmanifold.

\begin{example}
In the three-dimensional Heisenberg case, for some constant $A>0$, we let
\linebreak $\left\{ X_{1}=\frac{1}{\sqrt{A}}X,X_{2}=\frac{1}{\sqrt{A}}%
Y,Z\right\} $ be an orthonormal frame with $\left[ X,Y\right] =Z$. We choose
a basis of $\Sigma _{3}\cong \mathbb{C}^{2}$ so that 
\begin{equation*}
\left( Z\diamond \right) =\left( 
\begin{array}{cc}
i & 0 \\ 
0 & -i%
\end{array}%
\right) ,~\left( X_{1}\diamond \right) =\left( 
\begin{array}{cc}
0 & i \\ 
i & 0%
\end{array}%
\right) ,~\left( X_{2}\diamond \right) =\left( 
\begin{array}{cc}
0 & -1 \\ 
1 & 0%
\end{array}%
\right) .
\end{equation*}%
Then 
\begin{eqnarray*}
\left\langle Z,\left[ X_{1},X_{2}\right] \right\rangle &=&\frac{1}{A}, \\
\left( Z\diamond X_{1}\diamond X_{2}\diamond \right) &=&-\mathbf{1},
\end{eqnarray*}%
so the equation above becomes%
\begin{equation*}
D=\sum_{i=1}^{3}\left( E_{i}\diamond \right) \partial _{e_{i}}-\frac{1}{4A},
\end{equation*}%
as seen in \cite[Equation 3.2]{Am-Ba}, with $d^{2}T=\frac{1}{A}$ in their
notation.
\end{example}

\subsection{Analogue of Pesce's theorem for spinors\label%
{AnaloguePescesThmSpinorsSection}}

In this section, we decompose $L_{\varepsilon }^{2}\left( \Gamma \diagdown
G\right) $ as a direct sum of irreducible representations. Let $\alpha \in 
\mathfrak{g}^{\ast }$. Recall $B_{\alpha }\left( X,Y\right) :=\alpha \left( %
\left[ X,Y\right] \right) $, $\mathfrak{g}_{\alpha }=\ker B_{\alpha }$, so
that $\alpha \left( \left[ \mathfrak{g}_{\alpha },\mathfrak{g}\right]
\right) =0$. Let $\mathfrak{g}^{\alpha }$ be a maximal polarizer of $\alpha $%
, meaning that it is a subalgebra of $\mathfrak{g}$ such that $\alpha \left( %
\left[ \mathfrak{g}^{\alpha },\mathfrak{g}^{\alpha }\right] \right) =0$ and
there does not exist a subalgebra $\mathfrak{h}$ with the same property such
that $\mathfrak{g}^{\alpha }\subsetneq \mathfrak{h}\subseteq \mathfrak{g}$.
Note that for every $\alpha \in \mathfrak{g}^{\ast }$ and every choice of $%
\mathfrak{g}^{\alpha }$,%
\begin{equation*}
\mathfrak{g}_{\alpha }\subset \mathfrak{g}^{\alpha }.
\end{equation*}%
Given $\mathfrak{g}^{\alpha }$, let $G^{\alpha }=\exp \left( \mathfrak{g}%
^{\alpha }\right) $.

\begin{lemma}
(Lemma 4 from \cite[Appendix A]{Pe}) \label{PesceLemmaA4}Let $\alpha \in 
\mathfrak{g}^{\ast }$, $\alpha \left( \left[ \mathfrak{g},\mathfrak{g}\right]
\right) \neq 0$ and $B_{\alpha }\left( X,Y\right) =\alpha \left( \left[ X,Y%
\right] \right) \in \mathbb{Z}$ for all $X,Y\in \log \Gamma $. Then there
exists a basis $\left\{
U_{1},...,U_{m},V_{1},...,V_{m},W_{1},...,W_{k}\right\} $of $\mathfrak{g}$
formed of elements of $\log \Gamma $, and there exist integers $%
r_{1},...,r_{k}$ such that

\begin{enumerate}
\item We have%
\begin{eqnarray*}
B_{\alpha }\left( U_{i},V_{i}\right) &=&\alpha \left( \left[ U_{i},V_{i}%
\right] \right) =r_{i}, \\
B_{\alpha }\left( U_{i},V_{j}\right) &=&0\text{ if }i\neq j,\text{ and} \\
B_{\alpha }\left( U_{i},U_{j}\right) &=&B_{\alpha }\left( V_{i},V_{j}\right)
=0\text{ for all }i,j.
\end{eqnarray*}

\item $\left\{ W_{1},...,W_{k}\right\} $ is a basis of $\mathfrak{g}_{\alpha
}$, $\left\{ W_{1},...,W_{k_{1}}\right\} $ is a basis of $\left[ \mathfrak{g}%
,\mathfrak{g}\right] $, $k_{1}\leq k$,

\item $\left[ \mathfrak{g},\mathfrak{g}\right] \cap \log \Gamma =\mathrm{span%
}_{\mathbb{Z}}\left\{ W_{1},...,W_{k_{1}}\right\} $.
\end{enumerate}
\end{lemma}

\begin{remark}
It follows from Pesce's proof of this Lemma that we may also choose $\left\{
W_{1},...,W_{k_{0}}\right\} $ to be a basis of $\mathfrak{z}$, with $%
k_{1}\leq k_{0}\leq k$.
\end{remark}

As before, $\log \Gamma $ generates a lattice $\mathcal{L}$ in $\mathfrak{g}$%
. Let $\mathcal{A}_{\alpha }=\mathcal{L}\diagup \left( \mathcal{L}\cap 
\mathfrak{g}_{\alpha }\right) $. When $\pi _{\alpha }$ occurs, this will be
a lattice in $\mathfrak{g}\diagup \mathfrak{g}_{\alpha }$.

\begin{proposition}
(Version of Pesce Occurrence Condition (\cite[Proposition 9 of Appendix A]%
{Pe}) for Dirac spinors) The representation $\pi _{\alpha }$ appears in $%
L_{\varepsilon }^{2}\left( \Gamma \diagdown G\right) $ if and only if 
\begin{equation}
\alpha \left( \log \gamma \right) \in \mathbb{Z}+\frac{1-\varepsilon \left(
\gamma \right) }{4}  \label{occurrenceConditionForSpinors}
\end{equation}%
for all $\gamma \in \Gamma \cap G_{\alpha }$. In this case, the multiplicity
of $\pi _{\alpha }$ is $m_{\alpha }=1$ if $\alpha \left( \left[ \mathfrak{g},%
\mathfrak{g}\right] \right) =\left\{ 0\right\} $, and otherwise%
\begin{equation*}
m_{\alpha }=\sqrt{\det \left( \overline{B_{\alpha }}\right) },
\end{equation*}%
where the determinant is computed with respect to (any) lattice basis of $%
\mathcal{A}_{\alpha }\subset \mathfrak{g}\diagup \mathfrak{g}_{\alpha }$.
\end{proposition}

\begin{proof}
Items 4 through 8 in \cite[Appendix A]{Pe} apply in this situation.

If $\alpha \left( \left[ \mathfrak{g},\mathfrak{g}\right] \right) =0$, then $%
\mathfrak{g}=\mathfrak{g}_{\alpha }=\mathfrak{g}^{\alpha }$. Then condition (%
\ref{occurrenceConditionForSpinors}) is equivalent to Theorem \ref%
{occurInL2EpsTheorem}. In addition, using Theorem \ref%
{multiplicityAndOccurInL2EpsTheorem}, 
\begin{eqnarray*}
m\left( \pi _{\alpha },L_{\varepsilon }^{2}\left( \Gamma \diagdown G\right)
\right) &=&\#\left( \left( G^{\alpha }\diagdown G\right) _{\varepsilon
}\diagup \Gamma \right) \\
&=&\#\left( \left( G\diagdown G\right) _{\varepsilon }\diagup \Gamma \right)
=1.
\end{eqnarray*}%
For the remainder of the proof, we assume $\alpha \left( \left[ \mathfrak{g},%
\mathfrak{g}\right] \right) \neq 0$. First, we assume $\pi _{\alpha }$
appears in $L_{\varepsilon }^{2}\left( \Gamma \diagdown G\right) $. Then, by
Theorem \ref{occurInL2EpsTheorem}, there exists $\alpha ^{\prime }$ in the
coadjoint orbit of $\alpha $ such that $\left( \overline{\alpha ^{\prime }}%
,G^{\alpha ^{\prime }}\right) $ is an $\varepsilon $-integral point, where $%
\overline{\alpha ^{\prime }}=\overline{\alpha }\circ I_{x}$ ($I_{x}=$
conjugation by $x$), $\alpha ^{\prime }=\alpha \circ $\textrm{$Ad$}$\left(
x\right) $ and $G^{\alpha ^{\prime }}=I_{x^{-1}}\left( G^{\alpha }\right) $
such that 
\begin{equation*}
\alpha ^{\prime }\left( \log \gamma \right) \in \left\{ 
\begin{array}{ll}
\mathbb{Z} & \text{if~}\varepsilon \left( \gamma \right) =1 \\ 
\frac{1}{2}+\mathbb{Z~} & \text{if~}\varepsilon \left( \gamma \right) =-1%
\end{array}%
\right.
\end{equation*}%
for all $\gamma \in \Gamma \cap G^{\alpha ^{\prime }}$. In the two-step
case, $G^{\alpha ^{\prime }}=G^{\alpha }$ since $I_{x}\left( y\right)
y^{-1}\in Z\left( G\right) $ for all $x,y\in G$, and $Z\left( G\right)
\subseteq G_{\alpha }\subseteq G^{\alpha ^{\prime }}$. Also in the two-step
case, if $\alpha ,\alpha ^{\prime }$ lie in the same coadjoint orbit, then
there exists $X\in \mathfrak{g}$ such that $\alpha ^{\prime }=\alpha \circ
\left( I+\mathrm{ad}\left( X\right) \right) $. Thus%
\begin{equation*}
\alpha \left( \log \gamma +\mathrm{ad}\left( X\right) \log \gamma \right)
\in \left\{ 
\begin{array}{ll}
\mathbb{Z} & \text{if~}\varepsilon \left( \gamma \right) =1 \\ 
\frac{1}{2}+\mathbb{Z~} & \text{if~}\varepsilon \left( \gamma \right) =-1%
\end{array}%
\right.
\end{equation*}%
for all $\gamma \in \Gamma \cap G^{\alpha }$. This implies the same
condition is met for all $\gamma \in \Gamma \cap G_{\alpha }$, in which case 
$\alpha \left( \mathrm{ad}\left( X\right) \log \gamma \right) =\alpha \left( %
\left[ X,\log \gamma \right] \right) =0$, by definition of $G_{\alpha }$.

On the other hand, suppose%
\begin{equation*}
\alpha \left( \log \gamma \right) \in \mathbb{Z}+\frac{1-\varepsilon \left(
\gamma \right) }{4}
\end{equation*}%
for all $\gamma \in \Gamma \cap G_{\alpha }$. Note that if $X,Y\in \log
\Gamma $, then $\left[ X,Y\right] \in \log \Gamma $ since $\left[ \exp
X,\exp Y\right] =\exp \left( \left[ X,Y\right] \right) $ (since $G$ is
two-step). Therefore, $\alpha \left( \left[ X,Y\right] \right) \in \mathbb{Z}
$, since $\varepsilon \left( \left[ \exp X,\exp Y\right] \right) =1$. We can
then use Lemma \ref{PesceLemmaA4} to construct a basis $\left\{
U_{1},...,U_{m},V_{1},...,V_{m},W_{1},...,W_{k}\right\} \subset \log \Gamma $
of $\mathfrak{g}$ and integers $r_{1},...,r_{m}$ such that $\alpha \left( %
\left[ U_{j},V_{j}\right] \right) =r_{j}$. Set $\mathfrak{h}=\mathrm{span}_{%
\mathbb{R}}\left\{ V_{1},...,V_{m},W_{1},...,W_{k}\right\} $. Then $%
\mathfrak{h}$ is a rational ideal of $\mathfrak{g}$, since $\left[ \mathfrak{%
h},\mathfrak{g}\right] \subseteq \mathfrak{z}\subseteq \mathfrak{h}$
(two-step condition), and $\mathfrak{h}$ is a polarizer of $\alpha $. Set $%
H=\exp \left( \mathfrak{h}\right) $, which is a normal subgroup of $G$. Note
that%
\begin{equation*}
H=\left\{ \prod_{i=1}^{m}\exp \left( y_{i}V_{i}\right) \prod_{j=1}^{k}\exp
\left( z_{j}W_{j}\right) :y_{i},z_{j}\in \mathbb{R}\right\} .
\end{equation*}%
Define $\overline{\alpha }\left( \exp \left( X\right) \right) =\exp \left(
2\pi i\alpha \left( X\right) \right) $ for all $X\in \mathfrak{h}$. By
Theorem \ref{occurInL2EpsTheorem}, to prove that $\pi _{\alpha }$ occurs, we
need only construct an $\varepsilon $-integral point in the $G$-orbit of $%
\left( \overline{\alpha },H\right) $. For $x\in G$, define $x_{i}$, $y_{i}$, 
$z_{j}$ by the formula 
\begin{equation*}
x=\prod_{i^{\prime }=1}^{m}\exp \left( x_{i^{\prime }}U_{i^{\prime }}\right)
\prod_{i=1}^{m}\exp \left( y_{i}V_{i}\right) \prod_{j=1}^{k}\exp \left(
z_{j}W_{j}\right) ,
\end{equation*}%
and define $p_{i}$, $q_{i}$, $\eta _{j}$ by 
\begin{equation*}
\alpha \left( \sum u_{i}U_{i}+v_{i}V_{i}+\sum w_{j}W_{j}\right) =\sum \left(
p_{i}u_{i}+q_{i}v_{i}\right) +\sum \eta _{j}w_{j},
\end{equation*}%
for all $u_{i},v_{i},w_{j}\in \mathbb{R}$, $1\leq i\leq m,1\leq j\leq k$. By 
\cite[Lemma 7, Appendix A]{Pe}, \cite[Theorem 8, Appendix A]{Pe}, 
\begin{equation*}
H\cap \Gamma =\left\{ \prod_{i=1}^{m}\exp \left( t_{i}V_{i}\right)
\prod_{j=1}^{k}\exp \left( s_{j}W_{j}\right) :t_{i},s_{j}\in \mathbb{Z}%
\right\} .
\end{equation*}%
We need to show that there exists $x\in G$ such that $\left( \overline{%
\alpha }\circ I_{x}\right) \left( \gamma \right) =\varepsilon \left( \gamma
\right) $ for all $\gamma \in H\cap \Gamma $. First note that 
\begin{eqnarray*}
\left( \overline{\alpha }\circ I_{x}\right) \left( \exp \left( W_{j}\right)
\right) &=&\overline{\alpha }\left( \exp \left( W_{j}+\left[ \log \left(
x\right) ,W_{j}\right] \right) \right) \\
&=&\exp \left( 2\pi i\alpha \left( W_{j}+\left[ \log \left( x\right) ,W_{j}%
\right] \right) \right) \\
&=&\exp \left( 2\pi i\alpha \left( W_{j}\right) \right) \text{ since }%
W_{j}\in \mathfrak{g}_{\alpha } \\
&=&\overline{\alpha }\left( \exp \left( W_{j}\right) \right) =\varepsilon
\left( \exp \left( W_{j}\right) \right)
\end{eqnarray*}%
since $W_{j}\in \left( \log \Gamma \right) \cap \mathfrak{g}_{\alpha }$.
Next, 
\begin{eqnarray}
\left( \overline{\alpha }\circ I_{x}\right) \left( \exp \left( V_{j}\right)
\right) &=&\overline{\alpha }\left( \exp \left( V_{j}+\left[ \log \left(
x\right) ,V_{j}\right] \right) \right)  \notag \\
&=&\exp \left( 2\pi i\alpha \left( V_{j}+\left[ \log \left( x\right) ,V_{j}%
\right] \right) \right)  \notag \\
&=&\exp \left( 2\pi i\alpha \left( V_{j}+x_{j}\left[ U_{j},V_{j}\right]
\right) \right)  \notag \\
&=&\exp \left( 2\pi i\left( q_{j}+x_{j}r_{j}\right) \right) .
\label{qjPrime}
\end{eqnarray}%
By setting $x_{j}=-\frac{q_{j}}{r_{j}}$ or $-\frac{q_{j}+\frac{1}{2}}{r_{j}}$
depending on whether $\varepsilon \left( \exp \left( V_{j}\right) \right)
=\pm 1$, we conclude. 
\begin{equation*}
\left( \overline{\alpha }\circ I_{x}\right) \left( \exp \left( V_{j}\right)
\right) =\varepsilon \left( \exp \left( V_{j}\right) \right) \text{.}
\end{equation*}%
We have shown that for all $\gamma \in H\cap \Gamma $ there exists $x\in G$
such that $\left( \overline{\alpha }\circ I_{x}\right) \left( \gamma \right)
=\varepsilon \left( \gamma \right) $.

We now calculate the multiplicity with which $\pi _{\alpha }$ appears. In
fact, for $x\in G$, the calculations above show that for all $X\in \mathfrak{%
h,}$ $\left( \overline{\alpha }\circ I_{x}\right) \left( \exp \left(
X\right) \right) $ depends only on the $x_{i}$ and not on $y_{i}$ or $z_{j}$%
, $1\leq i\leq m,1\leq j\leq k$. Thus from (\ref{qjPrime}) the orbit of $%
\left( \overline{\alpha },H\right) $ is the set of characters of $H$ 
\begin{equation*}
\left\{ \left( \chi _{q^{\prime }},H\right) :q^{\prime }\in \mathbb{R}%
^{m}\right\} ,
\end{equation*}%
where where after a bit of calculation identical to \cite[p.453, lines -8
through -5]{Pe} 
\begin{equation*}
\chi _{q^{\prime }}\left( \prod_{i=1}^{m}\exp \left( t_{i}V_{i}\right)
\prod_{j=1}^{k}\exp \left( s_{j}W_{j}\right) \right) =\exp \left( 2\pi
i\left( \sum_{i=1}^{m}q_{i}^{\prime }t_{i}+\sum_{j=1}^{k}\eta
_{j}s_{j}\right) \right) \text{.}
\end{equation*}%
Then $\left( \chi _{q^{\prime }},H\right) $ is an $\varepsilon $-integer
point if and only if%
\begin{eqnarray*}
q_{i}^{\prime } &\in &\mathbb{Z}\text{ whenever }\varepsilon \left( \exp
\left( V_{i}\right) \right) =1, \\
q_{i}^{\prime } &\in &\frac{1}{2}+\mathbb{Z}\text{ whenever }\varepsilon
\left( \exp \left( V_{i}\right) \right) =-1.
\end{eqnarray*}%
Note also from (\ref{qjPrime}) that two $\varepsilon $-integer points $%
\left( \chi _{q^{\prime }},H\right) $ and $\left( \chi _{q^{\prime \prime
}},H\right) $ are in the same $\Gamma $-orbit if and only if $q_{i}^{\prime
}-q_{i}^{\prime \prime }\in r_{i}\mathbb{Z}$, $i=1,...,m$. So the number $%
m_{\alpha }$ of $\Gamma $-orbits in the $\varepsilon $-integer points is $%
r_{1}r_{2}...r_{m}$. Next, it is clear that the images of $%
U_{1},...,U_{m},V_{1},...,V_{m}$ form a basis of $\mathcal{A}_{\alpha }$. So%
\begin{equation*}
\det \left( \overline{B_{\alpha }}\right) =\det \left( B_{\alpha }\left(
U_{i},V_{j}\right) \right) ^{2}=\left( r_{1}r_{2}...r_{m}\right)
^{2}=m_{\alpha }^{2}.
\end{equation*}
\end{proof}

\section{ Decomposition of the Dirac operator on two-step nilmanifolds\label%
{2StepDiracSection}}

\vspace{0in}We continue with the notation of the previous section; recall
that $k_{0}$ is the dimension of the center $\mathfrak{z}$ and $%
n=k_{0}+m_{0} $ is the dimension of $\mathfrak{g}=\mathfrak{z}\oplus 
\mathfrak{v}$, and we will choose the orthonormal basis $\left\{
E_{1},...,E_{n}\right\} =\left\{
Z_{1},...,Z_{k_{1}},...,Z_{k_{0}},X_{1},...,X_{m_{0}}\right\} $ so that $%
\left\{ Z_{j}\right\} _{j=1}^{k_{1}}$ is an orthonormal basis of $\left[ 
\mathfrak{g},\mathfrak{g}\right] $, $\left\{ Z_{j}\right\} _{j=1}^{k_{0}}$
is an orthonormal basis of $\mathfrak{z}$ and $\left\{ X_{j}\right\} $ is an
orthonormal basis of $\mathfrak{v}$. From formula (\ref{Dirac2Step}) and
this choice of basis, the Dirac operator is now%
\begin{equation*}
D=\sum_{i=1}^{n}\left( E_{i}\diamond \right) \rho _{\varepsilon \ast }\left(
E_{i}\right) \mathbf{+}\frac{1}{2}\sum_{a\leq k_{1}}Z_{a}\diamond j\left(
Z_{a}\right) ,
\end{equation*}%
acting on 
\begin{equation}
\mathcal{H}=L^{2}\left( \Gamma \diagdown G,G\times _{\varepsilon }\Sigma
_{n}\right) \cong L_{\varepsilon }^{2}\left( \Gamma \diagdown G\right)
\otimes \Sigma _{n},  \label{isomorphismToTensorProduct}
\end{equation}%
which we decompose using Kirillov theory.

Choose an element $\alpha \in \mathfrak{g}^{\ast }$. Our strategy is as
follows. We first construct a subspace $\mathcal{H}_{\alpha }$ of $%
L^{2}\left( \Gamma \diagdown G,G\times _{\varepsilon }\mathbb{C}^{k}\right) $
that is invariant with respect to $\rho _{\varepsilon }$ and invariant by $D$%
. Once we have done this, by Kirillov theory, let $\overline{\mathcal{H}%
_{\alpha }}$ be the irreducible $\rho _{\varepsilon }$-subspace of $%
L_{\varepsilon }^{2}\left( \Gamma \diagdown G\right) $ corresponding to the
coadjoint orbit of $\alpha $, and let 
\begin{equation*}
\mathcal{H}_{\alpha }\cong \overline{\mathcal{H}_{\alpha }}\otimes \Sigma
_{n}
\end{equation*}%
through the isomorphism above. While $\overline{\mathcal{H}_{\alpha }}$ is $%
\rho _{\varepsilon }$-irreducible, $\mathcal{H}_{\alpha }$ is not for $n\geq
2.$ We express $D$ acting on $\mathcal{H}_{\alpha }$, and because of the
two-step structure, we are able to solve explicitly the partial differential
equation for eigenvalues via Hermite functions.

Since $\sum_{i=1}^{n}\rho _{\varepsilon \ast }\left( E_{i}\right) \left(
E_{i}\diamond \right) $ is independent of the choice of basis $\left\{
E_{1},...,E_{n}\right\} $, the second term is similarly independent of
choices and independent of the representation $\rho _{\varepsilon }$. Define%
\begin{eqnarray}
D_{\rho _{\varepsilon }} &=&\sum_{i=1}^{n}\left( E_{i}\diamond \right) \rho
_{\varepsilon \ast }\left( E_{i}\right)   \notag \\
M &=&\frac{1}{2}\sum_{a\leq k_{0}}Z_{a}\diamond j\left( Z_{a}\right) =\frac{1%
}{2}\sum_{a\leq k_{1}}Z_{a}\diamond j\left( Z_{a}\right) ,
\label{M_definition1}
\end{eqnarray}%
so that $D=D_{\rho _{\varepsilon }}+M$ with $M$ a hermitian linear
transformation independent of invariant subspace. We have used the notation (%
\ref{Dirac2Step}). Note that $\rho _{\varepsilon }$ and $\left( Y\diamond
\right) $ commute if $\left( Y\diamond \right) $ is a constant
transformation --- that is, if $Y$ is left-invariant. Thus $M$ commutes with 
$\rho _{\varepsilon }$ because each $\left\langle Z_{a},\left[ X_{b},X_{i}%
\right] \right\rangle $ is constant on $\Gamma \diagdown G$.

As before, we define the symplectic form on $\mathfrak{g}$ by $B_{\alpha
}\left( U,V\right) :=\alpha \left( \left[ U,V\right] \right) $, and let $%
\mathfrak{g}_{\alpha }=\ker B_{\alpha }=\left\{ U\in \mathfrak{g}:B_{\alpha
}\left( U,\cdot \right) =0\right\} $, $k_{\alpha }=\dim \mathfrak{g}_{\alpha
}$. We have two cases.

\subsection{\textbf{Finite-dimensional }$\overline{\mathcal{H}_{\protect%
\alpha }}$-\textbf{irreducible subspaces:} $k_{\protect\alpha }=n$, i.e. $%
\protect\alpha \left( \left[ \mathfrak{g},\mathfrak{g}\right] \right) =0$.}

In this case, $\mathfrak{g}_{\alpha }=\mathfrak{g}$, and $\mathfrak{g}%
^{\alpha }=\mathfrak{g}$ is a maximal polarizer of $\alpha $. Then $%
G^{\alpha }=\exp \left( \mathfrak{g}^{\alpha }\right) =G.$ Define 
\begin{eqnarray*}
\mathcal{H}_{\alpha } &=&\left\{ \sigma \in \mathcal{H}:\sigma \left(
hx\right) =\overline{\alpha }\left( h\right) \sigma \left( x\right) \text{
for all }h\in G^{\alpha },x\in G\right\} \\
&=&\left\{ \sigma \in \mathcal{H}:\sigma \left( hx\right) =\overline{\alpha }%
\left( h\right) \sigma \left( x\right) \text{ for all }h\in G,x\in G\right\}
\\
&=&\overline{\alpha }\left( \cdot \right) \otimes \Sigma _{n},
\end{eqnarray*}%
where 
\begin{equation*}
\overline{\alpha }\left( h\right) =e^{2\pi i\alpha \left( \log h\right) }.
\end{equation*}%
For $\sigma \in \mathcal{H}_{\alpha }$, we have, since $\alpha \left( \left[ 
\mathfrak{g},\mathfrak{g}\right] \right) =0$, for $p\in \Gamma \diagdown G$, 
\begin{eqnarray*}
\rho _{\varepsilon \ast }\left( U\right) \sigma \left( p\right) &=&\left. 
\frac{d}{dt}\right\vert _{0}\sigma \left( p\exp \left( tU\right) \mathbf{1}%
\right) =\left. \frac{d}{dt}\right\vert _{0}e^{2\pi i\alpha \log \left(
p\exp \left( tU\right) \right) }\sigma \left( \mathbf{1}\right) \\
&=&\left. \frac{d}{dt}\right\vert _{0}e^{2\pi i\alpha \left( \log \left(
p\right) +tU+\frac{1}{2}\left[ \log \left( p\right) ,tU\right] \right)
}\sigma \left( \mathbf{1}\right) \\
&=&\left. \frac{d}{dt}\right\vert _{0}e^{2\pi i\alpha \left( \log \left(
p\right) +tU\right) }\sigma \left( \mathbf{1}\right) .
\end{eqnarray*}%
We have $\rho _{\varepsilon \ast }\left( Z_{a}\right) \sigma =0$, and $\rho
_{\varepsilon \ast }\left( X_{i}\right) \sigma =\frac{\partial }{\partial
x_{i}}\sigma =2\pi i\alpha \left( X_{i}\right) \sigma $. Thus, 
\begin{eqnarray}
\left. D\right\vert _{\mathcal{H}_{\alpha }} &=&\sum_{i=1}^{m_{0}}2\pi
i\alpha \left( X_{i}\right) \left( X_{i}\diamond \right) +\sum_{Z_{j}\not\in %
\left[ \mathfrak{g},\mathfrak{g}\right] }2\pi i\alpha \left( Z_{j}\right)
\left( Z_{j}\diamond \right)  \notag \\
&&\mathbf{+}\frac{1}{2}\sum_{a\leq k_{1}}Z_{a}\diamond j\left( Z_{a}\right) ,
\label{FiniteDimDiracOperatorFormula}
\end{eqnarray}%
which is a constant matrix. The eigenvalues of $\left. D\right\vert _{%
\mathcal{H}_{\alpha }}$ are then the eigenvalues of this Hermitian matrix.

\subsection{\textbf{Infinite-dimensional }$\overline{\mathcal{H}_{ \protect%
\alpha }}$-\textbf{irreducible subspaces:} $k_{\protect\alpha }<n$, so that $%
\protect\alpha \left( \left[ \mathfrak{g},\mathfrak{g}\right] \right) $ is
not identically zero.\label{Case2Section}}

\vspace{1pt}Choose a new orthonormal basis of $\mathfrak{g}$:%
\begin{equation*}
\left\{ W_{1},...,W_{k_{\alpha }},U_{1},...,U_{m},V_{1},...,V_{m}\right\} ,
\end{equation*}%
where $n=k_{\alpha }+2m$, $\left\{ W_{j}\right\} $ is a basis of $\mathfrak{g%
}_{\alpha }$ with $W_{1}=Z_{1},...,W_{k_{0}}=Z_{k_{0}}\in \mathfrak{z}$, $%
W_{k_{0}+1},...,W_{k_{\alpha }}\in \mathfrak{g}_{\alpha }\cap \mathfrak{z}%
^{\bot }$.%
\begin{eqnarray*}
B_{\alpha }\left( U_{i},V_{i}\right) &=&\alpha \left( \left[ U_{i},V_{i}%
\right] \right) =d_{i}>0,0<d_{1}\leq d_{2}\leq \cdots \leq d_{m}, \\
B_{\alpha }\left( U_{i},V_{j}\right) &=&0\text{ if }i\neq j \\
B_{\alpha }\left( U_{i},U_{j}\right) &=&B_{\alpha }\left( V_{i},V_{j}\right)
=0\text{ for all }i,j.
\end{eqnarray*}%
Note the similarity with Lemma \ref{PesceLemmaA4}, but we have replaced some
of the $V_{j}$ with their negatives in order to make $d_{j}$ positive. We
may assume $n-k_{\alpha }$ is even, since the restriction of $B_{\alpha }$
to $\mathfrak{g}_{\alpha }^{\bot }$ is a symplectic form. Then the
polarizing subalgebra $\mathfrak{g}^{\alpha }$ (meaning that $\mathfrak{g}%
^{\alpha }$ is a subalgebra of $\mathfrak{g}$ such that $\alpha \left( \left[
\mathfrak{g}^{\alpha },\mathfrak{g}^{\alpha }\right] \right) =0$ and is
maximal with respect to inclusion) will be chosen to be%
\begin{equation*}
\mathfrak{g}^{\alpha }=\mathrm{span}\left\{
V_{1},...,V_{m},W_{1},...,W_{k_{\alpha }}\right\} ,
\end{equation*}%
and again $G^{\alpha }:=\exp \left( \mathfrak{g}^{\alpha }\right) $. We
have, with $\overline{\alpha }\left( h\right) =\exp \left( 2\pi i\alpha
\left( \log h\right) \right) $,%
\begin{equation*}
\mathcal{H}_{\alpha }=\left\{ \sigma \in \mathcal{H}:\sigma \left( hx\right)
=\overline{\alpha }\left( h\right) \sigma \left( x\right) \text{ for all }%
h\in G^{\alpha },x\in G\right\} .
\end{equation*}%
Let $\overline{\mathcal{H}_{\alpha }}$ be the $\rho _{\varepsilon }$%
-irreducible subspace of $L_{\varepsilon }^{2}\left( \Gamma \diagdown
G\right) $ such that 
\begin{equation*}
\mathcal{H}_{\alpha }\cong \overline{\mathcal{H}_{\alpha }}\otimes \Sigma
_{n}
\end{equation*}%
through the isomorphism (\ref{L2TensorProductIsomorphism}). Let $\beta :%
\overline{\mathcal{H}_{\alpha }}\rightarrow L_{\mathbb{C}}^{2}\left( \mathbb{%
R}^{m}\right) $ be the unitary isomorphism defined by $\beta \left( F\right)
\left( t\right) =F\left( \exp \left( t_{1}U_{1}\right) ...\exp \left(
t_{m}U_{m}\right) \right) $. Note that the map%
\begin{equation*}
t\in \mathbb{R}^{k}\mapsto G^{\alpha }\exp \left( t_{1}U_{1}\right) ...\exp
\left( t_{m}U_{m}\right) \in G^{\alpha }\diagdown G
\end{equation*}%
pushes the Euclidean metric onto a right $G$-invariant metric on $G^{\alpha
}\diagdown G$. Note that $\overline{\mathcal{H}_{\alpha }}=\beta ^{-1}\left(
L_{\mathbb{C}}^{2}\left( \mathbb{R}^{m}\right) \right) $, and for $x=h\exp
\left( t_{1}U_{1}\right) ...\exp \left( t_{m}U_{m}\right) $ an arbitrary
element of $G$ with $h\in G^{\alpha }$, and $f\in $ $L_{\mathbb{C}%
}^{2}\left( \mathbb{R}^{m}\right) $,%
\begin{eqnarray*}
\left( \beta ^{-1}f\right) \left( x\right) &=&\left( \beta ^{-1}f\right)
\left( h\exp \left( t_{1}U_{1}\right) ...\exp \left( t_{m}U_{m}\right)
\right) \\
&=&\overline{\alpha }\left( h\right) f\left( t_{1},...,t_{m}\right) .
\end{eqnarray*}%
Here $\pi _{\alpha }$ is the representation of $G$ on $\overline{\mathcal{H}%
_{\alpha }}$ induced from the character $\overline{\alpha }$ of $G^{\alpha }$%
; we have for $f\in \overline{\mathcal{H}_{\alpha }}$,%
\begin{equation*}
\left( \pi _{\alpha }\left( x\right) f\right) \left( y\right) =\left( \rho
_{\varepsilon x}f\right) \left( y\right) =f\left( yx\right) .
\end{equation*}

We define the representation $\pi _{\alpha }^{\prime }$ of $G$ on $L_{%
\mathbb{C}}^{2}\left( \mathbb{R}^{m}\right) $ by 
\begin{equation*}
\pi _{\alpha }^{\prime }\left( x\right) =\beta \circ \pi _{\alpha }\left(
x\right) \circ \beta ^{-1}
\end{equation*}%
for all $x\in G$.

For any $x,y\in G$, let $[x,y]=xyx^{-1}y^{-1}$. To compute the action of $G$
on $L_{\mathbb{C}}^{2}\left( \mathbb{R}^{m}\right) $, recall that since $G$
is $2$-step (following \cite[p. 447, proof of Prop. 9]{Pe}), for any $%
h_{0}\in G^{\alpha }$,

\begin{itemize}
\item $\prod\limits_{j=1}^{m}\exp \left( t_{j}U_{j}\right) h_{0}=\left[
\prod\limits_{j=1}^{m}\exp \left( t_{j}U_{j}\right) ,h_{0}\right]
h_{0}\prod\limits_{j=1}^{m}\exp \left( t_{j}U_{j}\right) $

\item $\prod\limits_{\ell =1}^{m}\exp \left( t_{\ell }U_{\ell }\right)
\prod\limits_{j=1}^{m}\exp \left( s_{j}U_{j}\right) =\exp \left(
-\sum_{1\leq j<\ell\leq m}t_{\ell }s_{j}\left[ U_{j},U_{\ell }\right]
\right) \prod\limits_{j=1}^{m}\exp \left( \left( t_{j}+s_{j}\right)
U_{j}\right) $

\item $\left[ \prod\limits_{j=1}^{m}\exp \left( t_{j}U_{j}\right) ,h_{0}%
\right] =\exp \left[ \sum_{j=1}^{m}t_{j}U_{j},\log h_{0}\right] .$
\end{itemize}

For any $x\in G$, by the calculations above, there exists $h_{0}\in
G^{\alpha }$ and real numbers $s_{\ell }\in \mathbb{R}$ such that $%
x=h_{0}\prod\limits_{j =1}^{m}\exp \left( s_{j}U_{j}\right) $. For any $f\in
L_{\mathbb{C}}^{2}\left( \mathbb{R}^{m}\right) $, $t,s\in \mathbb{R}^{m}$%
\begin{equation*}
\left( \pi _{\alpha }^{\prime }\left( x\right) f\right) \left( t\right)
=\left( \beta ^{-1}f\right) \left( \prod\limits_{\ell =1}^{m}\exp \left(
t_{\ell }U_{\ell }\right) h_{0}\prod\limits_{j=1}^{m}\exp \left(
s_{j}U_{j}\right) \right) .
\end{equation*}%
Since $\left( \beta ^{-1}f\right) \left( hg\right) =\overline{\alpha }\left(
h\right) \left( \beta ^{-1}f\right) \left( g\right) $, we see%
\begin{equation*}
\left( \pi _{\alpha }^{\prime }\left( x\right) f\right) \left( t\right) =%
\overline{\alpha }\left( \left[ \prod\limits_{\ell =1}^{m}\exp \left(
t_{\ell }U_{\ell }\right) ,h_{0}\right] h_{0}\exp \left( -\sum_{1\leq
j<\ell\leq m}t_{\ell }s_{j}\left[ U_{j},U_{\ell }\right] \right) \right)
f\left( t+s\right) .
\end{equation*}%
We have used the fact that $\exp \left( \left[ \mathfrak{g},\mathfrak{g}%
\right] \right) \subset Z\left( G\right) \subset G^{\alpha }$ and the
calculations above. Since the restriction of $B_{\alpha }$ to $\mathbb{R}%
U_{1}\oplus ...\oplus \mathbb{R}U_{m}\times \mathbb{R}U_{1}\oplus ...\oplus 
\mathbb{R}U_{m}$ is zero, we have%
\begin{equation*}
\left( \pi _{\alpha }^{\prime }\left( x\right) f\right) \left( t\right)
=f\left( t+s\right) e^{2\pi i\alpha \left( \log h_{0}+\left[
\sum_{j=1}^{m}t_{j}U_{j},\log h_{0}\right] \right) }.
\end{equation*}%
Now, define the vector $w\in \mathbb{R}^{m}$ by 
\begin{equation*}
w:=\left( \frac{\alpha \left( V_{1}\right) }{d_{1}},...,\frac{\alpha \left(
V_{m}\right) }{d_{m}}\right) .
\end{equation*}%
Define the unitary isomorphism $T_{w}:L_{\mathbb{C}}^{2}\left( \mathbb{R}%
^{m}\right) \rightarrow L_{\mathbb{C}}^{2}\left( \mathbb{R}^{m}\right) $ by%
\begin{equation*}
\left( T_{w}f\right) \left( t\right) =f\left( t-w\right) ,
\end{equation*}%
and define 
\begin{equation*}
\pi _{\alpha }^{\prime \prime }\left( x\right) =T_{w}\circ \pi _{\alpha
}^{\prime }\left( x\right) \circ T_{w}^{-1}
\end{equation*}%
for all $x\in G$. We claim that the representation $\pi _{\alpha \ast
}^{\prime \prime }=\rho _{\varepsilon \ast }$ is given by%
\begin{eqnarray}
\pi _{\alpha \ast }^{\prime \prime }\left( U_{j}\right) f\left( t\right) &=&%
\frac{\partial }{\partial t_{j}}f\left( t\right) ,  \notag \\
\pi _{\alpha \ast }^{\prime \prime }\left( V_{j}\right) f\left( t\right)
&=&2\pi it_{j}d_{j}f\left( t\right) ,  \notag \\
\pi _{\alpha \ast }^{\prime \prime }\left( W_{j}\right) f\left( t\right)
&=&2\pi i\alpha \left( W_{j}\right) f\left( t\right) .
\label{RLowerStarCalcs}
\end{eqnarray}%
To see this (see also \cite[Section 3]{Pe}), we have for $r\in \mathbb{R}$, 
\begin{eqnarray*}
\pi _{\alpha }^{\prime \prime }(\exp (rU_{j}))f(t) &=&(T_{w}\pi _{\alpha
}^{\prime }(\exp (rU_{j}))T_{-w}f)(t) \\
&=&\left( \pi _{\alpha }^{\prime }(\exp (rU_{j}))T_{-w}f\right) (t-w) \\
&=&(T_{-w}f)(t-w+re_{j}) \\
&=&f(t+re_{j}),
\end{eqnarray*}%
with $e_{j}$ the $j^{\mathrm{th}}$ standard unit vector in $\mathbb{R}^{m}$.
Also,%
\begin{eqnarray*}
\pi _{\alpha }^{\prime \prime }(\exp (rV_{j}))f(t) &=&(\pi _{\alpha
}^{\prime }\left( \exp (rV_{j})\right) T_{-w}f)(t-w) \\
&=&(T_{-w}f)(t-w)e^{2\pi i\alpha (rV_{j}+[(t_{j}-w_{j})U_{j},V_{j}])} \\
&=&f(t)e^{2\pi i(r\alpha (V_{j})+t_{j}d_{j}r-w_{j}d_{j}r)} \\
&=&f(t)e^{2\pi it_{j}d_{j}r}.
\end{eqnarray*}

We have 
\begin{eqnarray*}
\pi _{\alpha }^{\prime \prime }(\exp (rW_{j}))f(t) &=&(\pi _{\alpha
}^{\prime }\left( \exp (rW_{j})\right) T_{-w}f)(t-w) \\
&=&f(t)e^{2\pi i\alpha (rW_{j})}.
\end{eqnarray*}

With $W_{1}=Z_{1},...,W_{k_{0}}=Z_{k_{0}}$, equation (\ref{Dirac2Step})
becomes (see (\ref{M_definition1})) 
\begin{eqnarray*}
\left. D\right\vert _{\mathcal{H}_{\alpha }} &=&\sum_{j=1}^{n}E_{j}\diamond
\rho _{\varepsilon \ast }\left( E_{j}\right) \mathbf{+}\frac{1}{2}%
\sum_{a\leq k_{1}}Z_{a}\diamond j\left( Z_{a}\right) \\
&=&\sum_{j=1}^{n}E_{j}\diamond \rho _{\varepsilon \ast }\left( E_{j}\right)
+M \\
&=&\sum_{j=1}^{k_{\alpha }}\left( W_{j}\diamond \right) \rho _{\varepsilon
\ast }\left( W_{j}\right) +\sum_{j=1}^{m}\left( U_{j}\diamond \right) \rho
_{\varepsilon \ast }\left( U_{j}\right) +\sum_{j=1}^{m}\left( V_{j}\diamond
\right) \rho _{\varepsilon \ast }\left( V_{j}\right) +M \\
&=&\sum_{j=1}^{k_{\alpha }}2\pi i\alpha \left( W_{j}\right) \left(
W_{j}\diamond \right) +\sum_{j=1}^{m}\left( U_{j}\diamond \right) \frac{%
\partial }{\partial t_{j}}+\sum_{j=1}^{m}2\pi it_{j}d_{j}\left(
V_{j}\diamond \right) +M \\
&=&\sum_{j=1}^{m}\left( U_{j}\diamond \right) \frac{\partial }{\partial t_{j}%
}+\sum_{j=1}^{m}2\pi id_{j}\left( V_{j}\diamond \right) t_{j}+M_{\alpha
}^{\prime },
\end{eqnarray*}%
where $M_{\alpha }^{\prime }$ is defined as the constant Hermitian
transformation%
\begin{equation}
M_{\alpha }^{\prime }=M+\sum_{j=1}^{k_{\alpha }}2\pi i\alpha \left(
W_{j}\right) \left( W_{j}\diamond \right) ,  \label{M_alpha_prime_definition}
\end{equation}%
with $M$ as in (\ref{M_definition1}).

We have%
\begin{eqnarray}
\left. D\right\vert _{\mathcal{H}_{\alpha }} &=&M_{\alpha }^{\prime
}+\sum_{j=1}^{m}\left( \left( U_{j}\diamond \right) \frac{\partial }{%
\partial t_{j}}+2\pi id_{j}\left( V_{j}\diamond \right) t_{j}\right)
\label{DOnH_alpha_No_Cjs} \\
&=&M_{\alpha }^{\prime }+\sum_{j=1}^{m}\left( U_{j}\diamond \right) \left( 
\frac{\partial }{\partial t_{j}}-2\pi id_{j}\left( U_{j}\diamond \right)
\left( V_{j}\diamond \right) t_{j}\right) ,  \notag
\end{eqnarray}%
so that%
\begin{equation}
\left. D\right\vert _{\mathcal{H}_{\alpha }}=M_{\alpha }^{\prime
}+\sum_{j=1}^{m}\left( U_{j}\diamond \right) \left( \frac{\partial }{%
\partial t_{j}}+\Theta _{j}t_{j}\right) ,
\label{D_on_H_alpha_infinite_formula}
\end{equation}%
where we define 
\begin{equation}
\Theta _{j}=-2\pi id_{j}\left( U_{j}\diamond \right) \left( V_{j}\diamond
\right) ,  \label{C_j_definition}
\end{equation}%
a Hermitian symmetric linear transformation.

\subsection{Matrix choices\label{matrixChoicesSection}}

We now make specific choices of the matrices $\left( U_{j}\diamond \right) $%
, $\left( V_{j}\diamond \right) $, where $U_{j}$, $V_{r}$, $W_{k}$ are from
the basis chosen at the beginning of Section \ref{Case2Section} relative to
a particular $\alpha $. We continue to use the positive real numbers $d_{j}$
as defined in that section as well. Note that any other choices would yield
the same Dirac spectrum. See \cite[Part I, section 5]{LawM} for details on
the representations of Clifford algebras, on which much of this material is
based.

Let%
\begin{eqnarray*}
\sigma _{1} &=&\left( 
\begin{array}{cc}
0 & 1 \\ 
1 & 0%
\end{array}%
\right) ,~\sigma _{2}=\left( 
\begin{array}{cc}
0 & i \\ 
-i & 0%
\end{array}%
\right) \\
\mathbf{1}^{\prime } &=&\left( 
\begin{array}{cc}
1 & 0 \\ 
0 & -1%
\end{array}%
\right) ,~\mathbf{1}=\left( 
\begin{array}{cc}
1 & 0 \\ 
0 & 1%
\end{array}%
\right)
\end{eqnarray*}%
We view $\Sigma _{n}=\mathbb{C}^{2^{\left\lfloor n/2\right\rfloor
}}=\bigotimes\limits_{\left\lfloor n/2\right\rfloor \text{ times}}\mathbb{C}%
^{2}$. Observe that multiplication satisfies

\begin{equation}
\begin{tabular}{|c|c|c|c|}
\hline
$\mathbf{1}$ & $\mathbf{1}^{\prime }$ & $\sigma _{1}$ & $\sigma _{2}$ \\ 
\hline
$\mathbf{1}^{\prime }$ & $\mathbf{1}$ & $-i\sigma _{2}$ & $i\sigma _{1}$ \\ 
\hline
$\sigma _{1}$ & $i\sigma _{2}$ & $\mathbf{1}$ & $-i\mathbf{1}^{\prime }$ \\ 
\hline
$\sigma _{2}$ & $-i\sigma _{1}$ & $i\mathbf{1}^{\prime }$ & $\mathbf{1}$ \\ 
\hline
\end{tabular}
\label{CayleyTable}
\end{equation}%
with multiplication on the left given by the column items.

Let 
\begin{equation*}
\left( U_{1}\diamond \right) =i\sigma _{1}\otimes \mathbf{1\otimes
...\otimes 1},~\left( V_{1}\diamond \right) =i\sigma _{2}\otimes \mathbf{%
1\otimes ...\otimes 1},
\end{equation*}%
and in general, for $1\leq j\leq m$,

\begin{eqnarray}
\left( U_{j}\diamond \right) &=&i\mathbf{1}^{\prime }\otimes ...\mathbf{%
\otimes 1}^{\prime }\otimes \sigma _{1}\otimes \mathbf{1\otimes ...\otimes 1}%
,  \notag \\
\left( V_{j}\diamond \right) &=&i\mathbf{1}^{\prime }\otimes ...\mathbf{%
\otimes 1}^{\prime }\otimes \sigma _{2}\otimes \mathbf{1\otimes ...\otimes 1,%
}  \label{UjVjDiamond}
\end{eqnarray}%
where each $\left( U_{j}\diamond \right) $ and each $\left( V_{j}\diamond
\right) $ has $j-1$ leading factors of $\mathbf{1}^{\prime }$ and a total of 
$n^{\prime }=\left\lfloor \frac{n}{2}\right\rfloor =m+\left\lfloor \frac{%
k_{\alpha }}{2}\right\rfloor $ matrix factors of size $2\times 2$ .
Continuing, each $\left( W_{k}\diamond \right) $, $1\leq k\leq k_{\alpha }$,
is chosen to be%
\begin{equation}
\left( W_{k}\diamond \right) =i\mathbf{1}^{\prime }\otimes ...\mathbf{%
\otimes 1}^{\prime }\otimes \sigma \otimes \mathbf{1\otimes ...\otimes 1,}
\label{WkDiamond}
\end{equation}%
with $\sigma $ being $\sigma _{1}$ or $\sigma _{2}$ according to whether $k$
is odd or even, such that there are at least $m$ leading factors of $\mathbf{%
1}^{\prime }$ in the above expression. If the dimension $k_{\alpha }$ is
odd, then the last matrix is chosen to be 
\begin{equation}
\left( W_{k_{\alpha }}\diamond \right) =i\mathbf{1}^{\prime }\otimes ...%
\mathbf{\otimes 1}^{\prime }.  \label{WkalphaDiamond}
\end{equation}

With these choices, observe that from (\ref{C_j_definition}), 
\begin{equation*}
\Theta _{j}=2\pi d_{j}\left( \mathbf{1}\otimes \mathbf{1}\otimes ...\mathbf{%
\otimes 1}\otimes \mathbf{1}^{\prime }\otimes \mathbf{1\otimes ...\otimes 1}%
\right) ,
\end{equation*}%
with $\mathbf{1}^{\prime }$ in the $j^{\mathrm{th}}$ slot. We let%
\begin{equation}
v_{\ell }=e_{\ell _{1}}\otimes e_{\ell _{2}}\otimes ...\otimes e_{\ell
_{n^{\prime }}}  \label{v_ell_definition}
\end{equation}%
where each $e_{\ell _{\bullet }}$ is either $e_{1}=\left( 
\begin{array}{c}
1 \\ 
0%
\end{array}%
\right) $, $e_{-1}=\left( 
\begin{array}{c}
0 \\ 
1%
\end{array}%
\right) $, with $\ell =\left( \ell _{1},...,\ell _{n^{\prime }}\right) \in
\left\{ 1,-1\right\} ^{n^{\prime }}$, then $\left\{ v_{\ell }\right\} $
forms a basis of $\Sigma _{n}$. Then%
\begin{eqnarray*}
\Theta _{j}v_{\ell } &=&2\pi d_{j}\ell _{j}v_{\ell }, \\
\left( U_{j}\diamond \right) v_{\ell } &=&i\ell _{1}\ell _{2}...\ell
_{j-1}v_{\ell ^{j}}=\pm iv_{\ell ^{j}}~,
\end{eqnarray*}%
with $\ell ^{j}=\left( \ell _{1},...,-\ell _{j},..,\ell _{n^{\prime
}}\right) $, and where we mean $\ell _{1}\ell _{2}...\ell _{j-1}=1$ when $%
j=1 $.

We see that $\Theta _{j}$ commutes with every $\Theta _{j^{\prime }}$, and $%
\Theta _{j}^{2}=4\pi ^{2}d_{j}^{2}\mathbf{Id}$. Note that \linebreak $%
\left\{ \left( 2\pi d_{j}\ell _{j},v_{\ell }\right) :\ell \in \left\{
1,-1\right\} ^{n^{\prime }}\right\} $ is the set of eigenvalues and
simultaneous orthonormal eigenvectors of every $\Theta _{j}$, $j=1,...,m$. 
\vspace{1pt}

Let $\mathbf{p}=\left( p_{1},...,p_{m}\right) \in \mathbb{Z}^{m}$. We let $%
h_{\mathbf{p}}\left( t\right) =h_{p_{1}}\left( t_{1}\right)
...h_{p_{m}}\left( t_{m}\right) $ using the Hermite functions 
\begin{eqnarray}
h_{p}\left( t\right) &=&e^{t^{2}/2}\left( \frac{d}{dt}\right) ^{p}e^{-t^{2}}%
\text{ for}~p\geq 0,  \label{HermiteFunctions} \\
h_{p}\left( t\right) &=&0\text{ for~}p<0,  \notag
\end{eqnarray}%
which satisfy

\begin{gather*}
h_{p}^{\prime }\left( t\right) =th_{p}\left( t\right) +h_{p+1}\left( t\right)
\\
h_{p+2}\left( t\right) +2th_{p+1}\left( t\right) +2\left( p+1\right)
h_{p}\left( t\right) =0.
\end{gather*}%
The first equality is just the chain rule. To see the second equality, note
that by the product rule and the binomial theorem (or by induction), 
\begin{equation*}
\left( \frac{d}{dt}\right) ^{p+2}e^{-t^{2}}=\left( \frac{d}{dt}\right)
^{p+1}\left( -2te^{-t^{2}}\right) =-2t\left( \frac{d}{dt}\right)
^{p+1}e^{-t^{2}}-2(p+1)\left( \frac{d}{dt}\right) ^{p}e^{-t^{2}},
\end{equation*}%
and the result follows. Combining the two 
\begin{equation*}
h_{p}^{\prime }\left( t\right) =th_{p}\left( t\right) -2th_{p}\left(
t\right) -2ph_{p-1}\left( t\right) =-th_{p}\left( t\right) -2ph_{p-1}\left(
t\right) .
\end{equation*}%
Note that $\left\{ h_{\mathbf{p}}\left( t\right) :\mathbf{p}\in \left( 
\mathbb{Z}_{\geq 0}\right) ^{m}\right\} $ is a basis of $L^{2}\left( \mathbb{%
R}^{m},\mathbb{C}\right) $.

For $\mathbf{p}\in \mathbb{Z}^{m}$, let 
\begin{equation}
u_{\mathbf{p},\ell }\left( t\right) =h_{p_{1}}\left( \sqrt{2\pi d_{1}}%
t_{1}\right) h_{p_{2}}\left( \sqrt{2\pi d_{2}}t_{2}\right)
...h_{p_{m}}\left( \sqrt{2\pi d_{m}}t_{m}\right) v_{\ell },
\label{u_p_ell_formula}
\end{equation}%
with $v_{\ell }$ as in (\ref{v_ell_definition}). Observe that $u_{\mathbf{p},%
\mathbf{\ell }}=0$ if any coordinate of $\mathbf{p}$ is negative. In what
follows, we assume the coordinates of $\mathbf{p}$ are nonnegative. Then,
using the formulas above for $h_{p}^{\prime }\left( t\right) $,%
\begin{eqnarray*}
\frac{\partial }{\partial t_{j}}u_{\mathbf{p},\ell }\left( t\right) &=&-2\pi
d_{j}t_{j}u_{\mathbf{p},\ell }\left( t\right) -2p_{j}\sqrt{2\pi d_{j}}u_{%
\mathbf{p}-e_{j},\ell }\left( t\right) \\
&=&2\pi d_{j}t_{j}u_{\mathbf{p},\ell }\left( t\right) +\sqrt{2\pi d_{j}}u_{%
\mathbf{p}+e_{j},\ell }\left( t\right)
\end{eqnarray*}%
\begin{eqnarray*}
\left( \frac{\partial }{\partial t_{j}}+t_{j}\Theta _{j}\right) u_{\mathbf{p}%
,\ell }\left( t\right) &=&\left( 2\pi d_{j}(\ell _{j}-1)\right) t_{j}u_{%
\mathbf{p},\ell }\left( t\right) -2p_{j}\sqrt{2\pi d_{j}}u_{\mathbf{p}%
-e_{j},\ell }\left( t\right) \\
&=&\left( 2\pi d_{j}(\ell _{j}+1)\right) t_{j}u_{\mathbf{p},\ell }\left(
t\right) +\sqrt{2\pi d_{j}}u_{\mathbf{p}+e_{j},\ell }\left( t\right)
\end{eqnarray*}%
Recall that $\mathbf{p}$ has dimension $m$, and $\mathbf{\ell }$ has
dimension $n^{\prime }=\left\lfloor \frac{n}{2}\right\rfloor \geq m$. Now,
from (\ref{D_on_H_alpha_infinite_formula}) we have 
\begin{eqnarray}
Du_{\mathbf{p},\ell }\left( t\right) &=&\left( M_{\alpha }^{\prime
}+\sum_{j=1}^{m}\left( U_{j}\diamond \right) \left( \frac{\partial }{%
\partial t_{j}}+t_{j}\Theta _{j}\right) \right) u_{\mathbf{p},\ell }\left(
t\right)  \notag \\
&=&-2\sum_{j\leq m,\ell _{j}=1}p_{j}\sqrt{2\pi d_{j}}\left( U_{j}\diamond
\right) u_{\mathbf{p}-e_{j},\ell }\left( t\right)  \notag \\
&&+\sum_{j\leq m,\ell _{j}=-1}\sqrt{2\pi d_{j}}\left( U_{j}\diamond \right)
u_{\mathbf{p}+e_{j},\ell }\left( t\right) +M_{\alpha }^{\prime }u_{\mathbf{p}%
,\ell }\left( t\right)  \notag \\
&=&-2\sum_{j\leq m,\ell _{j}=1}ip_{j}\sqrt{2\pi d_{j}}\ell _{1}\ell
_{2}...\ell _{j-1}u_{\mathbf{p}-e_{j},\ell ^{j}}\left( t\right)  \notag \\
&&+\sum_{j\leq m,\ell _{j}=-1}i\sqrt{2\pi d_{j}}\ell _{1}\ell _{2}...\ell
_{j-1}u_{\mathbf{p}+e_{j},\ell ^{j}}\left( t\right) +M_{\alpha }^{\prime }u_{%
\mathbf{p},\ell }\left( t\right) .  \label{DequationWithSQRoots}
\end{eqnarray}%
Often the eigensections can be found as linear combinations of the $u_{%
\mathbf{p},\ell }\left( t\right) $.

We modify the basis so that it is more convenient. For fixed $\mathbf{p}%
=\left( p_{1},...,p_{m}\right) $ and $\mathbf{\ell }=\left( \ell
_{1},...,\ell _{m},...,\ell _{n^{\prime }}\right) $, let $\mathbf{E}_{%
\mathbf{\ell }}$ be the $m$-tuple defined by 
\begin{equation*}
\left( \mathbf{E}_{\mathbf{\ell }}\right) _{a}=\left\{ 
\begin{array}{ll}
0 & \text{if }\ell _{a}=1 \\ 
-1~ & \text{if }\ell _{a}=-1%
\end{array}%
\right. .
\end{equation*}%
Then%
\begin{equation}
\overline{u}_{\mathbf{p},\mathbf{\ell }}\left( t\right) =\left( \prod_{j\leq
m,~\ell _{j}=-1}\sqrt{2p_{j}}\right) u_{\mathbf{p}+\mathbf{E}_{\mathbf{\ell }%
},\mathbf{\ell }}\left( t\right) .  \label{ubar_p_ell}
\end{equation}

Using the fact that $\mathbf{p}+\mathbf{E}_{\ell }+e_{j}=\mathbf{p}+\mathbf{E%
}_{\ell ^{j}}$ if $\ell _{j}=-1$ and $\mathbf{p}+\mathbf{E}_{\ell }-e_{j}=%
\mathbf{p}+\mathbf{E}_{\ell ^{j}}$ if $\ell _{j}=1$, we compute%
\begin{eqnarray}
D\overline{u}_{\mathbf{p},\mathbf{\ell }}\left( t\right) &=&-\sum_{j\leq
m,\ell _{j}=1}2i\sqrt{\pi d_{j}p_{j}}\ell _{1}\ell _{2}...\ell _{j-1}%
\overline{u}_{\mathbf{p},\mathbf{\ell }^{j}}\left( t\right)  \notag \\
&&+\sum_{j\leq m,\ell _{j}=-1}2i\sqrt{\pi d_{j}p_{j}}\ell _{1}\ell
_{2}...\ell _{j-1}\overline{u}_{\mathbf{p},\mathbf{\ell }^{j}}\left(
t\right) +M_{\alpha }^{\prime }\overline{u}_{\mathbf{p},\mathbf{\ell }%
}\left( t\right) ,  \notag
\end{eqnarray}%
\qquad so that

\begin{equation}
D\overline{u}_{\mathbf{p},\mathbf{\ell }}\left( t\right) =-2i\sum_{j\leq m}%
\sqrt{\pi d_{j}p_{j}}\ell _{1}\ell _{2}...\ell _{j}\overline{u}_{\mathbf{p},%
\mathbf{\ell }^{j}}\left( t\right) +M_{\alpha }^{\prime }\overline{u}_{%
\mathbf{p},\mathbf{\ell }}\left( t\right) .  \label{D_eqn_new_basis}
\end{equation}

\section{Heisenberg Examples\label{HeisenbergExamplesSection}}

Heisenberg Lie algebras are the only two-step nilpotent Lie algebras with
one-dimensional center. Let $n=2m+1$; define the $n$-dimensional Heisenberg
Lie algebra by \linebreak $\mathfrak{g}=\mathrm{span}\left\{
X_{1},...,X_{m},Y_{1},...,Y_{m},Z\right\} $ with $\left[ X_{j},Y_{k}\right]
=\delta _{jk}Z$ and other basis brackets not defined by skew-symmetry equal
to zero. The $n$-dimensional Heisenberg Lie group $G$ is the simply
connected Lie group with Lie algebra $\mathfrak{g}$. A Heisenberg manifold
is a quotient of $G$ by a cocompact discrete subgroup $\Gamma $, where the
metric comes from a left-invariant metric on $G$. From \cite[Proposition 2.16%
]{Gord}, we see that every Heisenberg manifold is isometric to one with the
following metric and lattice. The metric may be chosen for $\Gamma \diagdown
G$ on $\left( X_{1},...,X_{m},Y_{1},...,Y_{m},Z\right) $ to be 
\begin{equation*}
g_{A}=\left( 
\begin{array}{ccc}
A & 0 & 0 \\ 
0 & A & 0 \\ 
0 & 0 & 1%
\end{array}%
\right) =\left( 
\begin{array}{cc}
\overline{g_{A}} & 0 \\ 
0 & 1%
\end{array}%
\right)
\end{equation*}%
where $A=\mathrm{diag}\left( a_{1},...,a_{m}\right) $ is a diagonal $m\times
m$ matrix with positive nondecreasing entries.

We identify $X_{i}$ with the matrix $E_{1,i+1}$, which is the matrix with $1$
in the $\left( 1,i+1\right) $-entry and all other entries zero. Similarly,
we identify $Y_{j}$ with $E_{j+1,m+2}$ and $Z$ with $E_{m+2,m+2}$. In this
section, we define $\exp \left( X_{i}\right) $ to be the matrix exponential $%
\exp \left( E_{1,i+1}\right) =I+E_{1,i+1}$, and we define $\exp \left(
Y_{j}\right) $ and $\exp \left( Z\right) $ in a similar way. For $v\in 
\mathbb{R}^{2m}$ and $z\in \mathbb{R}$, we denote 
\begin{equation*}
\left( v,z\right) =\left( 
\begin{array}{ccccc}
1 & v_{1} & ... & v_{m} & z \\ 
0 & 1 & ... & 0 & v_{m+1} \\ 
\vdots & \vdots & I & \vdots & \vdots \\ 
0 & 0 & ... & 1 & v_{2m} \\ 
0 & 0 & ... & 0 & 1%
\end{array}%
\right) ,
\end{equation*}
With this notation,%
\begin{multline}
\exp \left( x_{1}X_{1}+...+x_{m}X_{m}+y_{1}Y_{1}+...+y_{m}Y_{m}+zZ\right)
=\left( x_{1},...,x_{m},y_{1},...,y_{m},z+\frac{1}{2}x\cdot y\right) , \\
\log \left( x_{1},...,x_{m},y_{1},...,y_{m},z\right)
=x_{1}X_{1}+...+x_{m}X_{m}+y_{1}Y_{1}+...+y_{m}Y_{m}+\left( z-\frac{1}{2}%
x\cdot y\right) Z.  \label{explogformulas}
\end{multline}%
To get from the matrix coordinates to the exponential coordinates, we use
the change of coordinate mapping

\begin{equation*}
(v,z)\mapsto \exp (v_{1}X_{1}+\ldots
+v_{m}X_{m}+v_{m+1}Y_{1}+..+v_{2m}Y_{m}+(z-\frac{1}{2}(v_{1}v_{m+1}+\ldots
+v_{m}v_{2m}))Z).
\end{equation*}

Every cocompact discrete subgroup $\Gamma $ can be generated by $\exp \left( 
\mathcal{L}\right) $ and $\exp \left( rZ\right) $, where $\mathcal{L}$ is a $%
2m$-dimensional lattice in $\mathbb{R}^{2m}=\mathrm{span}\left\{
X_{1},...,X_{m},Y_{1},...,Y_{m}\right\} $, and $\exp \left( rZ\right) $, $%
r>0 $, generates a one-dimensional lattice in the center of $G$. We denote $%
\Gamma =\Gamma \left( \mathcal{L},r\right) $; note $\left( \mathcal{L}%
,r\right) $ will yield a cocompact discrete subgroup if and only if for all $%
V,V^{\prime }\in \mathcal{L}$, $\left[ \exp \left( V\right) ,\exp \left(
V^{\prime }\right) \right] =\exp \left( krZ\right) $ for some $k\in \mathbb{Z%
}$ \cite[proof of Theorem 2.4]{GW1}. Two such Heisenberg manifolds
determined by $\left( \mathcal{L},r,g_{A}\right) $ and $\left( \mathcal{L}%
^{\prime },r^{\prime },g_{A^{\prime }}\right) $ are isometric iff $%
g_{A}=g_{A^{\prime }}$, $r=r^{\prime }$, and there exists a matrix $\Phi \in 
\widetilde{Sp}\left( m,\mathbb{R}\right) \cap O\left( 2m,\overline{g_{A}}%
\right) \subset M_{2m}\left( \mathbb{R}\right) $ such that%
\begin{equation*}
\Phi \left( \mathcal{L}\right) =\mathcal{L}^{\prime }.
\end{equation*}%
(See \cite[Proposition 2.16]{Gord}). Here, $O\left( 2m,\overline{g_{A}}%
\right) $ is the orthogonal group, and $\widetilde{Sp}\left( m,\mathbb{R}%
\right) =\left\{ \beta \in GL\left( 2m,\mathbb{R}\right) :\beta ^{t}J\beta
=\pm J\right\} ,$ where $J=\left( 
\begin{array}{cc}
0 & I \\ 
-I & 0%
\end{array}%
\right) $.

\subsection{Three-dimensional case\label{3DHeisenbergExampleSection}}

\subsubsection{Eigenvalues}

For our Heisenberg manifold, we choose $\left\{ X,Y,Z\right\} $ so that $%
\left[ X,Y\right] =Z$ and $\left\{ \frac{1}{\sqrt{A}}X,\frac{1}{\sqrt{A}}%
Y,Z\right\} $ is an orthonormal frame, with $A>0$. With notation as in the
general case, we choose an element $\alpha \in \mathfrak{g}^{\ast }$, which
fixes a coadjoint orbit.

\textbf{Finite-dimensional irreducible subspaces: }If the one-form $\alpha
\left( Z\right) =0$, then $\mathfrak{g}_{\alpha }=\mathfrak{g}$, and $%
\mathfrak{g}^{\alpha }=\mathfrak{g}$ is the maximal polarizer of $\alpha $.
Then $G^{\alpha }=\exp \left( \mathfrak{g}^{\alpha }\right) =G$. Let 
\begin{equation*}
\mathcal{H}_{\alpha }=\left\{ f:\mathfrak{g}\rightarrow \Sigma _{n}~|~\text{%
for some }s\in \Sigma _{n},\text{all }h\in G,f\left( h\right) =\overline{%
\alpha }\left( h\right) s\text{ }\right\} ,
\end{equation*}%
where 
\begin{equation*}
\overline{\alpha }\left( h\right) =e^{2\pi i\alpha \left( \log h\right) }.
\end{equation*}%
From (\ref{FiniteDimDiracOperatorFormula}), 
\begin{eqnarray*}
\left. D\right\vert _{\mathcal{H}_{\alpha }} &=&\frac{2\pi i}{A}\alpha
\left( X\right) \left( X\diamond \right) +\frac{2\pi i}{A}\alpha \left(
Y\right) \left( Y\diamond \right) \\
&&\mathbf{+}\frac{1}{4A^{2}}\left\langle Z,\left[ X,Y\right] \right\rangle
\left( Z\diamond X\diamond Y\diamond \right) \\
&=&\frac{2\pi i}{A}\alpha \left( X\right) \left( X\diamond \right) +\frac{%
2\pi i}{A}\alpha \left( Y\right) \left( Y\diamond \right) \mathbf{+}\frac{1}{%
4A^{2}}\left( Z\diamond X\diamond Y\diamond \right) ,
\end{eqnarray*}%
which is a constant matrix. The eigenvalues of $\left. D\right\vert _{%
\mathcal{H}_{\alpha }}$ are then the eigenvalues of this Hermitian matrix.
We set%
\begin{multline*}
\left( X\diamond \right) =i\sqrt{A}\sigma _{1}=\left( 
\begin{array}{cc}
0 & i\sqrt{A} \\ 
i\sqrt{A} & 0%
\end{array}%
\right) ,~~\left( Y\diamond \right) =i\sqrt{A}\sigma _{2}=\left( 
\begin{array}{cc}
0 & -\sqrt{A} \\ 
\sqrt{A} & 0%
\end{array}%
\right) ,~~ \\
\left( Z\diamond \right) =i\mathbf{1}^{\prime }=\left( 
\begin{array}{cc}
i & 0 \\ 
0 & -i%
\end{array}%
\right) ,
\end{multline*}%
\vspace{1pt}

The matrix is%
\begin{equation*}
\left. D\right\vert _{\mathcal{H}_{\alpha }}=\left( 
\begin{array}{cc}
-\frac{1}{4A} & -\frac{2\pi }{\sqrt{A}}\left( \alpha \left( X\right)
+i\alpha \left( Y\right) \right) \\ 
-\frac{2\pi }{\sqrt{A}}\left( \alpha \left( X\right) -i\alpha \left(
Y\right) \right) & -\frac{1}{4A}%
\end{array}%
\right) .
\end{equation*}

The eigenvalues are%
\begin{equation}
\sigma _{\alpha }=\left\{ -\frac{1}{4A}+\frac{2\pi }{\sqrt{A}}\left\Vert
\alpha \right\Vert ,-\frac{1}{4A}-\frac{2\pi }{\sqrt{A}}\left\Vert \alpha
\right\Vert \right\} .  \label{torusSpectrum3dHeisenberg}
\end{equation}

\textbf{Infinite-dimensional irreducible subspaces:} On the other hand,
suppose $\alpha \left( Z\right) =\alpha \left( \left[ X,Y\right] \right) =d$
is nonzero.

From (\ref{D_eqn_new_basis}), (\ref{M_alpha_prime_definition}), and (\ref%
{v_ell_definition}), with $U=\frac{X}{\sqrt{A}}$, $V=\mathrm{sgn}\left(
d\right) \frac{Y}{\sqrt{A}}$, $W=\mathrm{sgn}\left( d\right) Z$, and with%
\begin{multline*}
\left( X\diamond \right) =i\sqrt{A}\sigma _{1}=\left( 
\begin{array}{cc}
0 & i\sqrt{A} \\ 
i\sqrt{A} & 0%
\end{array}%
\right) ,~~\left( \mathrm{sgn}\left( d\right) Y\diamond \right) =i\sqrt{A}%
\sigma _{2}=\left( 
\begin{array}{cc}
0 & -\sqrt{A} \\ 
\sqrt{A} & 0%
\end{array}%
\right) ,~~ \\
\left( W\diamond \right) =i\mathbf{1}^{\prime }=\left( 
\begin{array}{cc}
i & 0 \\ 
0 & -i%
\end{array}%
\right) ,
\end{multline*}%
we have, since $d_{1}=\frac{\left\vert d\right\vert }{A},m=1,\ell =\pm
1,E_{\ell }=\frac{\ell -1}{2}$,%
\begin{eqnarray*}
u_{p,\ell }\left( t\right) &=&h_{p}\left( \sqrt{2\pi d_{1}}t\right) v_{\ell
}, \\
\overline{u}_{p,\ell }\left( t\right) &=&\left\{ 
\begin{array}{cc}
u_{p,\ell }\left( t\right) & \ell =1 \\ 
\sqrt{2p}u_{p-1,\ell }\left( t\right) & \ell =-1%
\end{array}%
\right. ,
\end{eqnarray*}%
using (\ref{D_eqn_new_basis}). Note that $\overline{u}_{0,-1}=0$ and is not
included in the basis of sections. The invariant subspaces are $\mathrm{span}%
\left\{ \overline{u}_{0,1}\right\} $ and $\mathrm{span}\left\{ \overline{u}%
_{p,1},\overline{u}_{p,-1}\right\} $ for $p>0$: 
\begin{eqnarray*}
D\overline{u}_{0,1}\left( t\right) &=&M_{\alpha }^{\prime }\overline{u}%
_{0,1}\left( t\right) \\
D\overline{u}_{p,\ell }\left( t\right) &=&-2i\ell \sqrt{\pi \frac{\left\vert
d\right\vert }{A}p~}\overline{u}_{_{p,-\ell }}\left( t\right) +M_{\alpha
}^{\prime }\overline{u}_{_{p,\ell }}\left( t\right) \text{ for }p>0.
\end{eqnarray*}

From (\ref{M_alpha_prime_definition}) we have%
\begin{equation*}
M=\frac{1}{2}W\diamond j\left( W\right) ,
\end{equation*}

\begin{eqnarray*}
M_{\alpha }^{\prime } &=&M+2\pi i\alpha \left( W\right) \left( W\diamond
\right) \\
&=&\frac{1}{4A^{2}}\left\langle \mathrm{sgn}(d)Z,\left[ X,\mathrm{sgn}(d)Y%
\right] \right\rangle \left( W\diamond X\diamond \mathrm{sgn}(d)Y\diamond
\right) +2\pi i\alpha \left( W\right) \left( W\diamond \right) \\
&=&\frac{1}{4A^{2}}\left( 
\begin{array}{cc}
i & 0 \\ 
0 & -i%
\end{array}%
\right) \left( 
\begin{array}{cc}
0 & i\sqrt{A} \\ 
i\sqrt{A} & 0%
\end{array}%
\right) \left( 
\begin{array}{cc}
0 & -\sqrt{A} \\ 
\sqrt{A} & 0%
\end{array}%
\right) +2\pi i\left\vert d\right\vert \left( 
\begin{array}{cc}
i & 0 \\ 
0 & -i%
\end{array}%
\right) \\
&=&\left( 
\begin{array}{cc}
-2\pi \left\vert d\right\vert -\frac{1}{4A} & 0 \\ 
0 & 2\pi \left\vert d\right\vert -\frac{1}{4A}%
\end{array}%
\right) ,
\end{eqnarray*}%
so that%
\begin{equation*}
M_{\alpha }^{\prime }v_{\ell }=\left( -2\pi \left\vert d\right\vert \ell -%
\frac{1}{4A}\right) v_{\ell }.
\end{equation*}%
This implies that $M_{\alpha }^{\prime }\overline{u}_{_{p,\ell }}\in \mathrm{%
span}\left\{ \overline{u}_{_{p,\ell }}\right\} $ for all $\ell \in \left\{
-1,1\right\} $, $p\in \mathbb{Z}_{\geq 0}$. For $p>0$, 
\begin{equation*}
D\overline{u}_{p,\ell }\left( t\right) =-2i\ell \sqrt{\pi \frac{\left\vert
d\right\vert }{A}p~}\overline{u}_{_{p,-\ell }}\left( t\right) +\left( -2\pi
\left\vert d\right\vert \ell -\frac{1}{4A}\right) \overline{u}_{_{p,\ell
}}\left( t\right) .
\end{equation*}%
The $p=0,\ell =1$ case ($\overline{u}_{0,-1}=u_{-1,-1}=0$) is%
\begin{equation*}
D\overline{u}_{0,1}\left( t\right) =\left( -2\pi \left\vert d\right\vert -%
\frac{1}{4A}\right) \overline{u}_{0,1}\left( t\right) .
\end{equation*}%
The matrix for $D$ restricted to the span of $\left\{ \overline{u}_{p,1},%
\overline{u}_{p,-1}\right\} $ for $p\geq 1$ is%
\begin{equation*}
\left( 
\begin{array}{cc}
-2\pi \left\vert d\right\vert -\frac{1}{4A} & 2i\sqrt{\pi \frac{\left\vert
d\right\vert }{A}p~} \\ 
-2i\sqrt{\pi \frac{\left\vert d\right\vert }{A}p~} & 2\pi \left\vert
d\right\vert -\frac{1}{4A}%
\end{array}%
\right) ,
\end{equation*}%
which has eigenvalues 
\begin{equation*}
-\frac{1}{4A}\pm 2\sqrt{\frac{\pi \left\vert d\right\vert p}{A}+\pi
^{2}d^{2}~}.
\end{equation*}

Thus, the list of all eigenvalues for the $\alpha \left( Z\right) =d\neq 0$
case is%
\begin{equation*}
\sigma _{\alpha }=\left\{ -\frac{1}{4A}-2\pi \left\vert d\right\vert
\right\} \cup \left\{ -\frac{1}{4A}\pm 2\sqrt{\frac{\pi \left\vert
d\right\vert p}{A}+\pi ^{2}d^{2}~}:p\geq 1\right\} .
\end{equation*}

\subsubsection{Occurrence conditions for lattice}

Here, the lattice $\mathcal{L}$ should be a two-dimensional lattice, say
spanned by $v=\left( v_{1},v_{2}\right) $ (corresponding to the matrix
element $\left( v_{1},v_{2},0\right) $ ) and $w=\left( w_{1},w_{2}\right) $.
The central lattice is spanned by $r$ (corresponding to $\left( 0,0,r\right) 
$ ). Let $\widetilde{Sp}\left( 1,\mathbb{R}\right) =\left\{ \beta \in
GL\left( 2,\mathbb{R}\right) :\beta ^{t}J\beta =\pm J\right\} $. The
condition $\beta ^{t}J\beta =\pm J$ is equivalent to $\det \beta =\pm 1$, so
in fact $\widetilde{Sp}\left( 1,\mathbb{R}\right) \cap O\left( 2,\mathbb{R}%
\right) =O\left( 2,\mathbb{R}\right) .$ This means we can rotate so that $%
v=\left( v_{1},0\right) $ with $v_{1}>0$, and so that $w=\left(
w_{1},w_{2}\right) $ with $w_{2}>0$. Because $v,w\cdot ,$ and $r$ generate a
cocompact discrete subgroup, we must have, for any $%
h_{1},h_{2},h,k_{1},k_{2},k\in \mathbb{Z}$, 
\begin{multline*}
\left( h_{1}v+h_{2}w,hr\right) \left( k_{1}v+k_{2}w,kr\right) \\
=\left( \left( h_{1}+k_{1}\right) v+\left( h_{2}+k_{2}\right) w,r\left(
h+k\right) +h_{1}k_{2}v_{1}w_{2}+h_{2}k_{2}w_{1}w_{2}\right) .
\end{multline*}%
is an element of the lattice, by closure for multiplication. Thus, for any
choice of integers $h_{1},h_{2},k_{1},k_{2},$ we must have $%
h_{1}k_{2}v_{1}w_{2}+h_{2}k_{2}w_{1}w_{2}\in r\mathbb{Z}$, i.e. $%
v_{1}w_{2}~,w_{1}w_{2}\in r\mathbb{Z}$. Letting $v_{1}w_{2}=rm_{v}$, $%
w_{1}w_{2}=rm_{w}$ , we have $v=\left( \frac{rm_{v}}{w_{2}},0\right) $, $%
w=\left( \frac{rm_{w}}{w_{2}},w_{2}\right) $. The parameters are%
\begin{equation}
A>0,r>0,w_{2}>0,m_{v}\in \mathbb{Z}_{>0},m_{w}\in \mathbb{Z}.
\label{conditions_on_parameters}
\end{equation}

In our matrix coordinate system, from the beginning of Section \ref%
{HeisenbergExamplesSection}, we have 
\begin{multline*}
\log \left( \frac{rm_{v}}{w_{2}}h_{1}+\frac{rm_{w}}{w_{2}}%
h_{2},w_{2}h_{2},rh\right) =\left( \frac{rm_{v}}{w_{2}}h_{1}+\frac{rm_{w}}{%
w_{2}}h_{2}\right) X+\left( w_{2}h_{2}\right) Y \\
+\left( hr-\frac{1}{2}rh_{1}h_{2}m_{v}-\frac{1}{2}rh_{2}^{2}m_{w}\right) Z.
\end{multline*}%
The commutator satisfies

\begin{equation*}
\left[ \left( \frac{rm_{v}}{w_{2}}h_{1}+\frac{rm_{w}}{w_{2}}%
h_{2},w_{2}h_{2},rh\right) ,\left( \frac{rm_{v}}{w_{2}}k_{1}+\frac{rm_{w}}{%
w_{2}}k_{2},w_{2}k_{2},kr\right) \right] =\left( 0,0,rm_{v}\left(
h_{1}k_{2}-h_{2}k_{1}\right) \right) .\allowbreak
\end{equation*}

We now determine a spin structure by fixing $\varepsilon :\Gamma \rightarrow
\left\{ 1,-1\right\} $. Let%
\begin{eqnarray*}
\varepsilon _{1} &=&\varepsilon \left( v,0\right) =\varepsilon \left( \frac{%
rm_{v}}{w_{2}},0,0\right) , \\
\varepsilon _{2} &=&\varepsilon \left( w,0\right) =\varepsilon \left( \frac{%
rm_{w}}{w_{2}},w_{2},0\right) , \\
\varepsilon _{3} &=&\varepsilon \left( 0,0,r\right) .
\end{eqnarray*}%
Since $\varepsilon _{1}\varepsilon _{2}\varepsilon _{1}^{-1}\varepsilon
_{2}^{-1}=\left( \varepsilon _{3}\right) ^{m_{v}}=1$ is the only relation,
the values of $\varepsilon _{1}$ and $\varepsilon _{2}$ are arbitrary ($\pm
1 $), but it may be that $\varepsilon _{3}$ is restricted by $\varepsilon
_{3}{}^{m_{v}}=1$. If $m_{v}$ is even, there is no restriction, but%
\begin{equation}
\text{if }m_{v}\text{ is odd, then }\varepsilon _{3}=1.
\label{m_v_odd_eps_condition}
\end{equation}%
Now we choose an arbitrary element $\alpha \in \mathfrak{g}^{\ast },$ we may
either choose $\alpha =\alpha _{3}Z^{\ast }$ or $\alpha =\alpha _{1}X^{\ast
}+\alpha _{2}Y^{\ast }$, since all possible coadjoint orbits may be
parametrized by such elements. The occurrence condition is calculated on $v$
and $w$. In particular, $\alpha \left( v\right) $ must be an integer or half
integer depending on whether $\varepsilon _{1}=\pm 1$. Likewise for $\alpha
\left( w\right) $. From Section \ref{MooreRichardsonPapersSection}, the
occurrence conditions are:%
\begin{eqnarray}
\alpha _{1}\frac{rm_{v}}{w_{2}} &\in &\mathbb{Z}+\frac{1-\varepsilon _{1}}{4}%
,  \label{torusCondition1} \\
\alpha _{1}\frac{rm_{w}}{w_{2}}+\alpha _{2}w_{2} &\in &\mathbb{Z}+\frac{%
1-\varepsilon _{2}}{4},  \label{torusCondition2} \\
\alpha _{3}r &\in &\mathbb{Z}+\frac{1-\varepsilon _{3}}{4}.
\label{centralCondition}
\end{eqnarray}

The multiplicities corresponding to these representations are as follows. If
we choose $\alpha \in \mathfrak{g}^{\ast }$ such that $\alpha =\alpha
_{1}X^{\ast }+\alpha _{2}Y^{\ast }$, then $m_{\alpha }=1$ (see Section \ref%
{MooreRichardsonPapersSection}). If we choose $\alpha \in \mathfrak{g}^{\ast
}$ such that $\alpha =\alpha _{3}Z^{\ast }$, then $\mathfrak{g}_{\alpha }=%
\mathfrak{z}$. We have 
\begin{equation*}
m_{\alpha }=\sqrt{\det \left( \left. B_{\alpha }\right\vert _{\text{span}%
\left\{ X,Y\right\} }\right) }
\end{equation*}%
with respect to a lattice basis of $\mathcal{L}$, chosen to be $v=\frac{%
rm_{v}}{w_{2}}X$, $w=\frac{rm_{w}}{w_{2}}X+w_{2}Y$, and thus 
\begin{equation*}
\left( 
\begin{array}{cc}
B_{\alpha }\left( v,v\right) & B_{\alpha }\left( v,w\right) \\ 
B_{\alpha }\left( w,v\right) & B_{\alpha }\left( w,w\right)%
\end{array}%
\right) =\left( 
\begin{array}{cc}
0 & \alpha \left( \left[ v,w\right] \right) \\ 
-\alpha \left( \left[ v,w\right] \right) & 0%
\end{array}%
\right) .
\end{equation*}%
So 
\begin{equation*}
m_{\alpha }=\left\vert \alpha \left( \left[ v,w\right] \right) \right\vert
=\left\vert \alpha _{3}\right\vert rm_{v}\in \mathbb{Z}_{>0}.
\end{equation*}%
The conditions (\ref{m_v_odd_eps_condition}) and (\ref{centralCondition})
confirm that $m_{\alpha }$ is an integer.

\vspace{1pt}Now we are ready to calculate the spectrum of the Dirac operator
on a general Heisenberg $3$-manifold with spin structure. Such a manifold
with spin structure is given by $\left( \mathcal{L},r,g_{A},\varepsilon
\right) $, and it is determined by the lattice basis $v=\frac{rm_{v}}{w_{2}}%
X,$ $w=\frac{rm_{w}}{w_{2}}X+w_{2}Y$ for $\mathcal{L}$ and $\varepsilon
_{1},\varepsilon _{2},\varepsilon _{3}$ as above with conditions (\ref%
{conditions_on_parameters}), (\ref{m_v_odd_eps_condition}), (\ref%
{torusCondition1}), (\ref{torusCondition2}), (\ref{centralCondition}).

We now calculate the part of the spectrum corresponding to each coadjoint
orbit in $\mathfrak{g}^{\ast }$. There are two cases, $\alpha _{3}=0$ and $%
\alpha _{3}\neq 0$. If $\varepsilon _{3}=-1$, the condition (\ref%
{centralCondition}) does not permit $\alpha _{3}=0$. As a consequence,
finite-dimensional irreducible subspaces do not occur. If $\varepsilon
_{3}=1 $, condition (\ref{centralCondition})\ is satisfied and $\alpha
=\alpha _{1}X^{\ast }+\alpha _{2}Y^{\ast }$. The conditions (\ref%
{torusCondition1}) and (\ref{torusCondition2}) are satisfied if and only if
there exist $j_{1},j_{2}\in \mathbb{Z}$ such that{\LARGE \ } 
\begin{align*}
\alpha _{1}& =\frac{w_{_{2}}}{rm_{v}}\left( j_{1}+\frac{1-\varepsilon _{_{1}}%
}{4}\right) , \\
\alpha _{2}& =\frac{1}{w_{2}}\left[ \left( j_{2}+\frac{1-\varepsilon _{2}}{4}%
\right) -\frac{m_{w}}{m_{v}}\left( j_{1}+\frac{1-\varepsilon _{_{1}}}{4}%
\right) \right] ,
\end{align*}%
with eigenvalues%
\begin{eqnarray*}
\sigma _{\alpha } &=&\left\{ -\frac{1}{4A}+2\pi \left\Vert \alpha
\right\Vert ,-\frac{1}{4A}-2\pi \left\Vert \alpha \right\Vert \right\} \\
&=&\left\{ -\frac{1}{4A}+2\pi \sqrt{\frac{\alpha _{1}^{2}+\alpha _{2}^{2}}{A}%
},-\frac{1}{4A}-2\pi \sqrt{\frac{\alpha _{1}^{2}+\alpha _{2}^{2}}{A}}%
\right\} ,~
\end{eqnarray*}%
and the multiplicity of this representation is $m_{\alpha }=1$. If $\alpha
=0 $ is permitted --- i.e. $\varepsilon =1$ --- then $\mathcal{H}_{\alpha }$
is no longer irreducible, and the eigenspace corresponding to $-\frac{1}{4A}$
is two-dimensional.

We now consider the case $\alpha _{3}\neq 0$. By (\ref{centralCondition}) we
may choose $\alpha \in \mathfrak{g}^{\ast }$ in the coadjoint orbit such
that $\alpha =dZ^{\ast }=\frac{1}{r}\left( \kappa +\frac{1-\varepsilon _{3}}{%
4}\right) Z^{\ast }\neq 0$ with $\kappa \in \mathbb{Z}$, with eigenvalues%
\begin{equation*}
\left\{ -\frac{1}{4A}-2\pi \left\vert d\right\vert \right\} \cup \left\{ -%
\frac{1}{4A}\pm 2\sqrt{\frac{\pi \left\vert d\right\vert p}{A}+\pi ^{2}d^{2}}%
:p\geq 1\right\} , 
\end{equation*}%
or in other words%
\begin{align*}
\sigma _{\alpha }& =\left\{ -\frac{1}{4A}-\frac{2\pi }{r}\left\vert \kappa +%
\frac{1-\varepsilon _{3}}{4}\right\vert \right\} \cup \\
& \left\{ -\frac{1}{4A}\pm 2\sqrt{\frac{\pi p}{rA}\left\vert \kappa +\frac{%
1-\varepsilon _{3}}{4}\right\vert +\frac{\pi ^{2}}{r^{2}}\left( \kappa +%
\frac{1-\varepsilon _{3}}{4}\right) ^{2}}:p\in \mathbb{Z}_{>0}\right\} ,
\end{align*}%
and the multiplicity of this representation is 
\begin{equation*}
m_{\alpha }=m_{v}\left\vert \kappa +\frac{1-\varepsilon _{3}}{4}\right\vert
>0. 
\end{equation*}%
A special case occurs in \cite{Am-Ba}, with $r=T^{\prime }$, $A=\left(
d^{\prime }\right) ^{2}T^{\prime }$, $m_{v}=r^{\prime }$, $p=p^{\prime }$, $%
d=\frac{\tau ^{\prime }}{T^{\prime }}$, $m_{w}=0$, where the primes indicate
the notation used in \cite{Am-Ba}.

\subsection{Eta invariant of three-dimensional Heisenberg manifolds\label%
{EtaInvt3DHeisenSection}}

From (\ref{etaInvt3MfldsNegCase}), the eta invariant of the spin Dirac
operator corresponding to a spin structure on a three-dimensional manifold
is ($n=3,\widehat{n}=2,W$ is trivial so that $\mathrm{tr}\left( F^{W}\right)
=0$)

\begin{equation*}
\eta \left( 0\right) -\eta _{-\overline{\lambda }}\left( 0\right) =-\frac{%
\overline{\lambda }^{3}}{3\pi ^{2}}\mathrm{vol}\left( M\right) +\frac{%
\overline{\lambda }}{24\pi ^{2}}\int_{M}\mathrm{Scal}-2\#\left( \sigma
\left( D\right) \cap \left( \overline{\lambda },0\right) \right) -\#\left(
\sigma \left( D\right) \cap \left\{ 0,\overline{\lambda }\right\} \right) ,
\end{equation*}%
where $\overline{\lambda }=-\frac{1}{4A}$, and where the last two terms
count multiplicities. (Recall the rank of the spinor bundle is two.)

We now calculate $\eta _{-\overline{\lambda }}\left( 0\right) $. If $%
\varepsilon _{3}=1$, and from the formulas for $\sigma_\alpha$ and $m_\alpha$ in the last section, we have after deleting terms that cancel,

\begin{eqnarray*}
\eta _{-\overline{\lambda }}\left( s\right) &=&\sum_{\kappa \in \mathbb{%
Z\setminus }\left\{ 0\right\} }\left( -1\right) m_{v}\left\vert \kappa
\right\vert \left( \frac{2\pi }{r}\left\vert \kappa \right\vert \right) ^{-s}
\\
&=&-m_{v}\frac{\left( 2\pi \right) ^{-s}}{r^{-s}}\sum_{\kappa \in \mathbb{%
Z\setminus }\left\{ 0\right\} }\left( \left\vert \kappa \right\vert \right)
^{-s+1} \\
&=&-2m_{v}\frac{\left( 2\pi \right) ^{-s}}{r^{-s}}\sum_{\kappa \in \mathbb{Z}%
_{>0}}\kappa ^{-s+1} \\
&=&-2m_{v}\frac{\left( 2\pi \right) ^{-s}}{r^{-s}}\sum_{\kappa \in \mathbb{Z}%
_{>0}}\kappa ^{-\left( s-1\right) }=-2m_{v}\frac{\left( 2\pi \right) ^{-s}}{%
r^{-s}}\zeta _{R}\left( s-1\right) ,
\end{eqnarray*}%
where $\zeta _{R}$ is the Riemann zeta function. Thus, if $\varepsilon
_{3}=1 $.%
\begin{eqnarray*}
\eta _{-\overline{\lambda }}\left( 0\right) &=&-2m_{v}\zeta _{R}\left(
-1\right) \\
&=&-2m_{v}\left( -\frac{1}{12}\right) \\
&=&\frac{m_{v}}{6}.
\end{eqnarray*}%
On the other hand, if $\varepsilon _{3}=-1$, which implies $m_{v}$ is even, 
\begin{eqnarray*}
\eta _{-\overline{\lambda }}\left( s\right) &=&\sum_{\kappa \in \mathbb{Z}%
}\left( -1\right) m_{v}\left\vert \kappa +\frac{1}{2}\right\vert \left( 
\frac{2\pi }{r}\left\vert \kappa +\frac{1}{2}\right\vert \right) ^{-s} \\
&=&-m_{v}\frac{\left( 2\pi \right) ^{-s}}{r^{-s}}\sum_{\kappa \in \mathbb{Z}%
}\left\vert \frac{2\kappa +1}{2}\right\vert ^{-s+1} \\
&=&-m_{v}\frac{\left( 2\pi \right) ^{-s}}{r^{-s}}2^{s-1}\sum_{\kappa \in 
\mathbb{Z}}\left\vert 2\kappa +1\right\vert ^{-s+1} \\
&=&-2m_{v}\frac{\left( 2\pi \right) ^{-s}}{r^{-s}}2^{s-1}\zeta _{R,\text{odd}%
}\left( s-1\right) ,
\end{eqnarray*}%
where $\zeta _{R,\text{odd}}\left( s\right) =\sum_{n=0}^{\infty }\left(
2n+1\right) ^{-s}=\left( 1-2^{-s}\right) \zeta _{R}\left( s\right) $. Thus,
if $\varepsilon _{3}=-1$, 
\begin{eqnarray*}
\eta _{-\overline{\lambda }}\left( 0\right) &=&\left. -2m_{v}\frac{\left(
2\pi \right) ^{-s}}{r^{-s}}2^{s-1}\left( 1-2^{-\left( s-1\right) }\right)
\zeta _{R}\left( s-1\right) \right\vert _{s=0} \\
&=&-2m_{v}2^{-1}\left( 1-2\right) \zeta _{R}\left( -1\right) \\
&=&m_{v}\zeta _{R}\left( -1\right) =-\frac{m_{v}}{12}
\end{eqnarray*}%
In all cases,%
\begin{equation*}
\eta _{-\overline{\lambda }}\left( 0\right) =\frac{m_{v}}{24}\left(
3\varepsilon _{3}+1\right) .
\end{equation*}

From (\cite[Section 2]{Eb}), 
\begin{equation*}
\mathrm{Scal}=\frac{1}{4}\mathrm{Tr}\left( j\left( Z\right) ^{2}\right) =%
\frac{1}{4}\mathrm{Tr}\left( \left( 
\begin{array}{cc}
0 & -\frac{1}{A} \\ 
\frac{1}{A} & 0%
\end{array}%
\right) ^{2}\right) =-\frac{1}{2A^{2}}.
\end{equation*}%
Also,%
\begin{equation*}
\mathrm{vol}\left( M\right) =rA\det \left( 
\begin{array}{cc}
\frac{rm_{v}}{w_{2}} & \frac{rm_{w}}{w_{2}} \\ 
0 & w_{2}%
\end{array}%
\right) =r^{2}Am_{v}.
\end{equation*}%
From the expressions for the eigenvalues of $\sigma \left( D\right) $, we
see that $\#\left( \sigma \left( D\right) \cap \left\{ -\frac{1}{4A}\right\}
\right) $ is nonzero only if the part of the spectrum corresponding to $%
\alpha =0\in \mathfrak{g}^{\ast }$ is nontrivial. This happens only if $%
\varepsilon _{1}=\varepsilon _{2}=\varepsilon _{3}=1$. Thus,%
\begin{equation*}
\#\left( \sigma \left( D\right) \cap \left\{ -\frac{1}{4A}\right\} \right)
=\left\{ 
\begin{array}{ll}
2 & \text{if }\varepsilon =\mathbf{1} \\ 
0~ & \text{otherwise}%
\end{array}%
\right. .
\end{equation*}

To count $\#\left( \sigma \left( D\right) \cap \left( -\frac{1}{4A},0\right)
\right) $, the toral eigenvalues, i.e. the ones from the finite-dimensional
irreducible subspaces, are (see (\ref{torusSpectrum3dHeisenberg}))%
\begin{eqnarray*}
&&\left\{ -\frac{1}{4A}+\tau :0<\tau =2\pi \sqrt{\frac{\alpha
_{1}^{2}+\alpha _{2}^{2}}{A}}<\frac{1}{4A}\right\} \\
&=&\left\{ -\frac{1}{4A}+\tau :\tau =2\pi \frac{\left\Vert \alpha
\right\Vert }{\sqrt{A}},0<\left\Vert \alpha \right\Vert <\frac{1}{8\pi \sqrt{%
A}}\right\}
\end{eqnarray*}%
With fixed $r>0,w_{2}>0,m_{v}\in \mathbb{Z}_{>0},m_{w}\in \mathbb{Z}$, by (%
\ref{torusCondition1}), (\ref{torusCondition2}) the coadjoint orbit
represented by $\alpha =\alpha _{1}X^{\ast }+\alpha _{2}Y^{\ast }$ has an
associated irreducible representation that occurs with multiplicity one if
and only if 
\begin{align*}
\frac{\alpha _{1}rm_{v}}{w_{2}}& \in \mathbb{Z}+\frac{1-\varepsilon _{_{1}}}{%
4}, \\
\alpha _{2}w_{2}+\frac{rm_{w}}{w_{2}}\alpha _{1}& \in \mathbb{Z}+\frac{%
1-\varepsilon _{_{2}}}{4}.
\end{align*}%
The relevant nontoral eigenvalues, i.e. those from the infinite-dimensional
irreducible subspaces, are

\begin{equation*}
\sigma _{\alpha }=\left\{ 
\begin{array}{c}
-\frac{1}{4A}+2\sqrt{\frac{\pi p}{rA}\left\vert \kappa +\frac{1-\varepsilon
_{3}}{4}\right\vert +\frac{\pi ^{2}}{r^{2}}\left( \kappa +\frac{%
1-\varepsilon _{3}}{4}\right) ^{2}}:p\in \mathbb{Z}_{>0},\kappa \in \mathbb{Z%
}, \\ 
0<rp\left\vert \kappa +\frac{1-\varepsilon _{3}}{4}\right\vert +\pi A\left(
\kappa +\frac{1-\varepsilon _{3}}{4}\right) ^{2}<\frac{r^{2}}{64\pi A}%
\end{array}%
\right\} ,
\end{equation*}%
with multiplicity $m_{\alpha }=m_{v}\left\vert \kappa +\frac{1-\varepsilon
_{3}}{4}\right\vert >0$. Letting $\mu =\kappa +\frac{1-\varepsilon _{3}}{4}%
\in \mathbb{Z}+\frac{1-\varepsilon _{3}}{4}$, the inequality $0<rp\left\vert
\mu \right\vert +\pi A\mu ^{2}<\frac{r^{2}}{64\pi A}$ is equivalent to 
\begin{equation*}
0<\left\vert \mu \right\vert <\frac{r}{2\pi A}\left( \sqrt{\frac{1}{16}+p^{2}%
}-p\right) ,
\end{equation*}%
so the relevant nontoral eigenvalues in the open interval $\left( -\frac{1}{%
4A},0\right) $ associated to $\left. D\right\vert _{\mathcal{H}_{\alpha }}$
are 
\begin{equation*}
\sigma _{\alpha }=\left\{ 
\begin{array}{c}
-\frac{1}{4A}+2\sqrt{\frac{\pi p}{rA}\left\vert \mu \right\vert +\frac{\pi
^{2}}{r^{2}}\mu ^{2}}:p\in \mathbb{Z}_{>0},\mu \in \mathbb{Z}+\frac{%
1-\varepsilon _{3}}{4}, \\ 
0<\left\vert \mu \right\vert <\frac{r}{2\pi A}\left( \sqrt{\frac{1}{16}+p^{2}%
}-p\right)%
\end{array}%
\right\} ,
\end{equation*}%
with multiplicities $m_{\alpha }=m_{v}\left\vert \mu \right\vert $.

In summary, summing over all coadjoint orbits whose associated irreducible
representation occurs in $\rho _{\varepsilon }$ ,%
\begin{eqnarray*}
\#\left( \sigma \left( D\right) \cap \left( -\frac{1}{4A},0\right) \right)
&=&\#\left\{ 
\begin{array}{c}
\left( \alpha _{1},\alpha _{2}\right) :\frac{\alpha _{1}rm_{v}}{w_{2}}\in 
\mathbb{Z}+\frac{1-\varepsilon _{_{1}}}{4}, \\ 
\alpha _{2}w_{2}+\frac{rm_{w}}{w_{2}}\alpha _{1}\in \mathbb{Z}+\frac{%
1-\varepsilon _{_{2}}}{4},~0<\left\Vert \alpha \right\Vert <\frac{1}{8\pi 
\sqrt{A}}%
\end{array}%
\right\} \\
&&+m_{v}\sum_{p\in \mathbb{Z}_{>0}}\sum_{\substack{ \mu \in \mathbb{Z}+\frac{%
1-\varepsilon _{3}}{4}  \\ 0<\left\vert \mu \right\vert <\frac{r}{2\pi A}%
\left( \sqrt{\frac{1}{16}+p^{2}}-p\right) }}\left\vert \mu \right\vert .
\end{eqnarray*}

Likewise,%
\begin{eqnarray*}
\#\left( \sigma \left( D\right) \cap \left\{ 0\right\} \right) &=&\#\left\{ 
\begin{array}{c}
\left( \alpha _{1},\alpha _{2}\right) :\frac{\alpha _{1}rm_{v}}{w_{2}}\in 
\mathbb{Z}+\frac{1-\varepsilon _{_{1}}}{4}, \\ 
\alpha _{2}w_{2}+\frac{rm_{w}}{w_{2}}\alpha _{1}\in \mathbb{Z}+\frac{%
1-\varepsilon _{_{2}}}{4},~\left\Vert \alpha \right\Vert =\frac{1}{8\pi 
\sqrt{A}}%
\end{array}%
\right\} \\
&&+m_{v}\sum_{p\in \mathbb{Z}_{>0}}\sum_{\substack{ \mu \in \mathbb{Z}+\frac{%
1-\varepsilon _{3}}{4}  \\ \left\vert \mu \right\vert =\frac{r}{2\pi A}%
\left( \sqrt{\frac{1}{16}+p^{2}}-p\right) }}\left\vert \mu \right\vert .
\end{eqnarray*}

We now show that the last line produces at most two nonzero terms. If $\mu
_{1},\mu _{2}\in \mathbb{Z}+\frac{1-\varepsilon _{3}}{4}$ both satisfy $%
\left\vert \mu _{j}\right\vert =\frac{r}{2\pi A}\left( \sqrt{\frac{1}{16}%
+p_{j}^{2}}-p_{j}\right) >0$ and $\left\vert \mu _{1}\right\vert \neq
\left\vert \mu _{2}\right\vert $, solving for $\frac{r}{2\pi A}$ yields 
\begin{equation*}
k\left( \sqrt{1+16p_{1}^{2}}+4p_{1}\right) -h\left( \sqrt{1+16p_{2}^{2}}%
+4p_{2}^{2}\right) =0
\end{equation*}%
for some positive $h,k\in \mathbb{Z}+\frac{1-\varepsilon _{3}}{4}$, and 
\begin{equation*}
1-\frac{h^{2}}{k^{2}}+32\frac{h}{k}p_{1}p_{2}-32\frac{h^{2}}{k^{2}}%
p_{2}^{2}=8\frac{h}{k}\left( \allowbreak \frac{h}{k}p_{2}-p_{1}\right) \sqrt{%
16p_{2}^{2}+1}.
\end{equation*}%
If $\frac{h}{k}p_{2}=p_{1}$, then the equation above implies $p_{1}=p_{2}$.
On the other hand,if $\left( \allowbreak \frac{h}{k}p_{2}-p_{1}\right) $ is
not zero,

\begin{equation*}
\frac{1-\frac{h^{2}}{k^{2}}+32\frac{h}{k}p_{1}p_{2}-32\frac{h^{2}}{k^{2}}%
p_{2}^{2}}{8\frac{h}{k}\left( \allowbreak \frac{h}{k}p_{2}-p_{1}\right) }=%
\sqrt{16p_{2}^{2}+1}.
\end{equation*}%
Since the left side is rational and the right side is irrational, this is
impossible.

\vspace{1pt}

Thus, there are at most two nonzero summands in the expression below.

\begin{multline*}
m_{v}\sum_{p\in \mathbb{Z}_{>0}}\sum_{\substack{ \mu \in \mathbb{Z}+\frac{%
1-\varepsilon _{3}}{4}  \\ \left\vert \mu \right\vert =\frac{r}{2\pi A}%
\left( \sqrt{\frac{1}{16}+p^{2}}-p\right) }}\left\vert \mu \right\vert \\
=\left\{ 
\begin{array}{ll}
\frac{m_{v}r}{2\pi A}\left( \sqrt{\frac{1}{16}+p^{2}}-p\right) ~ & \text{if }%
\frac{r}{2\pi A}\left( \sqrt{\frac{1}{16}+p^{2}}-p\right) \text{ }\in 
\mathbb{Z}+\frac{1-\varepsilon _{3}}{4}\text{ for some }p\in \mathbb{Z}_{>0}
\\ 
0 & \text{otherwise}%
\end{array}%
\right.
\end{multline*}%
Then%
\begin{multline*}
m_{v}\sum_{p\in \mathbb{Z}_{>0}}\sum_{\substack{ \mu \in \mathbb{Z}+\frac{%
1-\varepsilon _{3}}{4}  \\ \mu =\pm \frac{r}{2\pi A}\left( \sqrt{\frac{1}{16}%
+p^{2}}-p\right) }}\left\vert \mu \right\vert \\
=\left\{ 
\begin{array}{ll}
\frac{m_{v}r}{\pi A}\left( \sqrt{\frac{1}{16}+p^{2}}-p\right) ~ & \text{if }%
\frac{r}{2\pi A}\left( \sqrt{\frac{1}{16}+p^{2}}-p\right) \text{ }\in 
\mathbb{Z}+\frac{1-\varepsilon _{3}}{4}\text{ for some }p\in \mathbb{Z}_{>0}
\\ 
0 & \text{otherwise}%
\end{array}%
\right.
\end{multline*}

In summary,

\begin{multline*}
\#\left( \sigma \left( D\right) \cap \left\{ 0\right\} \right) =\#\left\{ 
\begin{array}{c}
\left( \alpha _{1},\alpha _{2}\right) :\frac{\alpha _{1}rm_{v}}{w_{2}}\in 
\mathbb{Z}+\frac{1-\varepsilon _{_{1}}}{4}, \\ 
\alpha _{2}w_{2}+\frac{rm_{w}}{w_{2}}\alpha _{1}\in \mathbb{Z}+\frac{%
1-\varepsilon _{_{2}}}{4},~\left\Vert \alpha \right\Vert =\frac{1}{8\pi 
\sqrt{A}}%
\end{array}%
\right\} \\
+\left\{ 
\begin{array}{ll}
\frac{m_{v}r}{\pi A}\left( \sqrt{\frac{1}{16}+p^{2}}-p\right) ~ & \text{if }%
\frac{r}{2\pi A}\left( \sqrt{\frac{1}{16}+p^{2}}-p\right) \text{ }\in 
\mathbb{Z}+\frac{1-\varepsilon _{3}}{4}\text{ for some }p\in \mathbb{Z}_{>0}
\\ 
0 & \text{otherwise}%
\end{array}%
\right. .
\end{multline*}

Putting these calculations together, we have 
\begin{eqnarray*}
\eta \left( 0\right) -\eta _{-\overline{\lambda }}\left( 0\right) &=&-\frac{%
\overline{\lambda }^{3}}{3\pi ^{2}}\mathrm{vol}\left( M\right) +\frac{%
\overline{\lambda }}{24\pi ^{2}}\int_{M}\mathrm{Scal}-2\#\left( \sigma
\left( -D\right) \cap \left( \overline{\lambda },0\right) \right) -\#\left(
\sigma \left( -D\right) \cap \left\{ 0,\overline{\lambda }\right\} \right) \\
&=&-\frac{\left( -\frac{1}{4A}\right) ^{3}}{3\pi ^{2}}r^{2}Am_{v}+\frac{%
\left( -\frac{1}{4A}\right) }{24\pi ^{2}}\left( -\frac{1}{2A^{2}}\right)
\left( r^{2}Am_{v}\right) -N\left( A,r,w_{2},m_{v},m_{w},\varepsilon \right)
\\
&=&\frac{r^{2}m_{v}}{192\pi ^{2}A^{2}}+\frac{r^{2}m_{v}}{192\pi ^{2}A^{2}}%
-N\left( A,r,w_{2},m_{v},m_{w},\varepsilon \right) \\
&=&\frac{r^{2}m_{v}}{96\pi ^{2}A^{2}}-N\left(
A,r,w_{2},m_{v},m_{w},\varepsilon \right) ,
\end{eqnarray*}%
where $N\left( A,r,w_{2},m_{v},m_{w},\varepsilon \right) $ is the
nonnegative integer defined by

\begin{eqnarray}
N(\cdot ) &=&2\#\left( \sigma \left( D\right) \cap \left( \overline{\lambda }%
,0\right) \right) +\#\left( \sigma \left( D\right) \cap \left\{ 0,\overline{%
\lambda }\right\} \right)  \notag \\
&=&2\#\left\{ 
\begin{array}{c}
\left( \alpha _{1},\alpha _{2}\right) :\frac{\alpha _{1}rm_{v}}{w_{2}}\in 
\mathbb{Z}+\frac{1-\varepsilon _{_{1}}}{4}, \\ 
\alpha _{2}w_{2}+\frac{rm_{w}}{w_{2}}\alpha _{1}\in \mathbb{Z}+\frac{%
1-\varepsilon _{_{2}}}{4},~0<\left\Vert \alpha \right\Vert <\frac{1}{8\pi 
\sqrt{A}}%
\end{array}%
\right\}  \notag \\
&&+2m_{v}\sum_{p\in \mathbb{Z}_{>0}}\sum_{\substack{ \mu \in \mathbb{Z}+%
\frac{1-\varepsilon _{3}}{4}  \\ 0<\left\vert \mu \right\vert <\frac{r}{8\pi
A\left( \sqrt{1+16p^{2}}+4p\right) }}}\left\vert \mu \right\vert  \notag \\
&&+\#\left\{ 
\begin{array}{c}
\left( \alpha _{1},\alpha _{2}\right) :\frac{\alpha _{1}rm_{v}}{w_{2}}\in 
\mathbb{Z}+\frac{1-\varepsilon _{_{1}}}{4}, \\ 
\alpha _{2}w_{2}+\frac{rm_{w}}{w_{2}}\alpha _{1}\in \mathbb{Z}+\frac{%
1-\varepsilon _{_{2}}}{4},~\left\Vert \alpha \right\Vert =\frac{1}{8\pi 
\sqrt{A}}%
\end{array}%
\right\}  \notag \\
&&+\left\{ 
\begin{array}{ll}
\frac{m_{v}r}{\pi A}\left( \sqrt{\frac{1}{16}+p^{2}}-p\right) ~ & \text{if }%
\frac{r}{2\pi A}\left( \sqrt{\frac{1}{16}+p^{2}}-p\right) \text{ }\in 
\mathbb{Z}+\frac{1-\varepsilon _{3}}{4}\text{ for some }p\in \mathbb{Z}_{>0}
\\ 
0 & \text{otherwise}%
\end{array}%
\right.  \notag \\
&&+\left\{ 
\begin{array}{ll}
2 & \text{if }\varepsilon _{1}=\varepsilon _{2}=\varepsilon _{3}=1 \\ 
0~ & \text{otherwise}%
\end{array}%
\right. .  \label{N_equation}
\end{eqnarray}%
All the sums above are finite.

We summarize this result in the following theorem.

\begin{theorem}
The eta invariant of the spin Dirac operator on a three-dimensional
Heisenberg manifold with parameters $A,r,w_{2}>0$, $m_{v}\in \mathbb{Z}_{>0}$%
, $m_{w}\in \mathbb{Z}$ with spin structure determined by $\varepsilon
=\left( \varepsilon _{1},\varepsilon _{2},\varepsilon _{3}\right) \in
\left\{ \pm 1\right\} ^{3}$ satisfies 
\begin{equation*}
\eta \left( 0\right) =\frac{r^{2}m_{v}}{96\pi ^{2}A^{2}}+\frac{m_{v}}{24}%
\left( 3\varepsilon _{3}+1\right) -N\left( A,r,w_{2},m_{v},m_{w},\varepsilon
\right) ,
\end{equation*}%
where $N\left( A,r,w_{2},m_{v},m_{w},\varepsilon \right) $ is the
nonnegative integer given by the expression (\ref{N_equation}).
\end{theorem}

\begin{remark}
The expression above is consistent with the calculation of W. Zhang in \cite%
{Zhang}, who calculates the adiabatic limit of the $\mathrm{mod}1$ reduction
of the eta invariant on specific types of circle bundles $M\rightarrow B$,
where the metric on the base $B$ is blown up $\left( A\rightarrow \infty
\right) $. The Zhang formula for the case of the spin Dirac operator twisted
by a line bundle $L$ over the base (corresponding to an induced spin
structure) is 
\begin{eqnarray*}
\lim_{A\rightarrow \infty }\overline{\eta }\left( D_{M,L}\right) &=&\frac{1}{%
2}\dim \ker \left( D_{B,L}\right) \\
&&+\left\langle \widehat{A}\left( TB\right) ch\left( L\right) \frac{\tanh
\left( \frac{e}{2}\right) -\frac{e}{2}}{e\tanh \left( \frac{e}{2}\right) },%
\left[ B\right] \right\rangle \mathrm{mod}1,
\end{eqnarray*}%
where $e$ is the Euler class of the circle bundle. In our case, the base is
a flat torus, and the fibers of the circle bundle are the $Z$-parameter
curves.\ First, we consider the right side of the equation. The integer $%
\dim \ker \left( D_{B,L}\right) $ is either $2$ or zero (depending on
whether the line bundle $L$ is trivial or not), so that term is zero. The
Euler number (integral of the Euler class) of the circle bundle can be shown
to be the same as the positive integer $m_{v}$. We also note that since the
fiber spin structure considered by Zhang is induced from a trivial spin
structure on a disk, the induced spin structure restricted to each circle
fiber is the the nontrivial one, so that $\varepsilon _{3}=-1$. For this and
other reasons, we note that the Euler number $m_{v}$ is even. The relevant
characteristic forms are $\widehat{A}\left( TB\right) =1,~ch\left( L\right)
=1+\left( 2\text{-form}\right) ,\frac{\tanh \left( \frac{e}{2}\right) -\frac{%
e}{2}}{e\tanh \left( \frac{e}{2}\right) }=-\frac{e}{12}$, so the second term
is $-\frac{m_{v}}{12}$. The left side of Zhang's equation is%
\begin{equation*}
\lim_{A\rightarrow \infty }\left( \frac{r^{2}m_{v}}{96\pi ^{2}A^{2}}-\frac{%
m_{v}}{12}-N(...)\mathrm{mod}1\right) =-\frac{m_{v}}{12}
\end{equation*}%
in our case, so we see that our formula is consistent with this result.
\end{remark}

\begin{corollary}
From the expressions for $\eta \left( 0\right) $, we may consider families
of Heisenberg manifolds with constant $\eta \left( 0\right) $. For example,
if we let%
\begin{eqnarray*}
A &=&b_{1}r \\
w_{2} &=&b_{2}\sqrt{r}
\end{eqnarray*}%
for some constants $b_{1},b_{2}>0$. Holding $m_{v}$, $m_{w}$, $\varepsilon $%
, $b_{1}$, $b_{2}$ constant and letting $r$ vary, we obtain a family of
Heisenberg manifolds with constant $\eta \left( 0\right) $ yet with
different eigenvalues for $D$; even the point of symmetry $-\frac{1}{4A}$
varies with $r$.
\end{corollary}

\begin{corollary}
Consider the \textquotedblleft rectangular\textquotedblright\ Heisenberg
3-manifold (i.e. $m_{w}=0$). Suppose that the following conditions are met:

\begin{enumerate}
\item $A>\frac{r}{4\pi }$

\item $\frac{rm_{v}}{4\pi \sqrt{A}}<w_{2}<4\pi \sqrt{A}$
\end{enumerate}

Then if the spin structure is nontrivial ($\varepsilon \neq \mathrm{id}$),%
\begin{equation*}
\eta \left( 0\right) =\frac{r^{2}m_{v}}{96\pi ^{2}A^{2}}+\frac{m_{v}}{24}%
\left( 3\varepsilon _{3}+1\right) .
\end{equation*}%
Otherwise,%
\begin{equation*}
\eta \left( 0\right) =\frac{r^{2}m_{v}}{96\pi ^{2}A^{2}}+\frac{m_{v}}{6}-2.
\end{equation*}
\end{corollary}

\subsection{Dirac Operator eigenvalues for general Heisenberg nilmanifolds 
\label{DiracEigenvaluesGenHeisenbergSection}}

We use the notation of Section \ref{2StepDiracSection}. Suppose that $%
k_{0}=1 $ is the dimension of the center $\mathfrak{z}$ and $n=1+m_{0}$ is
the dimension of $\mathfrak{g}=\mathfrak{z}\oplus \mathfrak{v}$, and we will
choose the orthonormal basis $\left\{ Z,X_{1},...,X_{m_{0}}\right\} $, with $%
m_{0}=2m$, so that $Z$ is a unit vector and $\left\{ X_{j}\right\} $ is an
orthonormal basis of $\mathfrak{v}$. From formula (\ref{Dirac2Step}), the
Dirac operator is%
\begin{equation*}
D=\sum_{i=1}^{n}\left( E_{i}\diamond \right) \rho _{\varepsilon \ast }\left(
E_{i}\right) \mathbf{+}\frac{1}{4}\sum_{b<i\leq m_{0}}\left\langle Z,\left[
X_{b},X_{i}\right] \right\rangle \left( Z\diamond X_{b}\diamond
X_{i}\diamond \right) ,
\end{equation*}%
acting on 
\begin{equation*}
\mathcal{H}=L^{2}\left( \Gamma \diagdown G,G\times _{\varepsilon }\mathbb{C}%
^{k}\right) \cong L_{\varepsilon }^{2}\left( \Gamma \diagdown G\right)
\otimes \Sigma _{n},
\end{equation*}%
which we decompose using Kirillov theory. Using notation from Section \ref%
{2StepDiracSection}, the cases are:

\textbf{Case 1:} $k_{\alpha }=n$, i.e. $\alpha \left( Z\right) =0$.

As in (\ref{FiniteDimDiracOperatorFormula}), 
\begin{equation}
\left. D\right\vert _{\mathcal{H}_{\alpha }}=\sum_{i=1}^{m_{0}}2\pi i\alpha
\left( X_{i}\right) \left( X_{i}\diamond \right) \mathbf{+}\frac{1}{4}%
\sum_{b<i\leq m_{0}}\left\langle Z,\left[ X_{b},X_{i}\right] \right\rangle
\left( Z\diamond X_{b}\diamond X_{i}\diamond \right) ,
\label{toroidalGenHeisenberg}
\end{equation}%
which is a constant matrix. The eigenvalues of $\left. D\right\vert _{%
\mathcal{H}_{\alpha }}$ are then the eigenvalues of this Hermitian matrix.%
\vspace{1pt}

\textbf{Case 2:} $k_{\alpha }<n$, so that $\alpha \left( Z\right) \neq 0$.

For every noncentral vector $v$, there exists a vector $w$ such that $%
B_{\alpha }\left( v,w\right) =\alpha \left( \left[ v,w\right] \right)
=\alpha \left( Z\right) \neq 0$; we must have $\mathfrak{g}_{\alpha }=%
\mathfrak{z}$ and $k_{\alpha }=1$. From (\ref{M_definition1}), (\ref%
{M_alpha_prime_definition}), (\ref{DequationWithSQRoots}), the Dirac
operator may be expressed in terms of the basis $\left\{ \overline{u}_{%
\mathbf{p},\mathbf{\ell }}\right\} $ as 
\begin{equation*}
D\overline{u}_{\mathbf{p},\mathbf{\ell }}=-\sum_{j}2i\sqrt{\pi d_{j}p_{j}}%
\ell _{1}\ell _{2}...\ell _{j}\overline{u}_{\mathbf{p},\mathbf{\ell }%
^{j}}+M_{\alpha }^{\prime }\overline{u}_{\mathbf{p},\mathbf{\ell }}~,
\end{equation*}%
where in this case%
\begin{equation*}
M_{\alpha }^{\prime }\overline{u}_{\mathbf{p},\mathbf{\ell }}=2\pi i\alpha
\left( Z\right) \left( Z\diamond \right) \overline{u}_{\mathbf{p},\mathbf{%
\ell }}+\frac{1}{4}\sum_{j=1}^{m}\left\langle Z,\left[ U_{j},V_{j}\right]
\right\rangle \left( Z\diamond U_{j}\diamond V_{j}\diamond \right) \overline{%
u}_{\mathbf{p},\mathbf{\ell }}\text{ .}
\end{equation*}%
We use the matrix choices of Section \ref{matrixChoicesSection}, and for
convenience, we let%
\begin{equation*}
Z=i\underset{m\text{ times}}{\underbrace{\mathbf{1}^{\prime }\otimes
...\otimes \mathbf{1}^{\prime }}}~,
\end{equation*}%
and thus, since $\left\langle Z,\left[ U_{j},V_{j}\right] \right\rangle
=d_{j}$ , the formulas (\ref{UjVjDiamond}) and (\ref{CayleyTable}) yield 
\begin{equation*}
M_{\alpha }^{\prime }\overline{u}_{\mathbf{p},\mathbf{\ell }}=\left( -2\pi
\alpha \left( Z\right) \ell _{1}...\ell _{m}-\frac{1}{4}\sum_{j\leq
m}d_{j}\ell _{1}...\widehat{\ell _{j}}...\ell _{m}\right) \overline{u}_{%
\mathbf{p},\mathbf{\ell }}.
\end{equation*}%
In summary,

\begin{equation}
D\overline{u}_{\mathbf{p},\mathbf{\ell }}=-\sum_{j}2i\sqrt{\pi d_{j}p_{j}}%
\ell _{1}\ell _{2}...\ell _{j}\overline{u}_{\mathbf{p},\mathbf{\ell }%
^{j}}+\left( -2\pi \alpha \left( Z\right) \ell _{1}...\ell _{m}-\frac{1}{4}%
\sum_{j\leq m}d_{j}\ell _{1}...\widehat{\ell _{j}}...\ell _{m}\right) 
\overline{u}_{\mathbf{p},\mathbf{\ell }},
\label{DiracHeisenbergIrredSubspaceDecomp}
\end{equation}%
and we have the following.

\begin{proposition}
\label{HeisenbergDiracDecompGenDims}The infinite-dimensional subspace $%
\mathcal{H}_{\alpha }$ decomposes on any Heisenberg manifold as a direct sum
of finite-dimensional subspaces that are invariant by the Dirac operator. In
particular, the Dirac operator acts by the formula (\ref%
{DiracHeisenbergIrredSubspaceDecomp}) on the finite-dimensional invariant
subspace 
\begin{equation*}
\mathcal{U}_{\mathbf{p}}=\mathrm{span}\left\{ \overline{u}_{\mathbf{p},%
\mathbf{\ell }}:\mathbf{\ell }\in \left\{ -1,1\right\} ^{m}\right\} .
\end{equation*}
\end{proposition}

\begin{remark}
For any specific example of a Heisenberg manifold, the formula (\ref%
{toroidalGenHeisenberg}) allows us to calculate the eigenvalues of $D$
restricted to the finite-dimensional representations spaces $\mathcal{H}%
_{\alpha }$ with $\alpha \left( Z\right) =0$, and the previous proposition
allows us to calculate all other eigenvalues of $D$ explicitly.
\end{remark}

\subsection{Symmetries of invariants subspaces of higher dimensional
Heisenberg manifolds\label{VanishingEtaSection}}

It is well-known (see \cite[p.61, Remark 3a]{APS1}, \cite[p.174, Cor. 2.19]%
{AmHuMo}) that the spectrum of the Dirac operator on spin manifolds of
dimension congruent to $1\mod 4$ is symmetric about $0$. In following
sections, we explore the symmetry of the spectrum restricted to invariant
subspaces.

\subsubsection{Symmetry in the toroidal part of the spectrum for Heisenberg
manifolds}

Suppose that we are given a $\left( 2m+1\right) $-dimensional Heisenberg
manifold, and $\alpha \in \mathfrak{g}^{\ast }$ is chosen so that $\alpha
\left( Z\right) =0$. Then we may choose an orthonormal basis $\left\{
A_{1},A_{2},...,A_{m},B_{1},B_{2},...,B_{m}\right\} $ of $\mathfrak{z}^{\bot
}\subseteq \mathfrak{g}$ with the following properties:

\begin{enumerate}
\item $\alpha \left( A_{j}\right) =0$ if $j\geq 2$, $\alpha \left(
B_{j}\right) =0$ if $j\geq 1;$

\item $\left[ A_{i},B_{j}\right] =a_{j}\delta _{ij}Z$ for some real numbers $%
a_{j}$.
\end{enumerate}

(Simply choose $A_{1}$ orthogonal to $\ker \alpha $ and continue to form a
symplectic basis of $\mathfrak{z}^{\bot }$.) Then the restriction of $D$ to
the subspace $\mathcal{H}_{\alpha }$ is

\begin{eqnarray*}
\left. D\right\vert _{\mathcal{H}_{\alpha }} &=&\sum_{i=1}^{m_{0}}2\pi
i\alpha \left( X_{i}\right) \left( X_{i}\diamond \right) \mathbf{+}\frac{1}{4%
}\sum_{b<i\leq m_{0}}\left\langle Z,\left[ X_{b},X_{i}\right] \right\rangle
\left( Z\diamond X_{b}\diamond X_{i}\diamond \right) \\
&=&2\pi i\alpha \left( A_{1}\right) \left( A_{1}\diamond \right) \mathbf{+}%
\frac{1}{4}\sum_{j=1}^{m}a_{j}\left( Z\diamond A_{j}\diamond B_{j}\diamond
\right) .
\end{eqnarray*}

If $m$ is even, then observe that $A_{1}\diamond A_{2}\diamond
...A_{m}\diamond $ anticommutes with $\left. D\right\vert _{\mathcal{H}%
_{\alpha }}$ and is also invertible. Thus, it maps the $\lambda $ eigenspace
of $\left. D\right\vert _{\mathcal{H}_{\alpha }}$ isomorphically onto the $%
-\lambda $ eigenspace of $\left. D\right\vert _{\mathcal{H}_{\alpha }}$, and
therefore the spectrum of $\left. D\right\vert _{\mathcal{H}_{\alpha }}$ is
symmetric about zero and does not contribute to the eta invariant.

A more complicated argument can be used to show that for all $m\geq 2$, the
spectrum of $\left. D\right\vert _{\mathcal{H}_{\alpha }}$ is symmetric
about zero. Let%
\begin{equation*}
L_{j}=Z\diamond A_{j}\diamond B_{j}\diamond
\end{equation*}%
for $1\leq j\leq m$. Observe that $L_{j}$ is symmetric, $L_{j}^{2}=\mathbf{1}
$, and $L_{j}L_{k}=L_{k}L_{j}$ for all $j,k$. Also $A_{2}\diamond $ is
invertible and anticommutes with $L_{1}$, and $A_{1}\diamond $ anticommutes
with $L_{j}$ for $j>1$. Thus, the dimension of the $+1$ eigenspace of $L_{j}$
is the same as the dimension of the $-1$ eigenspace for $L_{j}$, and there
exists a basis of simulaneous eigenvectors of $\Sigma _{n}=\mathbb{C}%
^{2^{m}} $. Let $\left\{ v_{1},...,v_{2^{m-1}}\right\} $ be the subset of
this basis consisting of $+1$ eigenvectors of $L_{2}$. Since $A_{2}$
commutes with $L_{2}$ and anticommutes with $L_{1}$, the $+1$-eigenspace of $%
L_{2}$ is a direct sum of $+1$ and $-1$ eigenspaces of $L_{1}$ in equal
dimensions. Thus, we may further assume that $\left\{
v_{1},...,v_{2^{m-2}}\right\} $ are $+1$-eigenvectors of $L_{1}$ and that $%
\left\{ v_{2^{m-2}+1},...,v_{2^{m-1}}\right\} $ are $-1$-eigenvectors of $%
L_{1}$. Then $\left\{
v_{1},...,v_{2^{m-1}},iA_{1}v_{1},...,iA_{1}v_{2^{m-1}}\right\} $ provides a
basis of $\mathbb{C}^{2^{m}}$ for which $\left. D\right\vert _{\mathcal{H}%
_{\alpha }}$ corresponds to a block matrix with $2^{m-2}$-dimensional blocks
of the form%
\begin{equation*}
x\left( 
\begin{array}{cccc}
Q+R & 0 & I & 0 \\ 
0 & -Q+R & 0 & I \\ 
I & 0 & Q-R & 0 \\ 
0 & I & 0 & -Q-R%
\end{array}%
\right) ,
\end{equation*}%
where $x$ is a scalar and $Q$ and $R$ are (commuting) diagonal matrices. A
simple argument shows that the characteristic polynomial of such a matrix is
an even function, and thus the spectrum of $\left. D\right\vert _{\mathcal{H}%
_{\alpha }}$ is symmetric about zero and does not contribute to the eta
invariant, if $m\geq 2$. We summarize the results in the following theorem.

\begin{theorem}
On any Heisenberg manifold of dimension greater than $3$, the restriction of
the Dirac operator to any invariant subspace $\mathcal{H}_{\alpha }$ with $%
\alpha \left( Z\right) =0$ has spectrum that is symmetric about $0$.\label%
{HeisGreaterThan3Symmetry}
\end{theorem}

\begin{remark}
No Heisenberg three-manifolds have this property; see (\ref%
{torusSpectrum3dHeisenberg}).
\end{remark}

\subsubsection{Symmetry in the infinite-dimensional irreducible subspaces}

\vspace{1pt}Next, suppose that $\alpha \in \mathfrak{g}^{\ast }$ is chosen
so that $\alpha \left( Z\right) \neq 0$. Let $\mathcal{U}=\mathrm{span}%
\left\{ \overline{u}_{\mathbf{p},\mathbf{\ell }}:\mathbf{p}=\left(
p_{1},...,p_{m}\right) \in \left( \mathbb{Z}_{\geq 0}\right) ^{m},\mathbf{%
\ell }\in \left\{ -1,1\right\} ^{m}\right\} $. Let $L:\mathcal{U}\rightarrow 
\mathcal{U}$ be the linear map defined by 
\begin{equation*}
L\left( \overline{u}_{\mathbf{p},\mathbf{\ell }}\right) =\delta _{\mathbf{%
\ell }}\overline{u}_{\mathbf{p},-\mathbf{\ell }}~~,
\end{equation*}%
where $\delta _{\mathbf{\ell }}=\pm 1$ according to an unspecified formula.
Note that $L^{-1}\left( \overline{u}_{\mathbf{p},\mathbf{\ell }}\right)
=\delta _{\mathbf{\ell }}\delta _{-\mathbf{\ell }}L\overline{u}_{\mathbf{p},%
\mathbf{\ell }}=\delta _{-\mathbf{\ell }}\overline{u}_{\mathbf{p},-\mathbf{%
\ell }}$. Now, we have

\begin{eqnarray*}
L^{-1}DL\overline{u}_{\mathbf{p},\mathbf{\ell }} &=&\delta _{\mathbf{\ell }%
}L^{-1}D\overline{u}_{\mathbf{p},-\mathbf{\ell }} \\
&=&\delta _{\mathbf{\ell }}L^{-1}\left( 
\begin{array}{c}
-\sum_{j}2i\sqrt{\pi d_{j}p_{j}}\ell _{1}\ell _{2}...\ell _{j}\left(
-1\right) ^{j}\overline{u}_{\mathbf{p},-\mathbf{\ell }^{j}} \\ 
+\left( -2\pi \alpha \left( Z\right) \ell _{1}...\ell _{m}\left( -1\right)
^{m}-\frac{1}{4}\sum_{j\leq m}d_{j}\ell _{1}...\widehat{\ell _{j}}...\ell
_{m}\left( -1\right) ^{m-1}\right) \overline{u}_{\mathbf{p},-\mathbf{\ell }}%
\end{array}%
\right) \\
&=&\delta _{\mathbf{\ell }}\left( 
\begin{array}{c}
-\sum_{j}2i\sqrt{\pi d_{j}p_{j}}\ell _{1}\ell _{2}...\ell _{j}\left(
-1\right) ^{j}\delta _{\mathbf{\ell }^{j}}\overline{u}_{\mathbf{p},\mathbf{%
\ell }^{j}} \\ 
+\left( -2\pi \alpha \left( Z\right) \ell _{1}...\ell _{m}\left( -1\right)
^{m}-\frac{1}{4}\sum_{j\leq m}d_{j}\ell _{1}...\widehat{\ell _{j}}...\ell
_{m}\left( -1\right) ^{m-1}\right) \delta _{\mathbf{\ell }}\overline{u}_{%
\mathbf{p},\mathbf{\ell }}%
\end{array}%
\right) \\
&=&\delta _{\mathbf{\ell }}\left( -\sum_{j}2i\left( -1\right) ^{j}\delta _{%
\mathbf{\ell }^{j}}\sqrt{\pi d_{j}p_{j}}\ell _{1}\ell _{2}...\ell _{j}%
\overline{u}_{\mathbf{p},\mathbf{\ell }^{j}}\right) \\
&&+\left( -1\right) ^{m-1}\left( 2\pi \alpha \left( Z\right) \ell
_{1}...\ell _{m}-\frac{1}{4}\sum_{j\leq m}d_{j}\ell _{1}...\widehat{\ell _{j}%
}...\ell _{m}\right) \overline{u}_{\mathbf{p},\mathbf{\ell }}.
\end{eqnarray*}

\begin{remark}
For the case where $m$ is even, we define $\delta _{\mathbf{\ell }}=\left(
\ell _{1}\right) ^{2}\left( \ell _{2}\right) ^{3}...\left( \ell _{m}\right)
^{m+1}$, so that $\delta _{\mathbf{\ell }^{j}}\delta _{\mathbf{\ell }}\left(
-1\right) ^{j}=\left( -1\right) ^{j+1}\left( -1\right) ^{j}=-1$, and thus,
the matrix for $L^{-1}DL$ is the negative of the matrix for $D$ with $\alpha
\left( Z\right) $ replaced by its negative. Thus the spectrum $\sigma
_{\alpha }$ satisfies $\sigma _{-\alpha }=-\sigma _{\alpha }$ if $m$ is
even. The symmetry of the eigenvalues about $0$ follows from the fact that, $%
\alpha $ occurs if and only if $-\alpha $ occurs; see (\ref%
{occurrenceConditionForSpinors}). This confirms in this case the known fact
mentioned at the beginning of this section that the spectrum of the spin
Dirac operator on manifolds of dimension congruent to $1\mod 4$ is symmetric.
\end{remark}

\begin{remark}
For the case where $m$ is odd and the dimension is $2m+1$, we define $\delta
_{\mathbf{\ell }}=\left( \ell _{1}\right) ^{1}\left( \ell _{2}\right)
^{2}...\left( \ell _{m}\right) ^{m}$, so that $\delta _{\mathbf{\ell }%
^{j}}\delta _{\mathbf{\ell }}\left( -1\right) ^{j}=1$; we see in that case
that the matrix for $L^{-1}DL$ is the same as the matrix for $D$ with $%
\alpha \left( Z\right) $ replaced by its negative. Thus the spectrum $\sigma
_{\alpha }$ satisfies $\sigma _{-\alpha }=\sigma _{\alpha }$ if $m$ is odd.
Moreover, the eigenvalues of the Dirac operator need not be symmetric about $%
0$, and the eta invariant need not be zero, as can be seen from the $m=1$
case in Section \ref{EtaInvt3DHeisenSection}. Therefore, for Heisenberg
manifolds of dimension $2m+1$ with $m$ odd and greater than $1$, because of
this remark and Theorem \ref{HeisGreaterThan3Symmetry}, the methods of
Section \ref{etaZetaPerturbedSection} do not apply and cannot be used to
obtain a formula for the eta invariant.
\end{remark}

\section{Example of a five-dimensional non-Heisenberg nilmanifold\label%
{nonHeisenbergExampleSection}}

The purpose of this section is to exhibit an example of a two-step
nilmanifold for which the techniques used above fail to produce the Dirac
eigenvalues as eigenvalues of finite-dimensional matrices. Because this
manifold is $(4(1)+1)$-dimensional, we know a priori that the eta invariant
vanishes. We use the notation of Section \ref{2StepDiracSection} with a
specific class of examples. We have that $k_{0}=2$ is the dimension of the
center $\mathfrak{z} $ and $m_{0}=3$, and we have the orthonormal basis $%
\left\{ Z_{1},Z_{2},X,Y_{1},Y_{2}\right\} $ so that each $Z_{j}$ is a unit
vector and $\left\{ X,Y_{1},Y_{2}\right\} $ is an orthonormal basis of $%
\mathfrak{v} $. The only nontrivial bracket relations are $\left[ X,Y_{1}%
\right] =Z_{1},~\left[ X,Y_{2}\right] =Z_{2}$. From formula (\ref{Dirac2Step}%
), the Dirac operator is%
\begin{equation*}
D=\sum_{i=1}^{5}\rho _{\varepsilon \ast }\left( E_{i}\right) \left(
E_{i}\diamond \right) \mathbf{+}\frac{1}{4}\sum_{i=1,2}Z_{i}\diamond
X\diamond Y_{i}\diamond ,
\end{equation*}%
acting on 
\begin{equation*}
\mathcal{H}=L^{2}\left( \Gamma \diagdown G,G\times _{\varepsilon }\mathbb{C}%
^{4}\right) \cong L_{\varepsilon }^{2}\left( \Gamma \diagdown G\right)
\otimes \Sigma _{5},
\end{equation*}%
which we decompose as follows. For $\alpha \in \mathfrak{g}^{\ast }$, the
subspace $\mathcal{H}_{\alpha }$ of $L^{2}\left( \Gamma \diagdown G,G\times
_{\varepsilon }\mathbb{C}^{k}\right) $ is invariant with respect to $\rho
_{\varepsilon }$ and invariant by $D$. If $\overline{\mathcal{H}_{\alpha }}$
is the irreducible $\rho _{\varepsilon }$-subspace of $L_{\varepsilon
}^{2}\left( \Gamma \diagdown G\right) $ corresponding to the coadjoint orbit
of $\alpha $, we have $\mathcal{H}_{\alpha }\cong \overline{\mathcal{H}%
_{\alpha }}\otimes \Sigma _{n}$. As before, define the symplectic form on $%
\mathfrak{g}$ by $B_{\alpha }\left( u,v\right) =\alpha \left( \left[ u,v%
\right] \right) $, and let $\mathfrak{g}_{\alpha }=\ker B_{\alpha }=\left\{
u\in \mathfrak{g}:B_{\alpha }\left( u,\cdot \right) =0\right\} $, $k_{\alpha
}=\dim \mathfrak{g}_{\alpha }$.

\subsection{\textbf{Finite dimensional }$\overline{\mathcal{H}_{\protect%
\alpha }}$-irreducible\textbf{\ subspaces:} $k_{\protect\alpha }=5$, i.e. $%
\protect\alpha \left( \mathfrak{z}\right) =0$.}

As in (\ref{FiniteDimDiracOperatorFormula}), 
\begin{equation*}
\left. D\right\vert _{\mathcal{H}_{\alpha }}=2\pi i\alpha \left( X\right)
\left( X\diamond \right) +\sum_{j=1,2}2\pi i\alpha \left( Y_{j}\right)
\left( Y_{j}\diamond \right) \mathbf{+}\frac{1}{4}\sum_{i=1,2}Z_{i}\diamond
X\diamond Y_{i}\diamond .
\end{equation*}%
The eigenvalues of $\left. D\right\vert _{\mathcal{H}_{\alpha }}$ are then
the eigenvalues of this constant Hermitian linear transformation.\vspace{1pt}

We make the specific choices of the matrices $\left( E_{j}\diamond \right) $
as in Section \ref{Case2Section}. Note $\Sigma _{n}=\mathbb{C}^{2^{2}}=%
\mathbb{C}^{2}\otimes \mathbb{C}^{2}$. We have%
\begin{eqnarray*}
\left( X\diamond \right) &=&i\mathbf{1}^{\prime }\otimes \mathbf{1}^{\prime
},~\left( Y_{1}\diamond \right) =i\sigma _{1}\otimes \mathbf{1},~\left(
Y_{2}\diamond \right) =i\sigma _{2}\otimes \mathbf{1}, \\
\left( Z_{1}\diamond \right) &=&i\mathbf{1}^{\prime }\otimes \sigma
_{1},~\left( Z_{2}\diamond \right) =i\mathbf{1}^{\prime }\otimes \sigma
_{2},~
\end{eqnarray*}%
Recalling (\ref{CayleyTable}), our matrix is (using basis $v_{1,1}$, $%
v_{-1,1}$, $v_{1,-1}$, $v_{-1,-1}$ )%
\begin{eqnarray*}
\left. D\right\vert _{\mathcal{H}_{\alpha }} &=&2\pi i\alpha \left( X\right)
\left( X\diamond \right) +\sum_{j=1,2}2\pi i\alpha \left( Y_{j}\right)
\left( Y_{j}\diamond \right) \mathbf{+}\frac{1}{4}\sum_{i=1,2}Z_{i}\diamond
X\diamond Y_{i}\diamond \\
&=&-2\pi \alpha \left( X\right) \mathbf{1}^{\prime }\otimes \mathbf{1}%
^{\prime }-\sum_{j=1,2}2\pi \alpha \left( Y_{j}\right) \sigma _{j}\otimes 
\mathbf{1-}\frac{i}{4}\sum_{j=1,2}\left( \mathbf{1}^{\prime }\otimes \sigma
_{j}\right) \left( \mathbf{1}^{\prime }\otimes \mathbf{1}^{\prime }\right)
\left( \sigma _{j}\otimes \mathbf{1}\right) \\
&=&-2\pi \alpha \left( X\right) \mathbf{1}^{\prime }\otimes \mathbf{1}%
^{\prime }-\sum_{j=1,2}2\pi \alpha \left( Y_{j}\right) \sigma _{j}\otimes 
\mathbf{1+}\frac{1}{4}\left( \sigma _{1}\otimes \sigma _{2}-\sigma
_{2}\otimes \sigma _{1}\right)
\end{eqnarray*}%
\begin{equation*}
=\left( 
\begin{array}{cccc}
-2\pi \alpha \left( X\right) & -2\pi \alpha \left( Y_{1}\right) -i2\pi
\alpha \left( Y_{2}\right) & 0 & 0 \\ 
-2\pi \alpha \left( Y_{1}\right) +i2\pi \alpha \left( Y_{2}\right) & 2\pi
\alpha \left( X\right) & \frac{i}{2} & 0 \\ 
0 & -\frac{i}{2} & 2\pi \alpha \left( X\right) & -2\pi \alpha \left(
Y_{1}\right) -i2\pi \alpha \left( Y_{2}\right) \\ 
0 & 0 & -2\pi \alpha \left( Y_{1}\right) +i2\pi \alpha \left( Y_{2}\right) & 
-2\pi \alpha \left( X\right)%
\end{array}%
\right) .
\end{equation*}%
We may then determine that the four eigenvalues of $\left. D\right\vert _{%
\mathcal{H}_{\alpha }}$ are:%
\begin{eqnarray*}
&&\frac{1}{4}\pm \frac{1}{4}\sqrt{64\pi ^{2}\alpha \left( X\right)
^{2}+16\pi \alpha \left( X\right) +64\pi ^{2}\alpha \left( Y_{1}\right)
^{2}+64\pi ^{2}\alpha \left( Y_{2}\right) ^{2}+1}, \\
&&-\allowbreak \frac{1}{4}\pm \frac{1}{4}\sqrt{64\pi ^{2}\alpha \left(
X\right) ^{2}-16\pi \alpha \left( X\right) +64\pi ^{2}\alpha \left(
Y_{1}\right) ^{2}+64\pi ^{2}\alpha \left( Y_{2}\right) ^{2}+1}.
\end{eqnarray*}

Using the $\alpha \mapsto -\alpha $ symmetry, for a typical nilmanifold,
this portion of the spectrum will be symmetric about zero.

\subsection{\textbf{Infinite-dimensional }$\overline{\mathcal{H}_{\protect%
\alpha }}$-irreducible subspaces\textbf{:} $k_{\protect\alpha }<n$, so that $%
\protect\alpha \left( \mathfrak{z}\right) \neq 0$.}

In this case, a typical coadjoint orbit has an element of the form $\alpha
=b_{2}Y_{2}^{\ast }+g_{1}Z_{1}^{\ast }+g_{2}Z_{2}^{\ast }$, with $g_{1}$, $%
g_{2}$ not both zero.

Choose a new orthonormal basis of $\mathfrak{g}$:%
\begin{equation*}
\left\{ W_{1}=\frac{g_{1}Z_{1}+g_{2}Z_{2}}{\sqrt{g_{1}^{2}+g_{2}^{2}}},W_{2}=%
\frac{-g_{2}Z_{1}+g_{1}Z_{2}}{\sqrt{g_{1}^{2}+g_{2}^{2}}},W_{3}=\frac{%
-g_{2}Y_{1}+g_{1}Y_{2}}{\sqrt{g_{1}^{2}+g_{2}^{2}}},U=X,V=\frac{%
g_{1}Y_{1}+g_{2}Y_{2}}{\sqrt{g_{1}^{2}+g_{2}^{2}}}\right\} ,
\end{equation*}%
where $\left\{ W_{1},W_{2},W_{3}\right\} $ is a basis of $\mathfrak{g}%
_{\alpha }$, $B_{\alpha }\left( U,U\right) =B_{\alpha }\left( V,V\right) =0$%
, and 
\begin{equation*}
d:=B_{\alpha }\left( U,V\right) =\alpha \left( \left[ U,V\right] \right)
=\alpha \left( \left[ X,\frac{g_{1}Y_{1}+g_{2}Y_{2}}{\sqrt{%
g_{1}^{2}+g_{2}^{2}}}\right] \right) =\sqrt{g_{1}^{2}+g_{2}^{2}}.
\end{equation*}%
Then the polarizing subalgebra $\mathfrak{g}^{\alpha }$ will be chosen to be%
\begin{equation*}
\mathfrak{g}^{\alpha }=\mathrm{span}\left\{ V,W_{1},W_{2},W_{3}\right\} ,
\end{equation*}%
and again $G^{\alpha }:=\exp \left( \mathfrak{g}^{\alpha }\right) $.

Equation (\ref{DOnH_alpha_No_Cjs}) becomes 
\begin{equation*}
\left. D\right\vert _{\mathcal{H}_{\alpha }}=\left( U\diamond \right) \frac{%
\partial }{\partial t}+2\pi id\left( V\diamond \right) t+M_{\alpha }^{\prime
},
\end{equation*}%
where $M_{\alpha }^{\prime }$ is the constant Hermitian matrix (using $%
X_{1}=U,X_{2}=V,X_{3}=W_{3},W_{1},W_{2},k_{0}=2,m_{0}=3,k_{\alpha }=3,m=1$)%
\begin{equation*}
M_{\alpha }^{\prime }=\frac{1}{4}\sum_{a\leq 2;~b<i\leq 3}\left\langle W_{a},%
\left[ X_{b},X_{i}\right] \right\rangle \left( W_{a}\diamond X_{b}\diamond
X_{i}\diamond \right) +\sum_{j=1}^{3}2\pi i\alpha \left( W_{j}\right) \left(
W_{j}\diamond \right) ,
\end{equation*}%
from (\ref{M_definition1}), (\ref{M_alpha_prime_definition}).

We calculate 
\begin{equation*}
\left[ X_{1},X_{2}\right] =\left[ U,V\right] =W_{1},~\left[ X_{1},X_{3}%
\right] =W_{2},~\left[ X_{2},X_{3}\right] =0,
\end{equation*}%
so that%
\begin{eqnarray*}
M &=&\frac{1}{4}\sum_{a\leq 2;~b<i\leq 3}\left\langle W_{a},\left[
X_{b},X_{i}\right] \right\rangle \left( W_{a}\diamond X_{b}\diamond
X_{i}\diamond \right) \\
&=&\frac{1}{4}W_{1}\diamond X_{1}\diamond X_{2}\diamond +\frac{1}{4}%
W_{2}\diamond X_{1}\diamond X_{3}\diamond .
\end{eqnarray*}

Again we make the specific choices of the matrices $\left( E_{j}\diamond
\right) $ as in Section \ref{matrixChoicesSection}, with

\begin{eqnarray*}
\left( X_{1}\diamond \right) &=&i\sigma _{1}\otimes \mathbf{1},~\left(
X_{2}\diamond \right) =i\sigma _{2}\otimes \mathbf{1},~\left( X_{3}\diamond
\right) =i\mathbf{1}^{\prime }\otimes \mathbf{1}^{\prime }, \\
\left( W_{1}\diamond \right) &=&i\mathbf{1}^{\prime }\otimes \sigma
_{1},~\left( W_{2}\diamond \right) =i\mathbf{1}^{\prime }\otimes \sigma _{2},
\end{eqnarray*}%
Then, since $\alpha \left( W_{2}\right) =0$, $\alpha \left( X_{3}\right) =%
\frac{g_{1}b_{2}}{d}$,$\alpha \left( W_{1}\right) =d=\sqrt{%
g_{1}^{2}+g_{2}^{2}}$, we use (\ref{CayleyTable}) to obtain 
\begin{eqnarray*}
M_{\alpha }^{\prime } &=&\frac{1}{4}W_{1}\diamond X_{1}\diamond
X_{2}\diamond +\frac{1}{4}W_{2}\diamond X_{1}\diamond X_{3}\diamond +2\pi
i\alpha \left( W_{1}\right) W_{1}\diamond \\
&&+2\pi i\alpha \left( X_{3}\right) \left( X_{3}\diamond \right) \\
&=&\frac{1}{4}\left( i\mathbf{1}^{\prime }\otimes \sigma _{1}\right) \left(
i\sigma _{1}\otimes \mathbf{1}\right) \left( i\sigma _{2}\otimes \mathbf{1}%
\right) +\frac{1}{4}\left( i\mathbf{1}^{\prime }\otimes \sigma _{2}\right)
\left( i\sigma _{1}\otimes \mathbf{1}\right) \left( i\mathbf{1}^{\prime
}\otimes \mathbf{1}^{\prime }\right) +2\pi id\left( i\mathbf{1}^{\prime
}\otimes \sigma _{1}\right) \\
&&+2\pi i\frac{g_{1}b_{2}}{d}\left( i\mathbf{1}^{\prime }\otimes \mathbf{1}%
^{\prime }\right) \\
&=&-\frac{1}{4}\mathbf{1}\otimes \sigma _{1}+\frac{1}{4}\sigma _{1}\otimes
\sigma _{1}-2\pi d\mathbf{1}^{\prime }\otimes \sigma _{1}-2\pi \frac{%
g_{1}b_{2}}{d}\mathbf{1}^{\prime }\otimes \mathbf{1}^{\prime }.
\end{eqnarray*}

We need to determine what $M_{\alpha }^{\prime }$ does to the basis $\left\{ 
\overline{u}_{p,\mathbf{\ell }}\right\} $. We have%
\begin{eqnarray*}
\overline{u}_{p,\mathbf{\ell }} &=&\left\{ 
\begin{array}{ll}
u_{p,\mathbf{\ell }} & \text{if }\ell _{1}=0 \\ 
\sqrt{2p}u_{p-1,\mathbf{\ell }}~ & \text{if }\ell _{1}=-1%
\end{array}%
\right. , \\
u_{p,\mathbf{\ell }} &=&h_{p}\left( \sqrt{2\pi d}t\right) v_{\mathbf{\ell }}
\end{eqnarray*}%
Then%
\begin{eqnarray*}
\left( \mathbf{1}\otimes \sigma _{1}\right) \overline{u}_{p,\mathbf{\ell }}
&=&\left( \mathbf{1}\otimes \sigma _{1}\right) \left\{ 
\begin{array}{ll}
u_{p,\mathbf{\ell }} & \text{if }\ell _{1}=0 \\ 
\sqrt{2p}u_{p-1,\mathbf{\ell }}~ & \text{if }\ell _{1}=-1%
\end{array}%
\right. \\
&=&\left\{ 
\begin{array}{ll}
u_{p,\mathbf{\ell }^{2}} & \text{if }\ell _{1}=0 \\ 
\sqrt{2p}u_{p-1,\mathbf{\ell }^{2}}~ & \text{if }\ell _{1}=-1%
\end{array}%
\right. \\
&=&\overline{u}_{p,\mathbf{\ell }^{2}}.
\end{eqnarray*}%
\begin{eqnarray*}
\left( \sigma _{1}\otimes \sigma _{1}\right) \overline{u}_{p,\mathbf{\ell }}
&=&\left( \sigma _{1}\otimes \sigma _{1}\right) \left\{ 
\begin{array}{ll}
u_{p,\mathbf{\ell }} & \text{if }\ell _{1}=0 \\ 
\sqrt{2p}u_{p-1,\mathbf{\ell }}~ & \text{if }\ell _{1}=-1%
\end{array}%
\right. \\
&=&\left\{ 
\begin{array}{ll}
u_{p,-\mathbf{\ell }} & \text{if }\ell _{1}=0 \\ 
\sqrt{2p}u_{p-1,-\mathbf{\ell }}~ & \text{if }\ell _{1}=-1%
\end{array}%
\right. \\
&=&\left( \sqrt{2p}\right) ^{-\ell _{1}}\overline{u}_{p+\ell _{1},-\mathbf{%
\ell }}
\end{eqnarray*}%
\begin{eqnarray*}
\left( \mathbf{1}^{\prime }\otimes \sigma _{1}\right) \overline{u}_{p,%
\mathbf{\ell }} &=&\ell _{1}\left( \mathbf{1}\otimes \sigma _{1}\right) 
\overline{u}_{p,\mathbf{\ell }} \\
&=&\ell _{1}\overline{u}_{p,\mathbf{\ell }^{2}}
\end{eqnarray*}%
\begin{equation*}
\left( \mathbf{1}^{\prime }\otimes \mathbf{1}^{\prime }\right) \overline{u}%
_{p,\mathbf{\ell }}=\ell _{1}\ell _{2}\overline{u}_{p,\mathbf{\ell }}
\end{equation*}%
Substituting,%
\begin{eqnarray*}
M_{\alpha }^{\prime }\overline{u}_{p,\mathbf{\ell }} &=&\left( -\frac{1}{4}%
\mathbf{1}\otimes \sigma _{1}+\frac{1}{4}\sigma _{1}\otimes \sigma _{1}-2\pi
d\mathbf{1}^{\prime }\otimes \sigma _{1}-2\pi \frac{g_{1}b_{2}}{d}\mathbf{1}%
^{\prime }\otimes \mathbf{1}^{\prime }\right) \overline{u}_{p,\mathbf{\ell }}
\\
&=&-\frac{1}{4}\overline{u}_{p,\mathbf{\ell }^{2}}+\frac{1}{4}\left( \sqrt{2p%
}\right) ^{-\ell _{1}}\overline{u}_{p+\ell _{1},-\mathbf{\ell }} \\
&&-2\pi d\ell _{1}\overline{u}_{p,\mathbf{\ell }^{2}}-2\pi \frac{g_{1}b_{2}}{%
d}\ell _{1}\ell _{2}\overline{u}_{p,\mathbf{\ell }} \\
&=&\left( -\frac{1}{4}-2\pi d\ell _{1}\right) \overline{u}_{p,\mathbf{\ell }%
^{2}}+\frac{1}{4}\left( \sqrt{2p}\right) ^{-\ell _{1}}\overline{u}_{p+\ell
_{1},-\mathbf{\ell }}-2\pi \frac{g_{1}b_{2}}{d}\ell _{1}\ell _{2}\overline{u}%
_{p,\mathbf{\ell }}.
\end{eqnarray*}

From (\ref{DequationWithSQRoots}), we have%
\begin{eqnarray*}
D\overline{u}_{p,\mathbf{\ell }} &=&-2i\sqrt{\pi dp}\ell _{1}\overline{u}_{p,%
\mathbf{\ell }^{1}}+M_{\alpha }^{\prime }\overline{u}_{p,\mathbf{\ell }} \\
&=&-2i\sqrt{\pi dp}\ell _{1}\overline{u}_{p,\mathbf{\ell }^{1}}+\left( -%
\frac{1}{4}-2\pi d\ell _{1}\right) \overline{u}_{p,\mathbf{\ell }^{2}} \\
&&+\frac{1}{4}\left( \sqrt{2p}\right) ^{-\ell _{1}}\overline{u}_{p+\ell
_{1},-\mathbf{\ell }}-2\pi \frac{g_{1}b_{2}}{d}\ell _{1}\ell _{2}\overline{u}%
_{p,\mathbf{\ell }}~.
\end{eqnarray*}%
There are no apparent invariant subspaces for $D$ spanned by a finite number
of the $\overline{u}_{p,\mathbf{\ell }}$ . The matrix for $D$ is an infinite
band matrix. This shows the difficulty of computing the Dirac eigenvalues
for a general nilmanifold.

\section{Appendix: CCMoore/LenRichardson Papers and Adaptations\label%
{MooreRichardsonPapersSection}}

\subsubsection{Occurrence and Multiplicity Condition}

Let $\Gamma $ be a cocompact (i.e., $\Gamma \diagdown G$ compact) discrete
subgroup of the simply connected nilpotent Lie group $G$. Let $\varepsilon
:\Gamma \rightarrow \left\{ \pm 1\right\} \subset GL\left( \mathbb{C}%
^{k}\right) \ $\ be a homomorphism. Denote by $U_{\varepsilon }$ the
representation of $G$ induced by $\varepsilon $; in particular, 
\begin{equation*}
U_{\varepsilon }=L_{\varepsilon }^{2}\left( \Gamma \diagdown G\right)
=\left\{ f:G\rightarrow \mathbb{C}^{k}:f\left( \gamma g\right) =\varepsilon
\left( \gamma \right) f\left( g\right) \text{ for all }g\in G,\gamma \in
\Gamma \right\} \text{,}
\end{equation*}%
where the (left)\ action of $G$ on $U_{\varepsilon }$ \ is given by interior
right multiplication. Note that if $\varepsilon =1,$ then $U_{\varepsilon }=$
$L_{\varepsilon }^{2}\left( \Gamma \diagdown G\right) $ is the direct sum of 
$k$ copies of the quasi-regular representation $U=L^{2}\left( \Gamma
\diagdown G\right) $. As in the quasi-regular case, standard results in
representation theory imply in general that $U_{\varepsilon }$ can be
decomposed into the direct sum of irreducible representations of $G$, each
with finite multiplicity. A good reference for the standard representation
theory used in this appendix is \cite{CoGr}.

Our motivation for this construction is that spin structures over
nilmanifolds $\Gamma \diagdown G$ correspond exactly to homomorphisms $%
\varepsilon :\Gamma \rightarrow GL\left( \mathbb{C}^{k}\right) $, where the
image of $\varepsilon $ lies in the set $\left\{ \pm 1\right\} ,$ and $%
k=2^{\left\lfloor n/2\right\rfloor }$. The resulting spinor bundle is 
\begin{equation*}
\Sigma _{\varepsilon }=G\times _{\varepsilon }\mathbb{C}^{k}=G\times \mathbb{%
C}^{k}\diagup \left\{ \left( g,v\right) :\left( g,v\right) =\left( \gamma
g,\varepsilon \left( \gamma \right) v\right) \text{ for all }\gamma \in
\Gamma \right\} ;
\end{equation*}
see \cite[Prop 3.34, p. 114]{Be-G-V}. The sections of this bundle are
elements of $U_{\varepsilon }$, on which the Dirac operator acts.

In the quasi-regular case ($\varepsilon =1$), L. Richardson and R. Howe,
building on work of C. C. Moore, independently proved an exact occurrence
condition and multiplicity formula; they determined the irreducible
representations $\pi $ of $G$ that occur in $U=L^{2}\left( \Gamma \diagdown
G\right) $ and the corresponding multiplicities $m\left( \pi ,U\right) $.
The purpose of this appendix is to generalize their occurrence and
multiplicity formula from the quasi-regular to the case of general $%
\varepsilon $.

Before stating the main results, we require the following definitions and
observations.

Denote by $\widehat{G}$ the set of equivalence classes of irreducible
unitary representations of $G$. The Kirillov Correspondence is the bijection
between the set of orbits of the co-adjoint action of $G$ on $\mathfrak{g}%
^{\ast }$ and $\widehat{G}$. In particular, Kirillov Theory proves that to
each $\alpha \in \mathfrak{g}^{\ast }$ corresponds an irreducible unitary
representation $\pi _{\alpha }$ of $G$, every irreducible representation of $%
G$ is unitarily equivalent to such a $\pi _{\alpha }$, and two such
irredicuble unitary representations $\pi _{\alpha }$ and $\pi _{\alpha
^{\prime }}$ are unitarily equivalent if and only if $\alpha ^{\prime
}=\alpha \circ $\textrm{$Ad$}$\left( x^{-1}\right) $ for some $x\in G$.
Kirillov Theory applies mainly to nilpotent Lie groups, with generalizations
to some solvable groups.

Choose $\alpha \in \mathfrak{g}^{\ast }$ and let $\mathfrak{h}$ be any
subalgebra of $\mathfrak{g}$. Let $H=\exp \left( \mathfrak{h}\right) $ be
the unique simply connected Lie subgroup of $G$ with Lie algebra $\mathfrak{h%
}$. The subalgebra $\mathfrak{h}$ or the subgroup $H$ is \emph{subordinate
to }$\alpha $ iff $\alpha \left( \left[ \mathfrak{h,h}\right] \right) \equiv
0$. If in addition $\mathfrak{h}$ is maximal with respect to being
subordinate, then $\mathfrak{h}$ is called a \emph{maximal subordinate
subalgebra} for $\alpha $, or a \emph{polarizer} for $\alpha $.

The explicit mapping between elements of $\mathfrak{g}^{\ast }$ and $%
\widehat{G}$ is as follows. Since $G$ is nilpotent and simply connected, the
exponential map is a diffeomorphism with inverse $\log $. For $\alpha \in 
\mathfrak{g}^{\ast }$, let $\mathfrak{h}$ be a maximal subordinate
subalgebra of $\alpha $. Define $\overline{\alpha }\left( \cdot \right)
=e^{2\pi i\alpha \left( \log \left( \cdot \right) \right) }$, which is a
character on $H$ --- i.e., a (complex)\ one-dimensional representation. The
irreducible unitary representation $\pi _{\alpha }$ is the representation of 
$G$ induced by the representation $\overline{\alpha }$ of $H.$

Recall that we have fixed a cocompact, discrete subgroup $\Gamma $ of $G$.
We call the pair $\left( \overline{\alpha },H\right) $ \emph{rational} (with
respect to $\Gamma $)\ if it can be constructed with respect to a rational
covector $\alpha $, i.e. $\alpha \left( \log \Gamma \right) \subset \mathbb{Q%
}$.We call the pair $\left( \overline{\alpha },H\right) $ a\emph{\ special
maximal pair} if $\log H=\mathfrak{h}$ is a maximal subordinate subalgebra
for $\alpha $ that is special in the sense that it is algorithmically and
inductively constructed from $\alpha $ and $\Gamma $ as described in \cite[%
pp. 176-178]{Ri}. As Kirillov theory dictates that the representation $\pi
_{\alpha }$ is independent of the maximal subordinate subalgebra (up to
unitary equivalence), and as Richardson's paper shows that any covector $%
\alpha $ has a special maximal subordinate subalgebra, this additional
property is not a restriction.We call $\left( \overline{\alpha },H\right) $
an $\varepsilon $\emph{-integral point} if and only if for all $\gamma \in
\Gamma \cap H$, $\overline{\alpha }\left( \gamma \right) =e^{2\pi i\alpha
\left( \log \left( \gamma \right) \right) }=\varepsilon \left( \gamma
\right) $. The equivalent condition on the Lie algebra level is, for all $%
\gamma \in \Gamma \cap H$, 
\begin{equation*}
\alpha \left( \log \gamma \right) \in \left\{ 
\begin{array}{ll}
\mathbb{Z} & \text{if~}\varepsilon \left( \gamma \right) =1 \\ 
\mathbb{Z+}\frac{1}{2} & \text{if~}\varepsilon \left( \gamma \right) =-1%
\end{array}%
\right. .
\end{equation*}

Let $\pi \in \widehat{G}$ be induced from $\alpha \in \mathfrak{g}^{\ast }$
under the Kirillov correspondence. Let $F$ be the family of special maximal
characters of $\pi $, that is all possible pairs $\left( \overline{\alpha }%
,H\right) $ that induce $\pi $ with $\mathfrak{h}=\log \left( H\right) $ a
special maximal subordinate subalgebra. Now$\ $L. Richardson proved that $x$ 
$\in G$ acts on $F$ via%
\begin{equation*}
\left( \overline{\alpha },H\right) \cdot x=\left( \overline{\alpha }%
^{x},^{x^{-1}}H\right) ,
\end{equation*}%
where $I_{x}$ denotes conjugation by $x$, the function $\overline{\alpha }%
^{x}=\overline{\alpha }\circ I_{x}$, and $^{x^{-1}}H=x^{-1}Hx=I_{x^{-1}}%
\left( H\right) $. Note that $\left( \overline{\alpha },H\right) \cdot x$ is
an $\varepsilon $-integral point if and only if%
\begin{equation*}
\overline{\alpha }^{x}\left( \gamma \right) =\varepsilon \left( \gamma
\right)
\end{equation*}%
for all $\gamma \in \Gamma \cap \left( ^{x^{-1}}H\right) $ if and only if $%
\overline{\alpha }\left( \gamma \right) =\varepsilon \left( \gamma \right) $
for all $\gamma \in \left( ^{x^{-1}}\Gamma \right) \cap H$.

We may now state the following main results of this Appendix.

\begin{theorem}
\label{occurInL2EpsTheorem}If $\pi $ is induced by the special maximal
character $\left( \overline{\alpha },H\right) $ under the Kirillov
correspondence, then $m\left( \pi ,L_{\varepsilon }^{2}\left( \Gamma
\diagdown G\right) \right) >0$ if and only if there is an $\varepsilon $%
-integral point in \ the orbit $\left( \overline{\alpha },H\right) \cdot G$.
\end{theorem}

\begin{lemma}
Assume that $m\left( \pi ,L_{\varepsilon }^{2}\left( \Gamma \diagdown
G\right) \right) >0$, and let the special maximal character $\left( 
\overline{\alpha },H\right) $ induce $\pi $ under the Kirillov
correspondence. The action satisifes $\left( \overline{\alpha },H\right)
\cdot x=\left( \overline{\alpha },H\right) $, iff $x\in H$, so that we may
identify the $G$-orbit of $\left( \overline{\alpha },H\right) $ with $%
H\diagdown G$. \ If $\left( \overline{\alpha },H\right) $ is an $\varepsilon 
$-integral point and if $\gamma _{0}\in \Gamma $, then $\left( \overline{%
\alpha },H\right) \cdot \gamma _{0}$ is also an $\varepsilon $-integral
point.
\end{lemma}

Let $\left( H\diagdown G\right) _{\varepsilon }$ be the set of $\varepsilon $%
-integral points in $H\diagdown G$. As a result of the Lemma, $\Gamma $ acts
on $\left( H\diagdown G\right) _{\varepsilon }$.

\begin{theorem}
\label{multiplicityAndOccurInL2EpsTheorem}If the special maximal character $%
\left( \overline{\alpha },H\right) $ induces $\pi $ under the Kirillov
correspondence, then the multiplicity of $\pi $ in the $\varepsilon $-quasi
regular representation $U_{\varepsilon }=L_{\varepsilon }^{2}\left( \Gamma
\diagdown G\right) $, denoted $m\left( \pi ,L_{\varepsilon }^{2}\left(
\Gamma \diagdown G\right) \right) $, is the number of $\Gamma $-orbits in
the set $\left( H\diagdown G\right) _{\varepsilon }$ of $\varepsilon $%
-integral points in the $G$-orbit $H\diagdown G$ of $\left( \overline{\alpha 
},H\right) $. That is, 
\begin{equation*}
m\left( \pi ,L_{\varepsilon }^{2}\left( \Gamma \diagdown G\right) \right)
=\#\left( \left( H\diagdown G\right) _{\varepsilon }\diagup \Gamma \right) .
\end{equation*}
\end{theorem}

\subsubsection{Proof of Occurrence and Multiplicity}

The proofs of the Lemma and Theorems follows the outline in L. Richardson's
paper closely. We verify that a few key Lemmas of C. C. Moore extend to the $%
\varepsilon $-quasi-regular setting, and from there the proof primarily
follows that of L.. Richardson verbatim, after substuting our Lemmas for
those of Moore, and replacing \textquotedblleft integral
point\textquotedblright\ with \textquotedblleft $\varepsilon $-integral
point.\textquotedblright

\vspace{1pt}For any $\pi \in \widehat{G}$, suppose there exists $\gamma
_{0}\in \Gamma $ such that $N=\exp \left( \mathbb{R}\log \left( \gamma
_{0}\right) \right) $ is a one-dimensional rational normal subgroup of $G$
and $\pi \left( N\right) =1$. Let $\varphi $ be the natural projection of $G$
onto $G_{0}=G\diagup N$, so $\Gamma _{0}=\Gamma \cdot N\diagup N=\varphi
\left( \Gamma \right) $ is a cocompact discrete subgroup of $G_{0}$. Then
the representation $\pi _{0}$ of $G_{0}$ defined by $\pi _{0}\left( \varphi
\left( g\right) \right) =\pi \left( g\right) $ is well-defined and
irreducible, hence an element of $\widehat{G_{0}}$ (see \cite[Lemma 2.1]{Moo}%
).

\begin{lemma}
Generalized $\varepsilon $-Reduction Lemma. (generalization of \cite[Lemma
2.2]{Moo}, quoted as \cite[Lemma 2.6]{Ri})\label{ReductionLemma}

Note that $\varepsilon $ induces a homomorphism of $\Gamma _{0}$ iff $%
\varepsilon \left( \gamma _{0}\right) =1$. With notation as above, denote by 
$U_{0\varepsilon }$ the representation of $G_{0}$ induced by the $%
\varepsilon $-homomorphism of $\Gamma _{0}$, if it exists. If $m\left( \pi
,U_{\varepsilon }\right) \neq 0$ then $\varepsilon \left( \gamma _{0}\right)
=1$, and the multiplicity $m\left( \pi ,U_{\varepsilon }\right) $ of $\pi $
in $U_{\varepsilon }$ is equal to the multiplicity $m\left( \pi
_{0},U_{0\varepsilon }\right) $ of $\pi _{0}$ in $U_{0\varepsilon }$.
\end{lemma}

\begin{proof}
By normality, $N\subset Z\left( G\right) $. This follows from the
Campbell-Baker-Hausdorff formula, since for vectors $A\in \log G$, and $X\in
\log N$, we have $\exp \left( A\right) \exp \left( X\right) \exp \left(
A^{-1}\right) =\exp \left( rX\right) =\exp \left( X+\left[ A,X\right] +c_{2}%
\left[ A,\left[ A,X\right] \right] +...\right) $. Let $\mathrm{ad}\left(
A\right) ^{k}\left( X\right) $ be the first zero element of the sequence $%
\left( X,\left[ A,X\right] ,\left[ A,\left[ A,X\right] \right] ,...\right) $%
. Because $G$ is nilpotent, \linebreak $\left\{ X,\left[ A,X\right] ,\left[
A,\left[ A,X\right] \right] ,...,\mathrm{ad}\left( A\right) ^{k-1}\left(
X\right) \right\} $ is linearly independent. Since $rX=X+\left[ A,X\right]
+c_{2}\left[ A,\left[ A,X\right] \right] +...$, we have $\left[ A,X\right]
=0 $. Note that since $\pi \left( N\right) =1$, if $m\left( \pi
,U_{\varepsilon }\right) \neq 0$, $U_{\varepsilon }\left( n\right) f=f$ for
all $n$ in $N$ and $f$ in the corresponding invariant subspace $\mathcal{H}%
_{\pi }$. This means that while $f\left( \gamma g\right) =\varepsilon \left(
\gamma \right) f\left( g\right) $ for all $g\in G$, $\gamma \in \Gamma $, it
must also be true that $\left( U_{\varepsilon }\left( n\right) f\right)
\left( g\right) =f\left( gn\right) =f\left( ng\right) =f\left( g\right) $
for all $n\in N$. If $n\in \Gamma \cap N$, then in addition we have $f\left(
gn\right) =f\left( ng\right) =f\left( g\right) =\varepsilon \left( n\right)
f\left( g\right) $, which implies that $\varepsilon \left( \gamma \right) $
is the identity and $\varepsilon $ induces a homomorphism of $\Gamma _{0}$.
Thus $\varepsilon $ restricted to $\Gamma \cap N$ acts trivially on the
image of the sections of $\mathcal{H}_{\pi }$. Let $M=\Gamma \diagdown G$.
We can project $M$ onto $M_{0}=\Gamma _{0}\diagdown G_{0}$, and $M$ becomes
a fiber bundle over $M_{0}$ with circle $T\cong \left( \Gamma \cap N\right)
\diagdown N$ as fiber. Let 
\begin{equation*}
\mathcal{H}^{N}=\left\{ f:G\rightarrow \mathbb{C}^{k}:f\left( \gamma
g\right) =\varepsilon \left( \gamma \right) f\left( g\right) \text{ for all }%
\gamma \in \Gamma \text{ and }f\left( gn\right) =f\left( g\right) \text{ for
all }n\in N\right\} ,
\end{equation*}%
which is the set of sections on $M$ that are constant on the fibers of $%
M\rightarrow M_{0}$, i.e. such that $U_{\varepsilon }\left( n\right) f=f$
for all $n\in N$. This is an invariant subspace of $U_{\varepsilon }$,
because for such $f$, $U_{\varepsilon }\left( n\right) U_{\varepsilon
}\left( g\right) f=U_{\varepsilon }\left( g\right) U_{\varepsilon }\left(
n\right) f=U_{\varepsilon }\left( g\right) f$ for all $g\in G$, $n\in N$.
The projection of the space of all sections onto $\mathcal{H}^{N}$ lies in
the center of the commuting algebra of $U_{\varepsilon }$; that is, the
projection of $U_{\varepsilon }$ onto an invariant subspace must commute
everything that commutes with $U_{\varepsilon }$, because if $L$ commutes
with $U_{\varepsilon }$, then $\mathcal{H}^{N}$ is also an invariant
subspace of $L$, and thus the projection onto $\mathcal{H}^{N}$ commutes
with $L$. Let $U_{N}$ be the restriction of $U_{\varepsilon }$ to $\mathcal{H%
}^{N}$, and we define $U_{N_{0}}\left( \varphi \left( n\right) \right)
=U_{N}\left( n\right) $, the corresponding representation of $G_{0}$. Using
the realization of $U_{N_{0}}$ on sections of $M$ that are constant on the
fibers, we can also realize $U_{N_{0}}$ on the space $L^{2}\left(
M_{0},\Sigma _{\varepsilon }\right) $. It is clear that $U_{N_{0}}$ is
equivalent to $U_{0\varepsilon }$. We also have $m\left( \pi ,U_{\varepsilon
}\right) =m\left( \pi ,U_{N}\right) $, since $\pi $ is trivial on $N$, and $%
m\left( \pi ,U_{N}\right) =m\left( \pi _{0},U_{N_{0}}\right) =m\left( \pi
_{0},U_{0\varepsilon }\right) $, as desired.
\end{proof}

\begin{lemma}
(Pukansky, as stated in \cite[Lemma 2.2]{Ri}) Let $\mathfrak{g}$ be a
nilpotent Lie algebra with one dimensional center $\mathfrak{z=}\mathbb{%
\mathbb{R}
}Z_{1}$, with $G$ and $\Gamma $ as above. Then $\mathfrak{g=}\mathbb{%
\mathbb{R}
}X_{1}\oplus \mathbb{%
\mathbb{R}
}Y_{1}\oplus \mathbb{%
\mathbb{R}
}Z_{1}\oplus \mathfrak{g}^{\prime }$, where $\left[ X_{1},Y_{1}\right]
=Z_{1} $. Let $\mathfrak{g}_{1}\mathfrak{=}\mathbb{%
\mathbb{R}
}Y_{1}\oplus \mathbb{%
\mathbb{R}
}Z_{1}\oplus \mathfrak{g}^{\prime }=\left\{ X\in \mathfrak{g:}\left[ X,Y_{1}%
\right] =0\right\} $. The elements $Y_{1},Z_{1}$ may be chosen to lie in $%
\log \Gamma $; i.e., $\mathfrak{g}_{1}$ may be chosen to be rational with
respect to the cocompact discrete subgroup $\Gamma $ of $G$.
\end{lemma}

\begin{theorem}
\vspace{1pt}Kirillov's Theorem (as quoted in \cite[Theorem 2.3]{Ri})

If $G$ has one dimensional center, then every irreducible representation $%
\pi $ of $G$ such that $\pi $ is non-constant on the center is induced by a
necessarily irreducible representation of $G_{1}=\exp \left( \mathfrak{g}%
_{1}\right) $, with $\mathfrak{g}_{1}$ as in the previous Lemma.
\end{theorem}

\begin{definition}
We call the subgroup $G_{1}$ from the previous theorem a \emph{Kirillov
subgroup.}
\end{definition}

\begin{theorem}
$\varepsilon $-Generalized Moore's Algorithm (generalization of Moore's
Algorithm, quoted as \cite[Moore's Algorithm 2.7]{Ri}).\label%
{MooreAlgTheorem}\newline
Let $\pi $ be an irreducible representation of $G$, where $G$ has a one
dimensional center \vspace{1pt}$Z\left( G\right) $, and $\pi |_{Z\left(
G\right) }\neq \mathrm{id}$. Let $\pi _{1}$ be an irreducible representation
of $G_{1}$, a rational Kirillov subgroup of $G$ having codimension one, such
that $\pi _{1}$ induces $\pi $. Define $\pi _{1}^{x}\left( x_{1}\right) =\pi
_{1}\left( xx_{1}x^{-1}\right) $ for $x$ in $G$ and $x_{1}$ in $G_{1}$. Let $%
U_{1\varepsilon }$ be defined for $G_{1}$ and $\Gamma _{1}=\Gamma \cap G_{1}$
as $U_{\varepsilon }$ is defined for $G$ and $\Gamma $. Let $\widehat{G_{1}}$
denote the dual space of equivalence classes of unitary irreducible
representations of $G_{1}$. Let $A^{\prime }=\left\{ \rho _{1}\in \widehat{%
G_{1}}:m\left( \rho _{1},U_{1\varepsilon }\right) >0\text{ and }\rho
_{1}|_{Z\left( G\right) }\neq Id\right\} $. For all $\gamma \in \Gamma $,
since $G_{1}$ is normal and $\gamma \Gamma \gamma ^{-1}=\Gamma
,U_{1\varepsilon }^{\gamma }\cong U_{1\varepsilon }$. Thus $m\left( \rho
_{1}^{\gamma },U_{1\varepsilon }\right) =m\left( \rho _{1},U_{1\varepsilon
}\right) $, or $\left\{ \rho _{1}^{\gamma }:\rho _{1}\in A^{\prime }\right\}
=A^{\prime }$.Let $A$ be a subset of $A^{\prime }$ that meets each $\Gamma $%
-orbit in $A^{\prime }$ in exactly one element. Then 
\begin{equation*}
m\left( \pi ,U_{\varepsilon }\right) =\sum_{\rho _{1}\in \pi _{1}^{G}\cap
A}m\left( \rho _{1},U_{1\varepsilon }\right) \text{.}
\end{equation*}
\end{theorem}

\begin{proof}
The proof closely follows that in \cite[pp. 151--153]{Moo}.

Let $Z_{2}\left( G\right) $ be the subgroup of $G$ such that $Z_{2}\left(
G\right) \diagup Z\left( G\right) $ is the center of $G\diagup Z\left(
G\right) $. The group $Z_{2}\left( G\right) $ is a rational subgroup of $G$
(with respect to any lattice), see for example \cite[Chapter 5]{CoGr}, and
we may choose a rational subgroup $W$ of $Z_{2}\left( G\right) $ (and $G$)
of dimension $2$ that contains $Z\left( G\right) $. The centralizer $G_{0}$
of $W$ then has codimension $1$ in $G$ and is a rational normal subgroup
(see \cite[Lemma 2.2]{Ri}, quoted from \cite{Kir1}). Finally, since $G_{0}$
is codimension one and normal, we can find a rational one-parameter group $S$
such that $G=G_{0}\rtimes S$.

We now use the following, whose proof can be found in any book on Kirillov
theory. Denote by $U_{0\varepsilon }$the representation of $G_{0}$ induced
by the homomorphism $\varepsilon .$

\begin{lemma}
If $\pi \in \widehat{G}$ and if $\pi $ is nontrivial on $Z\left( G\right) $,
then $\pi $ is induced by some $\pi _{0}\in \widehat{G_{0}}$. The set of all
representations of $G_{0}$ that induce $\pi $ is the orbit of $\pi _{0}$
under $G$; that is, $\left\{ \pi _{0}^{x}:x\in G\right\} $ and $\pi
_{0}^{x}=\pi _{0}^{y}$ iff $x=y\mathrm{mod}G_{0}$, where $\pi _{0}^{x}=\pi
_{0}\circ i_{x}$. If $\overline{\pi }$ is the restriction of $\pi $ to $%
G_{0} $, then $\overline{\pi }=\int_{G\diagup G_{0}}\pi
_{0}^{x}~dx=\int_{S}\pi _{0}^{x^{\prime }}~dx^{\prime }$, where $dx$ and $%
dx^{\prime }$ refer to Haar measure in $G\diagup G_{0}\cong S$.
\end{lemma}

Now let $\left( U_{\varepsilon }\right) ^{s}$ be the subspace of $%
U_{\varepsilon }$ complementary to the stabilizer of $U_{\varepsilon
}|_{Z\left( G\right) }$. The projection onto the subspace corresponding to $%
\left( U_{\varepsilon }\right) ^{s}$ is in the center of the commuting
algebra of $U_{\varepsilon }$ (see similar argument in the proof of Lemma %
\ref{ReductionLemma}). Thus if $\pi $ is nontrivial on $Z\left( G\right) $
and occurs in $U_{\varepsilon }$, then it occurs in $\left( U_{\varepsilon
}\right) ^{s}$ just as often. Thus, 
\begin{eqnarray*}
U_{\varepsilon } &=&\sum_{\pi \in \widehat{G}}m\left( \pi ,U_{\varepsilon
}\right) \pi \\
\left( U_{\varepsilon }\right) ^{s} &=&\sum_{\pi \in B}m\left( \pi
,U_{\varepsilon }\right) \pi ,
\end{eqnarray*}%
where $B$ is the subset of $\widehat{G}$ consisting of those $\pi $ that are
nontrivial on $Z\left( G\right) $ and such that $m\left( \pi ,U_{\varepsilon
}\right) >0$. For each $\pi \in B$, choose a $\pi _{0}\in \widehat{G_{0}}$
that induces $\pi $. If $\overline{\left( U_{\varepsilon }\right) ^{s}}$ is
the restriction of $\left( U_{\varepsilon }\right) ^{s}$ to $G_{0}$, 
\begin{eqnarray}
\overline{\left( U_{\varepsilon }\right) ^{s}} &=&\sum_{\pi \in B}m\left(
\pi ,U_{\varepsilon }\right) \overline{\pi }  \notag \\
&=&\sum_{\pi \in B}m\left( \pi ,U_{\varepsilon }\right) \int_{G\diagup
G_{0}}\pi _{0}^{x}~dx.  \label{nonMackey}
\end{eqnarray}

On the other hand, we can decompose $\overline{U_{\varepsilon }}$, the
restriction of $U_{\varepsilon }$ to $G_{0}$, by using Mackey's subgroup
theorem. Indeed, let $U_{0\varepsilon }^{x}$ be the representation of $G_{0}$
induced by the\textbf{\ }$\varepsilon $-representation of $x\Gamma
x^{-1}\cap G_{0}=x\left( \Gamma \cap G_{0}\right) x^{-1}$\ (since $G_{0}$\
is normal). Note that as $x$ is fixed, we can extend the definition of $%
\varepsilon $ to $x\Gamma x^{-1}$. \ It is clear that $U_{0\varepsilon }^{x}$%
\ is the conjugate by $x$\ of $U_{0\varepsilon }$; i.e., $U_{0\varepsilon
}^{x}\left( n\right) =\left( U_{0\varepsilon }\right) ^{x}\left( n\right)
=U_{0\varepsilon }\left( xnx^{-1}\right) $. Then by Mackey's Theorem (\cite[%
Theorem 12.1]{Mackey}), $U_{0\varepsilon }^{x}$\ depends only on the double
coset $\Gamma \cdot x\cdot G_{0}$\ of $x$.\textbf{\ }But $G_{0}$ is normal,
and $\Gamma \cdot x\cdot G_{0}=\Gamma \cdot G_{0}\cdot x$ is a coset of the
subgroup $\Gamma G_{0}$. We know that $\Gamma G_{0}$ is closed ( basic fact
about nilpotent groups: $\Gamma $ is cocompact discrete, $G_{0}$ is normal
in $G$ ), and thus the double cosets fill out the group, allowing us to
apply Mackey's Theorem.

Finally, (also by Mackey) 
\begin{equation*}
\overline{U_{\varepsilon }}=\int_{\Gamma \cdot G_{0}\diagdown
G}U_{0\varepsilon }^{y}~dy.
\end{equation*}%
Now, if $\left( \overline{U_{\varepsilon }}\right) ^{s}$ is the part of $%
\overline{U_{\varepsilon }}$ that is orthogonal to the stabilizer of $%
Z\left( G\right) $, then $\left( \overline{U_{\varepsilon }}\right) ^{s}=%
\overline{\left( U_{\varepsilon }\right) ^{s}}$, since the center is in $%
G_{0}$. Finally, if $\left( U_{0\varepsilon }\right) ^{s}$ is the similar
subrepresentation of $U_{0\varepsilon }$ on which $Z\left( G\right) $ acts
nontrivially, then one immediately deduces from the above that%
\begin{equation*}
\left( \overline{U_{\varepsilon }}\right) ^{s}=\overline{\left(
U_{\varepsilon }\right) ^{s}}=\int_{\Gamma \cdot G_{0}\diagdown G}^{y}\left(
\left( U_{0\varepsilon }\right) ^{s}\right) ^{y}~dy.
\end{equation*}%
We write 
\begin{equation*}
\left( U_{0\varepsilon }\right) ^{s}=\sum_{\lambda _{0}\in A^{\prime
}}m\left( \lambda _{0},U_{0\varepsilon }\right) \lambda _{0},
\end{equation*}%
where $A^{\prime }$ is the set of elements of $\widehat{G_{0}}$ that do not
vanish on $Z\left( G\right) $ and for which $m\left( \lambda
_{0},U_{0\varepsilon }\right) >0$. We are using the fact that $m\left(
\lambda _{0},U_{0\varepsilon }\right) =m\left( \lambda _{0},\left(
U_{0\varepsilon }\right) ^{s}\right) $ for $\lambda _{0}\in A^{\prime }$.

If $\gamma \in \Gamma $, then $\gamma \Gamma \gamma ^{-1}\cap G_{0}=\Gamma
\cap G_{0}$, and from this it follows that $\left( \left( U_{0\varepsilon
}\right) ^{s}\right) ^{\gamma }=$ $\left( U_{0\varepsilon }\right) ^{s}$.
Therefore, we have $m\left( \lambda _{0}^{\gamma },U_{0\varepsilon }\right)
=m\left( \lambda _{0},U_{0\varepsilon }\right) $, and thus $\gamma \cdot
A^{\prime }=A^{\prime }$. Now let $A$ be a subset of $A^{\prime }$ such that 
$A$ meets each orbit of $\Gamma $ on $A^{\prime }$ in exactly one element.
Since $G_{0}$ acts trivially on $\widehat{G_{0}}$ and hence on $A^{\prime }$%
, a $\Gamma G_{0}$-orbit in $A^{\prime }$ is just a $\Gamma $-orbit in $%
A^{\prime }$. Moreover, $G_{0}$ (by Kirillov) is the subgroup of $\Gamma
G_{0}$ leaving any point in $A^{\prime }$ fixed. Therefore, we can write 
\begin{equation*}
\left( U_{0\varepsilon }\right) ^{s}=\sum_{\lambda _{0}\in A}m\left( \lambda
_{0},U_{0\varepsilon }\right) \sum_{s\in \Gamma \cdot G_{0}\diagdown
G_{0}}\lambda _{0}^{s},
\end{equation*}%
and thus 
\begin{equation*}
\left( \overline{U_{\varepsilon }}\right) ^{s}=\sum_{\lambda _{0}\in
A}m\left( \lambda _{0},U_{0\varepsilon }\right) \left[ \int_{\Gamma \cdot
G_{0}\diagdown G}\left( \sum_{s\in \Gamma \cdot G_{0}\diagdown G_{0}}\lambda
_{0}^{s}\right) ^{y}dy\right] .
\end{equation*}%
But since $G\diagup G_{0}$ is equivalent as a Borel space and measure space
to $\Gamma \cdot G_{0}\diagdown G\times \left( G_{0}\diagdown \Gamma \cdot
G_{0}\right) $ by choosing a Borel cross section, the representation in
square brackets is just%
\begin{equation*}
\int_{G\diagup G_{0}}\lambda _{0}^{x}~dx.
\end{equation*}%
Thus,%
\begin{equation}
\left( \overline{U_{\varepsilon }}\right) ^{s}=\sum_{\lambda _{0}\in
A}m\left( \lambda _{0},U_{0\varepsilon }\right) \int_{G\diagup G_{0}}\lambda
_{0}^{x}~dx.  \label{Mackey}
\end{equation}%
Now, since $G_{0}$ is type I and direct integral decompositions are
essentially unique, we may equate coefficients in (\ref{Mackey}) and (\ref%
{nonMackey}). We find then that the family of orbits $\left\{ \pi
_{0}^{G}:\pi _{0}\in B\right\} $ and $\left\{ \lambda _{0}^{G}:\lambda
_{0}\in A\right\} $ are the same. Moreover, the orbits of $\pi _{0}^{G}$ are
all distinct, whereas some of the orbits of $\lambda _{0}^{G}$ may coincide.
Thus, we can equate the multiplicities as follows:%
\begin{equation*}
m\left( \pi ,U_{\varepsilon }\right) =\sum_{\lambda _{0}\in \pi _{0}^{G}\cap
A}m\left( \lambda _{0},U_{0\varepsilon }\right) .
\end{equation*}

(End of generalized Moore Algorithm Proof)
\end{proof}

\begin{corollary}
\label{MooreInductionCorollary}Under the conditions of Moore's algorithm, $%
m\left( \pi ,U_{\varepsilon }\right) >0$ if and only if there is an
irreducible representation $\pi _{1}$ of the rational Kirillov subgroup $%
G_{1}$ such that $m\left( \pi _{1},U_{1\varepsilon }\right) >0$ and $\pi
=Ind_{G_{1}}^{G}\left( \pi _{1}\right) $.
\end{corollary}

\begin{remark}
\textbf{Abelian case:}\label{AbelianCaseSubsection}\newline
Suppose $\Gamma $ is a lattice in $G=\mathbb{R}^{n}$, given by generators $%
\gamma _{1},\gamma _{2},...,\gamma _{n}$. The coadjoint orbit of any $\alpha
\in \mathfrak{g}^{\ast }$ is $\left\{ \alpha \right\} $, and the maximal
abelian subalgebra is $\mathfrak{h=g}=\mathbb{R}^{n}$. By the Kirillov
correspondence this implies that irreducible representations of $G$ are
characters $x\mapsto e^{2\pi i\alpha \left( x\right) }$of $G$ determined by
elements $\alpha \in \mathfrak{g}^{\ast }=\left( \mathbb{R}^{n}\right)
^{\ast }$. Such an $\alpha $ occurs as a representation induced by $%
\varepsilon $ if%
\begin{equation*}
e^{2\pi i\alpha \left( \gamma \right) }=\varepsilon \left( \gamma \right)
\end{equation*}%
for all $\gamma \in \Gamma $.

This condition occurs exactly when $\alpha \left( \gamma \right) \in \mathbb{%
Z}$ whenever $\varepsilon \left( \gamma \right) =1$ and $\alpha \left(
\gamma \right) \in \mathbb{Z}+\frac{1}{2}$ when $\varepsilon \left( \gamma
\right) =-1$; i.e., the pair $\left( \overline{\alpha },H\right) $ is an $%
\varepsilon $-integral point. This means that there exists $k_{j}$, $%
l_{j}\in \mathbb{Z}$ such that 
\begin{equation*}
\alpha =\sum_{j,~\varepsilon \left( \gamma _{j}\right) =1}k_{j}\gamma
_{j}^{\ast }+\sum_{j,~\varepsilon \left( \gamma _{j}\right) =-1}\left( \frac{%
1}{2}+l_{j}\right) \gamma _{j}^{\ast }\text{,}
\end{equation*}%
where $\left\{ \gamma _{1}^{\ast },\gamma _{2}^{\ast },...,\gamma _{n}^{\ast
}\right\} $ is the basis of $\mathfrak{g}^{\ast }$ dual to $\left\{ \gamma
_{1},\gamma _{2},...,\gamma _{n}\right\} $. So $\pi _{\alpha }$ can be
written as 
\begin{equation*}
\pi _{\alpha }\left( t\right) =\underset{j,~\varepsilon \left( \gamma
_{j}\right) =1}{\prod }e^{2\pi ik_{j}t_{j}}\underset{j,~\varepsilon \left(
\gamma _{j}\right) =-1}{\prod }e^{\pi i\left( 2l_{j}+1\right) t_{j}},
\end{equation*}%
with $t=\sum t_{j}\gamma _{j}\in \mathbb{C}^{n}$.
\end{remark}

We now prove Theorem \ref{occurInL2EpsTheorem}, the $\varepsilon $%
-generalized Richardson occurrence condition.

\begin{proof}
\textbf{Forward Direction:} \newline
Suppose $H$ has codimension zero. This implies that $\alpha \left( \left[ 
\mathfrak{g},\mathfrak{g}\right] \right) =0$, by the definition of maximal
subordinate subalgebra. By possibly repeated application of Lemma~\ref%
{ReductionLemma}, we can factor out $\left[ \mathfrak{g},\mathfrak{g}\right] 
$, and the occurrence and multiplicity remain unchanged. This reduces the
problem to the abelian case, which is proved in Remark \ref%
{AbelianCaseSubsection}.

We now proceed inductively on the codimension of $H$: assume that the
theorem is known for codimension $k-1$ or less. Now suppose $\pi \in 
\widehat{G}$ and that $m\left( \pi ,U_{\varepsilon }\right) $ is greater
than zero. Let $\pi $ be induced from $\left( \overline{\alpha },H\right) $,
where the codimension of $H$ is $k$.

\textbf{Cases:}

\begin{enumerate}
\item Suppose that $\pi =1$ on $Z\left( G\right) $. Since the center is
always a rational subalgebra (for nilpotent groups, for any cocompact
lattice), then we pick a one-dimensional rational subgroup $N\subset Z\left(
G\right) $ on which $\pi $ is trivial. Then we can apply Lemma~\ref%
{ReductionLemma}, and we have reduced the codimension of $H$ by one.

\item Suppose that $\pi $ acts nontrivially on $Z\left( G\right) \neq G$ and
that $\dim \left( Z\left( G\right) \right) >1$. We have that $U_{\varepsilon
}\left( z\right) $ is multiplication by $\varepsilon \left( z\right) $ for
all $z\in \Gamma \cap Z\left( G\right) $ by the definition of $%
U_{\varepsilon }$. Write $\pi =\pi _{\lambda }$ for some rational $\lambda
\in \mathfrak{g}^{\ast }$. Since the kernel of $\lambda $ restricted to $%
\mathfrak{z}$ is rational and at least dimension one, we can pick a
one-dimensional rational subgroup $N\subset Z\left( G\right) $ on which $\pi 
$ is trivial. We now apply Lemma~\ref{ReductionLemma} and reduce the
codimension of $H$ by one.

\item Suppose that $\pi $ acts nontrivially on $Z\left( G\right) \neq G$ and
that $\dim \left( Z\left( G\right) \right) =1$. Let $G_{1}$ be the rational
Kirillov subgroup of $G$ corresponding to $\pi $, and note that the
codimension of $G_{1}$ is $1$ and $H\subset G_{1}$, by construction. Let $%
U_{1\varepsilon }$ be the restriction of $U_{\varepsilon }$ to $G_{1}$. By
Corollary~\ref{MooreInductionCorollary}, there is an irreducible
representation $\pi _{1}$ of $G_{1}$ such that $m\left( \pi
_{1},U_{1\varepsilon }\right) >0$ and $\pi _{1}$ induces $\pi $. Let $\pi
_{1}^{\prime }=\mathrm{Ind}_{H}^{G_{1}}\overline{\alpha }$, which then
induces $\pi $, and $\pi _{1}^{\prime }$ is also an irreducible
representation of $G_{1}$ by the Kirillov theory. But $\pi _{1}$ must be
equivalent to $\pi _{1}^{\prime \prime }\left( \cdot \right) =\left( \pi
_{1}^{\prime }\right) ^{x}\left( \cdot \right) :=\pi _{1}\left( x\left(
\cdot \right) x^{-1}\right) $ for some $x\in G$ by the Kirillov
correspondence. Since $m\left( \pi _{1}^{\prime \prime },U_{1\varepsilon
}\right) >0$, there exists $g_{1}\in G_{1}$ such that $f\circ \mathrm{Ad}%
\left( x\right) \circ \mathrm{Ad}\left( g_{1}\right) :\log \left( \Gamma
\cap G_{1}\right) \rightarrow \mathbb{Q}$ (again, see \cite[Cor. 2, p. 154]%
{Moo}). Note that we do not know that $\left( \overline{\alpha },H\right)
\cdot x$ is maximal. Write $\log \left( xg_{1}\right) =aX_{1}+P_{1}$, where $%
P_{1}\in \mathfrak{g}_{1}$ and $X_{1}$ is the first external vector for $%
\mathfrak{h}$, as in the construction of the special maximal subordinate
subalgebra in \cite[Section 3]{Ri}. Note that $Y_{1}\in \log \left( \Gamma
\right) $ from the construction satisfies $\left[ X_{1},Y_{1}\right]
=Z_{1}\in \log \left( \Gamma \right) $, which generates $Z\left( G\right) $.
Since $\mathfrak{g}_{1}\ $is the centralizer $C\left( Y_{1},\mathfrak{g}%
\right) $, we have%
\begin{eqnarray*}
\alpha \circ \mathrm{Ad}\left( x\right) \circ \mathrm{Ad}\left( g_{1}\right)
\left( Y_{1}\right) &=&\alpha \left( Y_{1}+\left[ aX_{1}+P_{1},Y_{1}\right] +%
\frac{1}{2}\left[ aX_{1}+P_{1},\left[ aX_{1}+P_{1},Y_{1}\right] \right]
...\right) \\
&=&\alpha \left( Y_{1}+a\left[ X_{1},Y_{1}\right] +0+0+...\right) \\
&=&\alpha \left( Y_{1}+aZ_{1}\right) ,
\end{eqnarray*}%
by the Campbell-Baker-Hausdorff formula. Since $Y_{1}\in \log \left( \Gamma
\right) $, $\alpha \left( Y_{1}+aZ_{1}\right) =\alpha \left( Y_{1}\right)
+a\alpha \left( Z_{1}\right) \in \mathbb{Q}$, but since $Y_{1},Z_{1}\in \log
\left( \Gamma \right) $ we have $a\in \mathbb{Q}$. Let $g_{0}=\exp \left(
aX_{1}\right) $. Then $\left( \overline{\alpha },H\right) \cdot g_{0}$
induces $\rho _{1}$ on $G_{1}$, which induces $\pi $ on $G$, where $\left( 
\overline{\alpha },H\right) \cdot g_{0}$ is a rational maximal character on $%
G_{1}$ and $m\left( \rho _{1},U_{1\varepsilon }\right) >0$. By construction, 
$\left( \overline{\alpha },H\right) \cdot g_{0}$ is maximal.

By the induction hypothesis, there is an $\varepsilon $-integral point in $%
\left( \overline{\alpha },H\right) \cdot g_{0}\cdot G_{1}$, so that $\left( 
\overline{\alpha },H\right) \cdot G$ has an $\varepsilon $-integral point.%
\newline
\end{enumerate}

\textbf{Converse:}

Suppose $\left( \overline{\alpha },H\right) \cdot G$ has an $\varepsilon $%
-integral point $\left( \overline{\alpha },H\right) \cdot g_{0}$. As above,
we reduce to the case where the dimension of the center is $1$ and $\pi $
restricted to $Z\left( G\right) $ is nontrivial. We know that the Kirillov
subgroup $G_{1}$ is normal in $G$, so our $\varepsilon $-integral point $%
\left( \overline{\alpha },H\right) \cdot g_{0}$ induces $\pi _{1}^{g_{0}}$,
which induces $\pi ^{g_{0}}$, which is equivalent to $\pi $. Also, $\left( 
\overline{\alpha },H\right) $ induces $\pi _{1}$, which induces $\pi $, and $%
\left( \overline{\alpha },H\right) \cdot g_{0}$ is a maximal character in $%
G_{1}$. It follows from the induction hypothesis that $m\left( \pi
_{1}^{g_{0}},U_{1\varepsilon }\right) >0$ which by Moore's induction implies
that $m\left( \pi ,U_{\varepsilon }\right) >0$.
\end{proof}

Assume that $m\left( \pi ,L_{\varepsilon }^{2}\left( \Gamma \diagdown
G\right) \right) >0$, and $\left( \overline{\alpha },H\right) $ induces $\pi 
$, where $\log \left( H\right) $ is a special maximal subordinate subalgebra
to $\alpha $ with respect to $\Gamma $. See \cite[Section 3]{Ri} for the
construction for the special subordinate subalgebra.

\begin{lemma}
If $x=\exp \left( X\right) $, and if $\left( \overline{\alpha },H\right)
\cdot x=\left( \overline{\alpha },H\right) $, then $x\in H$.
\end{lemma}

\begin{proof}
See \cite[Section 5]{Ri}, Lemma 5.1. The proof holds verbatim.
\end{proof}

As a result, we may identify the $G$-orbit of $\left( \overline{\alpha }%
,H\right) $ with $H\diagdown G$.

\begin{lemma}
If $\left( \overline{\alpha },H\right) $ is an $\varepsilon $-integral point
and if $\gamma _{0}\in \Gamma $, then $\left( \overline{\alpha },H\right)
\cdot \gamma _{0}$ is an $\varepsilon $-integral point.
\end{lemma}

\begin{proof}
Consider $\left( \overline{\alpha },H\right) \cdot \gamma _{0}$. Note that $%
\Gamma \cap ^{\gamma _{0}^{-1}}H=^{\gamma _{0}^{-1}}\left( \Gamma \cap
H\right) $ since $^{\gamma _{0}^{-1}}\Gamma =\Gamma $. But if $\gamma
_{0}^{-1}\gamma \gamma _{0}\in \Gamma \cap ^{\gamma _{0}^{-1}}H$, then $%
\overline{\alpha }^{\gamma _{0}}\left( \gamma _{0}^{-1}\gamma \gamma
_{0}\right) =\overline{\alpha }\left( \gamma \right) =\varepsilon \left(
\gamma \right) =\varepsilon \left( \gamma _{0}^{-1}\gamma \gamma _{0}\right) 
$ for every $\gamma \in \Gamma \cap H$. Also, $^{\gamma _{0}^{-1}}\left(
\Gamma \cap H\right) $ is uniform in $^{\gamma _{0}^{-1}}H$.
\end{proof}

Let $\left( H\diagdown G\right) _{\varepsilon }$ be the set of $\varepsilon $%
-integral points in $H\diagdown G$. As a result of the second Lemma, $\Gamma 
$ acts on $\left( H\diagdown G\right) _{\varepsilon }$.

We now prove Theorem \ref{multiplicityAndOccurInL2EpsTheorem}.

\begin{proof}
The proof of \cite[Theorem 5.3]{Ri} goes through, replacing the reference to
Lemma 2.6 with Lemma \ref{ReductionLemma} and the reference to Lemma 2.7
with Theorem \ref{MooreAlgTheorem}, and replacing the phrase
\textquotedblleft integral point\textquotedblright\ with \textquotedblleft $%
\varepsilon $-integral point\textquotedblright .
\end{proof}


\end{document}